\documentclass[11pt,fleqn]{article}
\usepackage[affil-it]{authblk}
\usepackage[latin1]{inputenc}
\usepackage{amsfonts,array}
\usepackage{latexsym,amssymb,vmargin,multicol,graphics,amsmath}
\usepackage[dvips]{graphicx}
\usepackage{float} 

\newtheorem{theorem}{Theorem}[section]
\newtheorem{prop}[theorem]{Proposition}
\newtheorem{lemma}[theorem]{Lemma}
\newtheorem{coro}[theorem]{Corollary}
\newtheorem{defi}[theorem]{Definition}
\newtheorem{rema}[theorem]{Remark}

\newenvironment{proof}{\par\indent\textit{Proof.\/}}{\raisebox{-.2ex}{$\Box$}}

\newcommand{\carre}{\raisebox{-.2ex}{$\Box$}}
\newcommand{\C}{\ensuremath{\mathbb C}}
\newcommand{\R}{\ensuremath{\mathbb R}}
\newcommand{\Q}{\ensuremath{\mathbb Q}}
\newcommand{\Z}{\ensuremath{\mathbb Z}}
\newcommand{\N}{\ensuremath{\mathbb N}}
\newcommand{\p}{\ensuremath{\mathbb P}}
\renewcommand\arraystretch{1.8}

\def\NM{{\bf N}}
\def\ZM{{\bf Z}}
\def\QM{{\bf Q}}
\def\RM{{\bf R}}
\def\CM{{\bf C}}
\def\definition{\noindent{\bf Définition}\ }
\def\definitions{\noindent{\bf Définitions}\ }
\def\point{\ \ $\bullet$\ }

\title{Dynamics of the fifth Painlevé foliation}
\author{E. Paul
\thanks{\noindent Institut of Mathematics of Toulouse, 118 route de Narbonne, 31062 Toulouse Cedex, France\\
emmanuel.paul math.univ-toulouse.fr}
\and J.P. Ramis\thanks{\noindent Institut of Mathematics of Toulouse, 118 route de Narbonne, 31062 Toulouse Cedex, France\\
Institut de France (Académie des Sciences), France\\ramis.jean-pierre@wanadoo.fr}}

\begin{document}

\maketitle

\textbf{Abstract.} \textsl{The leaves of the Painlevé foliations appear as the isomonodromic deformations of a rank 2 linear connection on a moduli space of connections. Therefore they are the fibers of the Riemann-Hilbert correspondence that sends each connection on its monodromy data, and this correspondence induces a conjugation between the dynamics of the foliation and a dynamic on a space of representations of some fundamental groupoid (a character variety). This one can be identified to a family of cubic surfaces through trace coordinates. We describe here the dynamics on the character variety related to the Painlevé V equation. We have here to consider irregular connections, and the representations of wild groupoids. We describe and compare all the dynamics which appear on this wild character variety: the tame dynamics, the confluent dynamics, the canonical symplectic dynamics and the wild dynamics.}

\tableofcontents

\newpage

\section*{Introduction}
\addcontentsline{toc}{section}{Introduction}

The transverse dynamics of an holomorphic foliation is usually defined by its holonomy groupoid. In general we use a specialization of this groupoid along a particular leaf $L$: the holonomy group of $L$. This one is a representation of the fundamental group of the leaf on a germ of transversal, defined by the analytic continuation of the solutions. It only describes the transverse structure of the foliation in a neigbourhood of $L$.

The Painlevé foliations are holomorphic foliations defined by the Painlevé equations. These ones are order two non linear ordinary differential equations. Their writings are available for example in \cite{GLS}.
We only recall here the Painlevé VI and Painlevé V families:\\
\begin{equation}\label{EVI} {\begin{aligned}\mathrm {(P_{VI})} 
\quad {\frac {\mathrm {d} ^{2}w}{\mathrm {d} t^{2}}}&={\frac {1}{2}}\left[{\frac {1}{w}}+{\frac {1}{w-1}}+{\frac {1}{w-t}}\right]\left({\frac {\mathrm {d} w}{\mathrm {d} t}}\right)^{2}-\left[{\frac {1}{w}}+{\frac {1}{t-1}}+{\frac {1}{w-t}}\right]{\frac {\mathrm {d} w}{\mathrm {d} t}}
\\&+{\frac {w(w-1)(w-t)}{2t^{2}(t-1)^{2}}}\left[(\alpha_4-1)^{2}-\alpha_{3}^{2}{\frac {t}{w^{2}}}+\alpha_{1}^{2}{\frac {t-1}{(w-1)^{2}}}+(1-\alpha_{2}^{2}){\frac {t(t-1)}{(w-t)^{2}}}\right],\end{aligned}}
\end{equation}

\begin{equation}\label{EV}\begin{aligned} (P_V)\ \frac{d^2w}{dt^2}&=\left(\frac{1}{2w}+\frac{1}{w-1}\right)\left(\frac{dw}{dt}\right)^2
-\frac{1}{t}\frac{dw}{dt}+
\frac{(w-1)^2}{t^2}\left(\frac{\alpha_1^2}{2}w-\frac{\alpha_3^2}{2w}\right)  \\&+  (1-\alpha_1-2\alpha_2-\alpha_3)\frac{w}{t}
-\frac{1}{2}\frac{w(w+1)}{w-1}.
\end{aligned}
\end{equation}
They define families of one dimensional foliation on $\{(t,w,w') \in \C^3\}$. They share with all the Painlevé differential equations type three properties, which will allow us to give a more global and explicit description of their dynamics. Let us detail these properties for the Painlevé VI foliation:
\bigskip\\
\textbf{1- The Painlevé property} : all the solutions $(w(t),w'(t))$ have a meromorphic continuation over the universal covering the time space $T=\p_1(\C)$ punctured at the fixed singularities 0, 1 and $\infty$. This property generalizes the property of holomorphic extension of the solutions of linear differential equations. The search of non linear differential equations satisfying this property was the initial motivation of Painlevé and then Gambier in order to construct new transcentendal functions. It gives rise to the six well known families of Painlevé equations.

According to this property, the movable singularities, i.e. the singularities of the solutions outside the fixed singularities defined by the equation itself, are only poles. K. Okamoto has introduced a semi-compactification process in order to get a foliation transverse to a fibration over the time space. 
After this process, one can define the dynamic of the Painlevé foliation by its non linear monodromy, which is defined by an action of the fundamental group of the time space $T$, over the Okamoto fiber (the space of initial values). We call it the \textsl{tame dynamics}.
\bigskip

\noindent\textbf{2- The isomonodromic property.} In order to study the non linear monodromy of this foliation, it is useful to introduce a conjugacy between this dynamical system and a simpler one: the Riemann-Hilbert map. Let us recall the different ingredients involved by this property for the Painlevé VI equation.
\medskip\\
\textit{$\bullet$ The  moduli space of connections $\mathcal{M}_{VI}$.} We consider the family of linear connections on the trivial bundle over $\p^1$ defined the family $\mathcal{S}_{VI}$ of rank 2 trace free differential systems
\begin{equation*}\label{syst_fuchs} (A):\ \frac{dY}{dx}=\left(\sum_{i=1}^4\frac{A_i}{x-p_i}\right)\cdot Y,\ p_i\neq p_j \mbox{ for }i\neq j,\ A_i\in sl_2(\C),\ \sum_{i=1}^4A_i=0,\ x\in \p^1,\end{equation*}
up to the global gauge action: $Y\mapsto Y\cdot P$, $P\in SL_2(\C)$, and to a change of the independent variable $x\mapsto\varphi(x)$, $\varphi\in Aut(\p_1)$. The extension of the system to $x=\infty\in \p^1$ is defined by the change of variable $z=x^{-1}$. The family $\mathcal{S}_{VI}$ is a 13 dimensional space (including the variables $p_i$), and the above quotient is a 7 dimensional space. The local parameter space is defined by 
\[\alpha(A)=(det(A_i)=\alpha_i^2/4,\ i=1,\ldots 4),\] 
where $\pm\alpha_i/2$'s are the eigenvalues of each residue matrix $A_i$. The \textsl{time parameter} is the cross ratio $t\in T=\p^1\setminus\{0,1,\infty\}$ of the configurations of the singular points $S=\{p_1,p_2,p_3,p_4\}$. Over a value $\alpha$, the fiber $\mathcal{M}_{VI}(\alpha)$ is a 3 dimension space, and $t(A)$ defines a fibration of $\mathcal{M}_{VI}(\alpha)$ over $T$.
\medskip\\
\textit{$\bullet$ The character variety $\chi_{VI}$ as representations of a fundamental groupoid.} For a system $A$ in $\mathcal{S}_{VI}$  we can define its monodromy representation by using the analytic continuation of a local matrix solution $Y_0$ along any element $\gamma$ of the fundamental group $\pi_1(\p^1\setminus S,x_0)$. We consider here another (equivalent) description of this monodromy, introduced by the authors in \cite{PR}, essential for what will follows.
Instead of considering as usually a fundamental group, we consider a \textit{fundamental groupoid}, which  allows to make use of several base points. A groupoid is a small category whose morphisms are invertible. For details on groupoids see \cite{Brown}. The fundamental groupoid $\pi_1(X)$ of a variety $X$ is the groupoid whose objects are the points of $X$ and the morphisms the paths between two points up to homotopy, with the obvious composition. In what follows, ``paths'' always means ``path up to an homotopy''. We consider the variety $X$ obtained by 4 real blowing up of the singular points $p_i$ in $\p^1$, and we choose a base point $s_i$ on each divisor (i.e. a direction out of $p_i$). Let $S=\{s_1,s_2,s_3,s_4\}$. The groupoid $\pi_1^{VI}(X,S)$ is the restriction of the fundamental groupoid $\pi_1(X)$ to this finite set of objects $S$.
$\pi_1^{VI}(X,S)$ has a presentation given by the following picture:
\begin{figure}[H]
\centering
\includegraphics[scale=0.7]{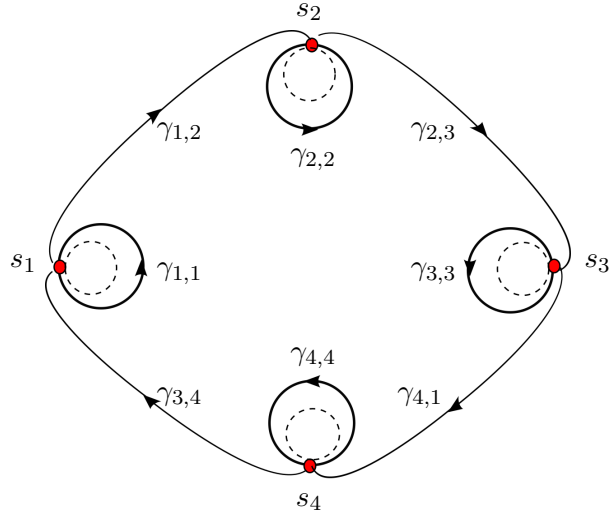}
\put(-210,80){$s_1$}
\put(-103,175){$s_2$}
\put(-103,-10){$s_4$}
\put(5,80){$s_3$}
\put(-155,77){$\gamma_{1,1}$}
\put(-105,120){$\gamma_{2,2}$}
\put(-60,77){$\gamma_{3,3}$}
\put(-105,45){$\gamma_{4,4}$}
\put(-155,130){$\gamma_{1,2}$}
\put(-60,130){$\gamma_{2,3}$}
\put(-65,30){$\gamma_{4,1}$}
\put(-155,30){$\gamma_{3,4}$}
\caption{A presentation of the groupoid $\pi_1^{VI}(X,S)$.}\label{pi1VI}
\end{figure}

\noindent The morphisms are generated by the local loops $\gamma_{i,i}$ based in $s_i$ homotopic to each divisor and the paths $\gamma_{i,i+1}$ from $s_i$ to $s_{i+1}$ satisfying the following relations:
\begin{equation*}
\begin{aligned}
&\mathcal{R}_{ext}:\ \gamma_{1,2}\gamma_{2,3}\gamma_{3,4}\gamma_{4,1}={\star}_1 \mbox{  (the trivial loop based in $s_1$)}\\
&\mathcal{R}_{int}:\ \gamma_{1,1}\gamma_{1,2}\gamma_{2,2}\gamma_{2,3}\gamma_{3,3}\gamma_{3,4}\gamma_{4,4}\gamma_{4,1}={\star}_1.\\
\end{aligned}
\end{equation*}
A linear representation $\rho$ of $\pi_1^{VI}(X,S)$ is a morphism of groupoid from $\pi_1^{VI}(X,S)$ into the category of vector spaces. In the Painlevé context, we only consider linear representations $\rho$ such that $\rho(s)=V_s$ is a 2-dimensional vector space and $\rho(\gamma_{s,t})$ belongs to the set of linear morphisms which preserve the area.\\
 A choice of a basis in each $V_s$ defines a map  $\rho(\gamma_{s,t})\simeq M_{\gamma_{s,t}}$ in $SL_2(\C)$. We obtain a representation of the groupoid $\pi_1^{VI}(X,S)$ in the group $SL_2(\C)$. If we change the basis in each $V_s$, we obtain another representation $\rho'$ of $\pi_1^{VI}(X,S)$ in the group $SL_2(\C)$ which is \emph{equivalent}, i.e. for any $s$, there exist matrices $P_s$ such that
\[M'_{\gamma_{s,t}}=P_s^{-1}\cdot M_{\gamma_{s,t}}\cdot P_t.\]
Therefore a rank two trace free linear representation of $\pi_1^{VI}(X,S)$ is also a class of equivalent representations of $\pi_1^{VI}(X,S)$ in $SL_2(\C)$. 
\begin{defi} The character variety $\chi_{VI}$ is the categorical quotient \footnote{This means that we consider the affine variety defined by the ring of invariants functions over the set of linear representations: see \cite{Brion} or \cite{GIT}.} of the set of rank 2 trace free linear representations up to the above equivalence.
\end{defi}
The local data of an element $\rho$ in $\chi_{VI}$ is the restriction of $\rho$ to the groupoid $\pi_1^{VI,loc}(X,S)$ obtained by forgetting the generators $\gamma_{i,i+1}$. It is characterized by the class of each $\rho(\gamma_{i,i})$, or by $a=(a_i=tr(\rho(\gamma_{i,i}),\ i=1,...,4)$. We denote by $\chi_{VI}(a)$ the corresponding fiber in $\chi_{VI}$.
\medskip\\
\textit{$\bullet$ The character variety as a family of cubic surfaces $\mathcal{C}_{VI}$.}
Given a presentation of the groupoid $\pi_1^{VI}(X,S)$ by figure (\ref{pi1VI}), we define the trace coordinates $(a,x)$  by
\[a_i(\rho)=tr(\rho(\gamma_{i,i})),i=1,\cdots 4,\ x_k(\rho)=tr(\rho(\gamma_{i,i}\gamma_{i,j}\gamma_{j,j})),\ \{i,j,k\}=\{1,2,3\}.\]
The trace coordinates do not change in a class of equivalent representations in $SL_2(\C)$. Therefore it can be convenient to make use of \emph{normalized} representations:
A representation $\rho$ of the groupoid $\pi_1^{VI}(X,S)$ in the group $SL_2(\C)$ is normalized if for any indexes $i$, $j$, $i\neq j$ we have: $\rho(\gamma_{i,j})=I.$
One can always choose the representation of the objects in order to get a normalized representation by choosing arbitrarily the representation of a first one and by representing the further consecutive ones such that $\rho(\gamma_{i,i+1})=I$.
A normalized representation is given by 4 matrices $M_i=\rho(\gamma_{i,i})$ unique up to a common conjugacy which satisfy, according to the relation $\mathcal{R}_{int}$, $M_1M_2M_3M_4=I$. We recover here the usual monodromy data.
The trace coordinates of a normalized representation $\rho$ are defined by
\[a_i=tr(M_i),i=1,\cdots 4,\ x_1=tr(M_2M_3),\ x_2=tr(M_3M_1),\ x_3=tr(M_1M_2).\]
\begin{prop} The trace map $Tr_{VI}$ defined by $(a,x)$ sends each $\chi_{VI}(a)$ on the cubic surface $\mathcal{C}_{VI}(\theta)$ in $\C^3$ defined by the equation $F_{VI}(x,\theta)=0$ with
\[F_{VI}(x,\theta)=x_1x_2x_3+x_1^2+x_2^2+x_3^2-\theta_1x_1-\theta_2x_2-\theta_3x_3+\theta_4,\]
and $\theta_i=a_ia_4+a_ja_k,\ \{i,j,k\}=\{1,2,3\}$, $\theta_4=a_1a_2a_3a_4+\sum_ia_i^2-4$.
\end{prop}
From \cite{Iwa03}, this trace map is an homeomorphism only on an open set $\chi^0_{VI}(a)$ in the character variety.  For a generic value of the local data $a$, we have $\chi^0_{VI}(a)=\chi_{VI}(a)$. This condition implies that $M_i\neq \pm I$ for $i=1,2,3$.
\medskip\\
\textsl{$\bullet$ The Riemann-Hilbert correspondence $RH_{VI}: \ \mathcal{M}_{VI}\rightarrow \chi_{VI}$}. Let $A$ in $\mathcal{M}_{VI}(\alpha)$. We choose a representative of each base point $s_i$ by a fundamental system of solutions $Y_s$ of $A$. The representation of a path $\gamma$ from $s$ to $t$ is obtained by comparing the analytic continuation $\widetilde{Y_s}^\gamma$ of $Y_s$ along $\gamma$ with $Y_t$:
\[\rho_A(\gamma)=M_\gamma\Leftrightarrow Y_t=\widetilde{Y_s}^\gamma\cdot M_\gamma.\]
A change of representation of the objects, or a gauge action on $A$ gives an equivalent representation. 
The Riemann-Hilbert map $RH_{VI}$ is defined by $RH_{VI}(A)=[\rho_A]$. 
The map induced by $RH_{VI}$ on the local parameters is defined by 
\[RH_{VI}^{loc}(\alpha_i)=2\cos(\pi\alpha_i)=a_i.\]

We do not discuss here about the surjectivity of $RH_{VI}$ (the so called Riemann-Hilbert problem) or its properness. For a survey on this wide subject, see \cite{InIwSai2} for the Painlevé VI case, and \cite{VdPSai} for the other cases.
\medskip\\
\textsl{$\bullet$ Isomonodromic families and Painlevé VI equation.} 
An isomonodromic family on $\mathcal{M}_{VI}$ is a fiber of the map $RH_{VI}$, parametrized by a variable $t$, i.e. a family of connections defined by  
$\frac{dY}{dx}=A(x,t)\cdot Y$
such that the monodromy representation is locally constant in $\chi_{VI}$ along this family. The relationship with the Painlevé VI equation was discovered by R. Fuchs in 1907.
\begin{theorem}\label{isoVI}
There exists coordinates $w,w',t$ over a Zariski open set in $\mathcal{M}_V(\alpha)$ such that the isomonodromic families are the solutions the Painlevé VI equation (\ref{EVI}). 
\end{theorem}
\textit{A sketch of proof \cite{IKSY}.} Given a fundamental system of solutions $Y(x,t)$, by isomonodromy, $\frac{d}{dt}Y(x,t)$ and 
$Y(x,t)$ have the same behaviour along any  continuation. Therefore the quotient 
\[B(x,t)=\frac{d}{dt}Y(x,t)\cdot Y(x,t)^{-1}\]
is univalued outside the fixed singular points. Since the singular points are regular singular points (see section 2 for the definition) $B$ extends to the singular set in a meromorphic way. Therefore $Y(x,t)$ satisfies
two rational linear systems $\frac{dY}{dx}=A(x,t)\cdot Y$ and $\frac{dY}{dt}=B(x,t)\cdot Y$. The compatibility condition requires that the 
two operators $d/dx-A$ and $d/dt-B$ must commute :
\begin{equation}\label{Schelsinger}\frac{dA}{dt}-\frac{dB}{dz}+[B,A]=0.\end{equation}
The pair $(A,B)$ is called an \textsl{isomonodromic Lax pair}, and the above equation the \textsl{Schelsinger equation}. If $A$ is irreducible, $B$ is unique. Now, one can find coordinates such that the Schelsinger equation (\ref{Schelsinger}) is equivalent to the Painlevé VI equation: see \cite{IKSY}. \carre
\bigskip\\
\textbf{3- The Hamiltonian property} :
Any Painlevé differential equation is equivalent to a (non autonomous) hamiltonian system. This fact was first remarked by J. Malmquist \cite{Malm}, and extended by K. Okamoto in \cite{Ok80}. For example the family of Painlevé VI equations
is equivalent to 
\begin{equation}\label{HVI}
\begin{aligned}
&\dot{p}=-\frac{\partial H_{VI}}{\partial q},\
\dot{q}=\frac{\partial H_{VI}}{\partial p},
\end{aligned}
\end{equation}
with
\[H_{VI}(\alpha)=\frac{q(q-1)(q-t)}{t(t-1)}\left(p^2-(\frac{\alpha_3}{q}+\frac{\alpha_1}{q-1}+\frac{\alpha_2-1}{q-t})p+
\frac{\beta}{4q(q-1)}\right).\]
with $\beta=(\alpha_1+\alpha_2+\alpha_3+\alpha_4-2)(\alpha_1+\alpha_2+\alpha_3-\alpha_4)$, and where the numbers $\pm\alpha_i/2$ are the eigenvalues of the residues matrices $A_i$\footnote{The numbering is chosen here such that the two first indices correspond to the singularities 1 and $t$ which are confluent in the usual confluent process of the Painlevé equations.}.

We denote by $P_{VI}(\alpha)$ the corresponding family of Painlevé vector fields, and by 
$\mathcal{F}_{VI}(\alpha)$ the family of dimension one holomorphic foliations on $\C^3_{(p,q,t)}\subset\C^2\times \p_1$ (where $\p^1$ is the complex one dimensional projective space). The fibers $t=0,1,\infty$ are singular and the foliation has 4 singular points on each of this fibers.

\begin{rema}
Painlev\'e sixth equation was not discovered by Painlev\'e and Gambier, using the Painlev\'e property, but by Richard Fuchs, using an isomonodromic deformation. R. Fuchs considered the monodromy preserving deformation of a second order Fuchsian differential equation with four regular singular points and an apparent singularity\footnote{Respectively $x=0,1,t,\infty$ and $x=q$.}~: $Dy=y''+a_1(x,t) y'+ a_2(x,t)=0$, where $a_1$ and $a_2$ are rational functions on $x$ depending holomorphically of $t$. We write the Riemann scheme of $Dy=0$~:
\[
\begin{Bmatrix}
x=0 & x=1 & x=t & x=q & x=\infty \\
0 & 0 & 0 & 0 & \kappa_0\\
\kappa_1 & \kappa_2 & \kappa_3 & 2 & \kappa_0+\kappa_4
\end{Bmatrix}.
\]
In this scheme the exponents $\kappa_i$, $i=0,1,2,3,4$ are complex parameters subject to the Fuchs relation~:
\[\kappa_1 + \kappa_2 + \kappa_3 + \kappa_4 + 2 \kappa_0=-1.\]
The  $\kappa_i$, $i=1,2,3,4$ are related to the eigenvalues $\pm\alpha_i/2$ of the Fuchsian systems $(A)$ by 
\[\kappa_1=\alpha_1$, $\kappa_2=\alpha_2$, $\kappa_3=\alpha_3$ and $\kappa_4=\alpha_4-1.\]
For a proof, see \cite{InIwSai2}.
It is easy to derive the Hamiltonian formulation (cf. \cite{NouYa}). The variable $q$ is the position of the apparent singularity, the variable $p$ is defined by $p=\mathrm{Res}_{x=q} a_2(x,t)dx$ and the hamiltonian by (\ref{HVI}).
\end{rema}

The family of Painlevé V equations (\ref{EV}) is equivalent to
\begin{equation}\label{HV}
\begin{aligned}
&\dot{p}=-\frac{\partial H_{V}}{\partial q},\ 
\dot{q}=\frac{\partial H_{V}}{\partial p},
\end{aligned}
\end{equation}
\[\mbox{with }H_{V}(\alpha)=t^{-1}\left(p(p+t)q(q-1)-\alpha_1p(q-1)+\alpha_2qt-\alpha_3pq\right).\]
We have used here the notations of \cite{TOS2005}. 
The values $\\alpha_i/2$ are not here eigenvalues of the residues matrices of the connection but are directly related to them by an affine invertible map
whose expression in available in the appendix C of \cite{JiMi1980}.

We denote by $P_V(\alpha)$ the corresponding families of Painlevé vector fields, and by 
$\mathcal{F}_{V}(\alpha)$ the family of dimension one holomorphic foliations on $\C^3_{(p,q,t)}\subset\C^2\times \p^1$. There is here only two singular fibers over $t=0$ and $t=\infty$, with 2 singular points over 0, and 5 singular points over $\infty$, 3 of them are of saddle-node type i.e. with a vanishing eigenvalue.

This construction provides a symplectic structure  $\Omega_{VI}(\alpha)$ on $\mathcal{M}_{VI}(\alpha)$ given by $dq\wedge dp$.
This Poisson structure on $\mathcal{M}_{VI}$ also comes from a general construction of Atiyah and Bott \cite{AtBo}: the symplectic structure on the infinite dimensional space of all the connections induces a Poisson structure by a symplectic reduction under the gauge action. There also exists several direct finite dimensional constructions, by using either ciliated fat graphs \cite{FoRo}, \cite{Aud}, or quasi-hamiltonian geometry and quivers varieties \cite{Boa07}, or symplectic reduction of multi-Poisson structures \cite{Chiba17}.

On the right-hand side, the space of linear representations $\chi_{VI}$ has also a Poisson structure which induces a symplectic form $\omega_{\chi_{VI}}(a)$ on each fiber $\chi_{VI}(a)$. This fact was first proved by Goldman in \cite{Gold}, and makes use of the Poincaré-Lefschetz duality on the tangent space. It can be extended to irregular cases by using the concept of decorated character variety : \cite{ChMR}, or quasi-hamiltonian geometry \cite{Boa14}.

There also exists a Poisson structure on the family of cubic surfaces $\mathcal{C}_{VI}(\theta)$ given by:
\[\omega_{VI}(\theta)=\frac{dx_1\wedge dx_2}{F_{VI,x_3}}=\frac{dx_2\wedge dx_3}{F_{VI,x_1}}=\frac{dx_3\wedge dx_1}{F_{VI,x_2}},\]
where $F_{VI,x_i}=\partial F_{VI}/\partial x_i=x_jx_k+2x_i-\theta_i=\partial F_{VI}/\partial x_i$ for i=1,2,3. This structure obtained by the Poincaré residue of the volume form coincide with the Goldman structure:
see \cite{Iwa03}.

Finally K.Iwasaki has proved that the Riemann-Hilbert map  $RH_{VI}$ is a Poisson morphism from $(\mathcal{M}_{VI},\Omega_{VI})$ to $(\mathcal{C}_{VI},2i\pi\omega_{VI})$: \cite{Iwa92}. One can also find a proof of this fact in \cite{Kl1} obtained from the Jimbo's asymptotic formula \cite{Jim82}.
This result has been extended for the Riemann-Hilbert map in a local irregular case by P. Boalch: see \cite{Boa01}.
\bigskip

One can summarize these three properties for the Painlevé foliation by the following statement: the dynamics of the Painlevé VI foliation $\mathcal{F}_{VI}(\kappa)$ is  conjugated through the Riemann-Hilbert map to the dynamics on $\mathcal{C}_{VI}(\theta)$ induces by the automorphisms of $\pi_1^{VI}(X,S)$. 

Since the inner automorphisms $\gamma_{i,j}\mapsto \alpha_{i,i}\gamma_{i,j}\alpha_{j,j}^{-1}$ acts trivially on the character variety, this action factorizes to the quotient 
\[Out(\pi_1^{VI}(X,S))=Aut(\pi_1^{VI}(X,S))/Inn(\pi_1^{VI}(X,S))\] 
which is defined by pure braids over $S$ in $X$. This braid group is generated by three elementary braids
$b_{1,2}$, $b_{2,3}$ and $b_{3,1}$ such that $b_{1,2}\cdot b_{2,3}\cdot b_{3,1}=1.$ The action of each generator on $\chi_{VI}$ has been first computed by Dubrovin and Mazzocco \cite{DubMa}. See also Iwasaki \cite{Iwa03}, or Cantat and Loray \cite{CanLo} or the authors in \cite{PR} for a simple computation using groupoids instead of groups. We have
\begin{prop}\label{dynamicPVI}
The action of $b_{i,j}$ is given by the automorphism:
\begin{equation*}
  h_{i,j}:\ \left\{
    \begin{aligned}
    &x_i\rightarrow -x_i +x_jx_k + x_ix_k^2 - \theta_j x_k + \theta_i  \\
    &x_j\rightarrow -x_j -x_ix_k + \theta_j \\
    &x_k\rightarrow x_k.
    \end{aligned}
    \right.
\end{equation*}
\end{prop}
This dynamics has the remarkable ``tame property'': any element and its inverse is given by polynomials.

\bigskip

\textbf{Contents of this article. } 
Our aim is to present the main tools that will allow us to extend this description for all the Painlevé equations, by considering here the first case of $P_V$. 
In this case, which can be obtained from $P_{VI}$, by a confluent process, new difficulties arise:

-- The monodromy around the two singular fibers $0$ and $\infty$ do not describe the whole transverse structure of the foliation. This monodromy will only give a ``tame'' dynamics generated by just one element. This is a consequence of the ``irregular'' type of the singularities of the foliation over $\infty$ which are now saddle-nodes, with non linear Stokes phenomenon.

-- On the other side, the class of connections on which $P_V$ naturally arises is now a class of irregular connections. The representation of such a connection must take into account beyond the usual monodromy, new operators related to such singular points: the linear Stokes operators and exponential tori. 

We solve these problems by constructing a \textit{wild} fundamental groupoid $\pi_1^V(X,S)$ with its linear representations (modulo equivalence): the wild character variety $\chi_V$.
The action of $Out(\pi_1^V(X,S))$ allows us to recover the tame dynamics.

In order to complete this dynamics, one can construct a family of (non invertible) confluent morphisms from $\pi_1^V(X,S)$ to $\pi_1^{VI}(X,S)$. The key point is that these morphisms induce families of \textsl{birational maps} between the corresponding character varieties. This will allow us to define and compute a confluent dynamics $Conf(P_V)$ on $\chi_V$. This one was previously found by Martin Klimes in \cite{Kl1}. Our technique based here on wild fundamental groupoids allows us to expect a generalization for all the other Painlevé equations. Notice hat the dynamics that we obtain is no longer polynomial but only a rational one.

We will also discuss about a \textsl{canonical dynamic} which exists on the cubic surfaces $\mathcal{C}_V$, without any reference to a Painlevé equation. The central idea is that here the cubic surface is symplectically birationally equivalent to $\left(\C^2,\frac{du}{u}\wedge \frac{dv}{v}\right)$ by a sequence of log-canonical coordinates which satisfies some exchange relations, which appear in cluster algebra. The pull-back of the symplectic Cremona group $Symp\left(\C^2,\frac{du}{u}\wedge \frac{dv}{v}\right)$ delivers a very rich structure denoted by $Dyn(\mathcal{C}_V)$. We will compare $Conf(P_V)$ and $Dyn(\mathcal{C}_V)$.

The wild dynamics of the Painlevé V foliation is defined by the non linear monodromy, non linear Stokes operators and non linear tori in a neighborhood of a saddle-node singular point. Following M. Klimes \cite{Kl1}, we will recall the definition of this dynamics and its image through the Riemann-Hilbert map. We also recover here the canonical symplectic dynamics, thus confirming in this case the rationality conjecture of the second author. 

The last part is dedicated to open problems.

\bigskip

\textbf{Acknowledgments.} We would like to thank Martin Klimes for the discussions we had together on this subject. His work \cite{Kl1} is a primary source of inspiration for this article.

\pagebreak

\section{The moduli space  $\mathcal{M}_V$}
According to \cite{VdPSai}, the moduli space of connections on which lives the Painlevé V foliation is defined by:
\[\mathcal{S}_V=\left\{\frac{dY}{dx}=A(x)\cdot Y,\ A(x)=\frac{A_0}{x-s_0}+\frac{A_1}{x-s_1}+A_\infty,\ A_0,\ A_1,\ A_\infty\in sl_2(\C)\right\}\]
such that $A_\infty$ is a non trivial semi-simple element. The singularities, i.e. the points $s$ such that $A$ is not holomorphic around $s$ are $s_0$, $s_1$ and $\infty$.
We suppose that the eigenvalues $\pm t$ of $A_\infty$ are distinct ($t\neq 0$). 
One can extend the basis to $\p_1(\C)$ by setting $z=x^{-1}$:
\[\frac{dY}{dz}=-z^{-2}A(z^{-1})\cdot Y.\]
As explained in the introduction, we consider this family up to a gauge action $Y=PZ$, which acts on the system by 
$A\rightarrow A^P=P^{-1}AP-P^{-1}\frac{dP}{dx}$, and up to a change of independent variable $\varphi(x)$, $\varphi$ in the Möbius group $Aut(\p_1)$.

\subsection{The local classification}
Suppose that $x=0$ is a singular point of a germ of connection defined by
\[x^{p+1}\frac{dY}{dx}=A(x)\cdot Y,\ A\in sl_2(\C\{x\}),\ A_0=A(0)\neq 0,\ p\geq 0.\]
The \textit{local formal classification} is the classification of such a system up to a gauge transformation $Y=PZ$, with $P$ in $SL_2\left(\C((x))\right)$.
The integer $p$ is called the Poincaré rank of the system. It is not a local gauge invariant. If $p=0$ the singular point is called a \textit{fuchsian} singularity. Given an element of $\mathcal{S}_V$, the points $s_0$ and $s_1$ are fuchsian singularities. The local formal classification around such a point is given by the following general result (cf \cite{Wasow}):
\begin{prop} We consider the fuchsian system $xY'=A(x)\cdot Y,\ A_0=A(0)\neq 0$. 
\begin{enumerate}
\item Suppose that the eigenvalues of $A_0$ do not differ from a positive integer (the non resonant case). There exists a local meromorphic gauge transformation $P$ in $SL_2\left(\C((x))\right)$ which conjugates the fuchsian singular system to $xZ'=A_0\cdot Z$
\item In this case, the system (1) admits a local fundamental solution $Y=P(x)x^{A_0}$, $P$ in $SL_2\left(\C((x))\right)$. Furthermore if $A$ is a convergent data, then $P$ is also a convergent matrix.
\end{enumerate}
\end{prop}
The general case reduces to the non resonant case by using shearing transformations, which can eventually modify the Poincaré rank. We also obtain a local fundamental system $Y=P(x)x^L$ for some constant matrix $L$. 
In any cases, the sectoral local solutions admits a moderate growth : fuchsian singularities are \textsl{regular singular points}.
The classification of non fuchsian singularities is given by the Hukuhara Turritin Theorem \cite{Balser} :
\begin{theorem}\label{fnf} We consider a $sl_2$-system $z^{p+1}Y'=A(z).Y,\ A_0=A(0)\neq 0,\ p\geq 1.$ 
There exists a formal meromorphic gauge transformation $Y=PZ$ and a ramification $z=y^k$ such that the differential system has the canonical form 
\[yZ'=(L+ D_1y^{-1}+\cdots +D_my^{-m})\cdot Z,\]
where the $D_i$'s are diagonal matrices, and $L$ commutes with the $D_i$'s. 
\end{theorem}
If the irregular part $D_1y^{-1}+\cdots +D_m y^{-m}$ vanishes, we still get a regular singular point. Otherwise, the singular point is \textit{irregular}, and the degree $r=m/k$ is called the Katz invariant of the irregular singular point. 

\begin{rema} If the eigenvalues of $A_0$ are distincts, we do not need to use a ramification $z=y^k$. Otherwise, for example if $A_0$ is a nilpotent element of $sl_2(\C)$, we first have to use shearing transformations in order to modify the leading term $A_0$, which may change the Poincaré rank and requires a ramification. This may occur for some Painlevé families.
\end{rema}

Given an element $A$ of $\mathcal{S}_V$, the singular point $x=\infty$ $(z=0)$ is an irregular singular point with Katz rank 1. Since the eigenvalues $\pm t$ of $A_\infty$ are distinct, we do not need here a ramification of $z$. The explicit computation of the residue matrix $L_\infty$ gives:
\[L_\infty=\left(\begin{array}{cc}a_0+a_1 & 0 \\ 0 & -a_0-a_1\end{array}\right).\]
Therefore the system $(A)$ has a formal fundamental matrix solution around $z=0$ given by 
\[\widehat{Y}(z)=P(z)z^{L_\infty}\exp Q(z),\ Q(z)=diag(\alpha z^{-1}, - \alpha z^{-1}).\]

\subsection{The global gauge moduli space}

We want to compute the categorical quotient $\mathcal{S}_V//Sl_2(\C)$, i.e the affine variety defined by the ring of invariant functions under the action of the constant gauge action of $Sl_2(\C)$. 
One can first use a constant gauge transformation in order to diagonalize $A_\infty$: 
\[A_\infty=\left(\begin{array}{cc}t/2 & 0 \\ 0 & -t/2\end{array}\right),\ A_i=\left(\begin{array}{cc}a_i & b_i \\ c_i & -a_i\end{array}\right),\ i=0,1.\]
Let $\mathcal{S}'_V$ be the subset of $\mathcal{S}_V$ defined be the above writing. For a fixed singular set $s_0$, $s_1$, $\infty$, it is a 7-1=6 dimensional variety. The remaining gauge action which normalizes the diagonal matrix $A_\infty$ is the group $N=<T,P>$ generated by \[T=\{T_m=\left(\begin{array}{cc}m & 0 \\ 0 & m^{-1}\end{array}\right),\ m\in\C^*\}\mbox{ and }
P=\left(\begin{array}{cc}0 & 1 \\ -1 & 0\end{array}\right).\] 
This group acts on $sl_2(\C)$ by
\[T_m\cdot \left(\begin{array}{cc}a & b \\ c & -a\end{array}\right)=\left(\begin{array}{cc}a & bm^2 \\ cm^{-2} & -a\end{array}\right),\
P\cdot\left(\begin{array}{cc}a & b \\ c & -a\end{array}\right)\cdot P^{-1}=\left(\begin{array}{cc}-a & -c \\ -b & a\end{array}\right)=-A^T.\]
After this first reduction,  $\mathcal{S}_V//SL_2(\C)=\mathcal{S}'_V//N$. The following functions are invariant by the action of $N$:
\begin{equation}
\begin{aligned}
&\alpha_0=a_0^2+b_0c_0\\
&\alpha_1=a_1^2+b_1c_1\\
&\alpha_\infty=(a_0+a_1)^2\\
&\tau=a_0t\\
&\beta_0=b_0c_1+b_1c_0\\
&\beta_1=t(b_0c_1-b_1c_0)
\end{aligned}
\end{equation}
Notice that the three coordinates $\alpha=(\alpha_i),\ i=0,1,\infty$ are local invariants around each singular point. 
As mentioned in the introduction the $\alpha_i$ parameters used in the expressions of the Painlevé V equation and in its hamiltonien $H_V$
are not exactly equal to the present parameters but only related to the eigenvalues of the $A_i$ by an affine invertible map : see \cite{JiMi1980}.

\begin{prop} $\mathcal{S}_V//SL_2(\C)=Spec(\alpha_0,\alpha_1,\alpha_\infty,\tau,\beta_0,\beta_1)$.\end{prop}
\begin{proof} It suffices to prove that the map $(\alpha_0,\alpha_1,\alpha_\infty,\tau,\beta_0,\beta_1):\ \mathcal{S'}_V//SL_2(\C)\rightarrow \C^6$ is invertible over the Zariski open set 
$a_0\neq 0,\ b_0c_0\neq 0$. Given a value $(\alpha_0,\alpha_1,\alpha_\infty,\tau,\beta_0,\beta_1)$, we first choose arbitrarily $t\neq 0$. Now $a_0$ is uniquely determined by $ta_0=\tau$, and $b_0c_0$ is uniquely determined by $\alpha_0-a_0^2$. The values $a_0^2$ and $b_0c_0$ determine a unique class $[A_0]$ modulo $N$. It suffices to prove that the choice of $A_0$ in this class determines in a unique way the triple $(A_0,A_1,A_\infty)$. $A_\infty$ is uniquely defined by $a_0t=\tau$. The two last equations define a linear system in $(b_1,c_1)$ which admits a unique solution for $b_0c_0\neq 0$.  The relation
\[2a_0a_1=\alpha_\infty^2-\alpha_0^2-\alpha_1^2+b_0c_0+b_1c_1\]
determines a unique value of $a_1$ for $a_0\neq 0$. Another choice of $A_0$ in $[A_0]$ will give
an equivalent triple  $(A_0,A_1,A_\infty)$ modulo $N$.
\end{proof}
\bigskip

This quotient is endowed with a Poisson structure whose Casimir functions are $\alpha_0$, $\alpha_1$, $\alpha_\infty$ and $\tau$. Therefore, each fiber has a canonical symplectic structure. See the introduction for references.

\subsection{The moduli space of connections $\mathcal{M}_V$}

We now consider the action of a global change of independent variable. We have previously fixed one singular point at $\infty$. The positions $s_0$ and $s_1$ are now free parameters.
One can use a translation $x\rightarrow x+\lambda$ in order to get $s_0=0$. This action do not modify $A_0$ $A_1$ and $A_\infty$. 
Now we want to normalize the singular point $s_1$ to the value 1, by using the last change of independent variable that we can use: $x\mapsto \mu x$. Since $s_1\neq s_0=0$, we can introduce the parameter $s_1^{-1}$. The linear system 
\[dY=(A_0+\frac{A_1}{1-s_1x^{-1}}+xA_\infty)\frac{dx}{x})\cdot Y\]
is changed in
\[dY=(A_0+\frac{A_1}{1-s_1\mu^{-1}x^{-1}}+x\mu A_\infty)\frac{dx}{x})\cdot Y.\]
This action do not modify $A_0$ and $A_1$ but modify $A_\infty$ in $\mu A_\infty$. Therefore it keeps invariant the local variables $\alpha_i$ and modify the others variables by:
\[(\tau,\beta_0,\beta_1,s_1^{-1})\rightarrow (\mu\beta_0,\beta_1,\mu\beta_2,\mu s_1^{-1}).\]
The quotient of this action is the weighted projective space $\p^3_{(1,0,1,1)}$. The usual choice of $\mu$ such that $s_1=1$, corresponds to a choice of chart in $\p^3_{(1,0,1,1)}$. Therefore
\begin{prop} The moduli space of connections $\mathcal{M}_V(\alpha) $ induced by $\mathcal{S}_V$ for fixed local invariants is isomorphic to $\p^3_{(1,0,1,1)}$. \end{prop}
\begin{rema}
\begin{enumerate}
	\item This space is identical to that obtained by H. Chiba in \cite{Chiba2} when he seeks a natural compactification on which lives the vector field $P_V$ by using its Newton polyhedra.
	\item Since one weight vanishes, the quotient space is isomorphic to $\p_2(\C)\times \C$, and it is not a compact space. This fact is shared with the quotient spaces corresponding to the equations $P_{VI}$ and $P_{III,D_6}$, while we obtain a compact weighted projective space for the other families $(I)$, $(II)$, $(IV)$, $(III,D_7)$, $(III,D_8)$: \cite{Chiba1}. To deal with this problem, we will use the trick introduced by H. Chiba, who introduces two copies of the previous space, glued by a convenient B\"acklund transformation.
\end{enumerate}
\end{rema}

\section{The wild fundamental groupoid $\pi_1^V(X,S)$}

The fundamental group of $\pi^1(\p^1\setminus\{s_0,s_1,s_\infty\},b_0)$ acts on a local fundamental system $Y_0$ around $b_0$ by analytic continuation of $Y_0$ along a path in $\p^1\setminus\{s_0,s_1,s_\infty\}$. We first enlarge this usual monodromy representation by considering new operators around the irregular point $s_\infty$: the formal monodromy, the Stokes operators and the exponential torus. We first describe theses operators in the present context, by using a point of view very similar to that introduced by Stokes himself in \cite{Sto}, starting from a formal solution and using the summability theory.

\subsection{Stokes operators, formal monodromy and exponential tori}\label{irregular sing}

\begin{defi} \label{gevrey}
\begin{enumerate}
	\item Let $k>0$. A formal power series $\widehat{f}=\sum_{i\geq 0} a_ix^i$ is $1/k$-Gevrey if there exists $M>0$ and $A>0$ such that for all $i$ 
$|a_i|\leq MA^i (i!)^{1/k}$.
  \item Let $V$ be a sector at the origin and $f$ an holomorphic function on $V$. $f$ is $1/k$-Gevrey asymptotic to $\widehat{f}$ if for every strict subsector $W$ in $V$, there exists $M_W>0$ and $A_W>0$ such that for all $n\geq 1$
	\[\mid (f(x)-\sum_{i=0}^{n-1} a_ix^i\mid \leq M_WA_W^n (n!)^{1/k}|x|^n.\]
\end{enumerate}
\end{defi}

We denote by $\C[[x]]_{1/k}$ the algebra of series which are $1/k$-Gevrey, $\mathcal{A}_{1/k}(V)$, the algebra of holomorphic functions on $V$ which are $1/k$-Gevrey asymptotic to some formal series.
We mention here the following facts (for the proofs see \cite{Lo3}):
\begin{itemize}
		\item[-] If $f$ is $1/k$-Gevrey asymptotic to $\widehat{f}$, then $\widehat{f}$ is a $1/k$-Gevrey series. Therefore the Taylor map $T:\ \mathcal{A}_{1/k}(V)\rightarrow \C[[x]]_{1/k}$ is well defined.
		\item[-] For $V$ ``narrow'' i.e. with an opening $\leq \pi/k$, the Taylor map $T$ is surjective: this is a Gevrey extension of the Borel-Ritt theorem.
		\item[-] For $V$ ``large'' i.e. with an opening $> \pi/k$, the Taylor map $T$ is injective: this is a Gevrey extension of the Watson Lemma. 
		\item[-] Any formal solution of any linear or non linear analytic differential equation is $1/k$-Gevrey for some $k$. This a theorem of Maillet. In the linear case, the second author gives the optimal Gevrey index by using the Newton polyhedra.
In particular, for a linear system $x^{k+1}dY/dx=A(x)\cdot Y$, this optimal order is the Katz rank of the system.
\end{itemize}

\begin{defi}
\begin{enumerate}
  \item Let $d\in S^1$ a direction. A formal power series $\widehat{f}=\sum_{k\geq 0} a_ix^i$ is $k$-summable in the direction $d$ if there exists a sector $V_d$ bisected by $d$, whose opening is greater than $\pi/k$ and an holomorphic function $f$ on $V$ which is $1/k$-Gevrey asymptotic to $\widehat{f}$.
	\item A formal power series $\widehat{f}$ is $k$-summable if $\widehat{f}$ is $k$-summable in any direction $d$ excepted a finite number of directions (the singular directions).
\end{enumerate}
\end{defi}
Let $\C\{x\}_{1/k,d}\subset\C[[x]]_{1/k}$ be the algebra of $k$-summable series in the direction $d$, and $\C\{x\}_{1/k}\subset\C[[x]]_{1/k}$ the algebra of $k$-summable series. 
Since $V_d$ is a large sector, the summation operator $\Sigma_d:\ \C\{x\}_{1/k,d}\rightarrow \mathcal{A}(V_d)$ is an injective morphism of \C-differential algebra.

\begin{theorem}\label{sommable} Let $x^2dY/dx=A(x)Y$, A(x) in $sl_2(\C\{x\})$, be a germ at $x=0$ of meromorphic differential system. We suppose that the eigenvalues of $A(0)$ are distincts (which avoids a ramification of the independent variable), and that the singularity is irregular, with a Katz rank equal to 1.
\begin{enumerate}
	\item The formal series appearing in a formal fundamental system of solutions $\widehat{Y}$ are 1-summable.
	\item For any non singular direction $d$, the operator $\Sigma_d$ extends to a unique differential morphism from the algebra $\C\{x\}_1[x^\lambda,x^{-\lambda},\exp t/x, \exp(-t/x)]$ into $\mathcal{O}(V_d)$, such that this morphism induces the identity map on $\C[x^\lambda,x^{-\lambda},\exp t/x, \exp(-t/x)]$. This allows us to define $\Sigma_d(\widehat{Y})$ for a formal matrix solution $\widehat{Y}$.
	\item Given a formal fundamental system of solutions 
	\[\widehat{Y}(x)=P(x)x^{L}\exp Q(x),\ Q(x)=diag(tx,-tx),\]
	the singular directions $d$ on which $\Sigma_d(\widehat{Y})$ is not defined are characterized by $\arg(x)=d$ if and only if $exp(2tx)$ has a maximal decay, i.e. by $tx\in\R^-$. 
\end{enumerate}
\end{theorem}

Let $\Sigma$ be the set of singular directions. The Stokes operators $\{S_d,\ d\in\Sigma\}$ are defined in the following way. We choose a regular direction $d_0$. For any singular direction $d$, we choose two regular directions $d^-$ and $d^+$ such that the arc $(d^-,d^+)$ contains the only one singular direction $d$. We choose a determination $\widehat{Y}_0$ of $\widehat{Y}$ around a $d_0$. This choice induces a determination  $\widehat{Y}^-$ of 
$\widehat{Y}$ around the direction $d^-$ (resp. a determination $\widehat{Y}^+$ of $\widehat{Y}$ around  $d^+$) by analytic continuation of the logarithm along the positive arc $(d_0,d^-)$ (resp. $(d_0,d^+)$). The summations $Y^-=\Sigma_{d^-}(\widehat{Y}^-)$ and $Y^+=\Sigma_{d^+}(\widehat{Y}^+)$ are defined on large sectors $V_d^-$ and $V_d^+$ which intersects non trivially on an open sector around $d$. Therefore there exists some constant matrix $S_d$ such that $Y^+=Y^- S_d$. 
One can easily check that
$S_d$ does not depend on the choices of $d^+$ and $d^-$ around $d$. A change of choice for $\widehat{Y}$ or for its determination $\widehat{Y}_0$ around a $d_0$ will modify the family $\{S_ d,\ d\in\Sigma\}$ by a common conjugation.

\begin{rema}\label{pregroupoid} Starting from a formal solution $\widehat{Y}_0$, the Stokes operator $S_d$ are defined by the successive operations: analytic continuation of the formal solution $\widehat{Y}_0$ along an arc from $d_0$ to $d^-$, summation at $d^-$, analytic continuation of $Y^-$ until $d^+$ on $V^+\cap V^-$, anti-summation at $d^+$ (i.e. Taylor expansion of the actual solution $Y^+$) and finally analytic continuation of the formal solution $\widehat{Y}^+$ from $d^+$ to $d_0$. 
\end{rema}

\begin{defi}
\begin{enumerate}
	\item The Stokes operators at each singular direction are defined by the above composition of operators. Given a formal solution $\widehat{Y}_0$, they are characterized by matrices $U_d$, up to a common conjugation.
 \item The formal monodromy is the operator generated by a loop $\widehat{\gamma}$ around $0$ acting on a determination of $\widehat{Y}$. For a given determination $\widehat{Y}_0$, it is defined by a matrix denoted by $\widehat{M}$.
 \item The geometric monodromy is the operator generated by a loop $\gamma$ around $0$ acting on a determination of an actual solution $Y$. For the actual solution $Y_0$ obtained by summation of $\widehat{Y}_0$, it is defined by a matrix denoted by $M$.
 \item The (formal) exponential torus is an action of the algebraic group $\C^*$ on $\widehat{Y}_0$ induced by a rescaling of the exponential map: $\exp q\rightarrow \tau \exp q$. If $\widehat{Y}_0 =Px^L\exp Q$ is chosen such that $Q$ is diagonal, the action of $\tau$
is given by a diagonal matrix $T_\tau=diag(\tau,\tau^{-1})$. In the general case $T_\tau$ is a Cartan component of $SL_2(\C)$
\end{enumerate}
\end{defi}
The motivation for the introduction of the exponential torus action is given by the following heuristic interpretation in a one dimensional context: we consider the equation  $x^2 y'+y=0$.
\label{confluheuristic}
We introduce an \emph{unfolding} of the irregular singularity at the origin~: we replace 
$D=x^2d/dx+1$ by $D_\varepsilon :=x(x-\varepsilon) d/dx+1$
($\varepsilon\in \C^*$). Then the irregular singularity is replaced by two regular singularities at $0$ and $\varepsilon$, with respective exponents $1/\varepsilon$ and $-1/\varepsilon$. An important point to notice is the \emph{coupling} between the relative positions of the singularity and the exponents.
The general solution of $D_\varepsilon y=0$ is 
$y=Cx^{1/\varepsilon}(x-\epsilon)^{-1/\varepsilon}$. It is \emph{invariant} by the monodromy action of a simple loop around the two points $0$ and $\varepsilon$ but the monodromy action of a simple loop around $0$ passing between the two points will transform $f$ into $e^{-2i\pi/\varepsilon} f$. When $\varepsilon\rightarrow 0$, \emph{this monodromy action does not have a limit}.
However we can replace the continuous limit by a \emph{discrete} limit.
We fix $\varepsilon_0\in \C^*$ and we define a sequence
$(\varepsilon_n)_{n\in \N}$ by 
$\frac{1}{\varepsilon_n}:=\frac{1}{\varepsilon_0}+n$, $n\in \N$.
Then the sequence $(e^{-2i\pi/\varepsilon_n})_{n\in \N}$ is constant
and we can interpret $f\mapsto f\tau_0$, with 
$\tau_0:=e^{-2i\pi/\varepsilon_0}$ as the action of a simple loop passing between two \emph{infinitely near} points\footnote{The idea
of considering an irregular singularity as \emph{a pack of infinitely near regular singularities} is due to Ren\'e Garnier in $1919$ \cite{Gar}. It will be used as a guideline in our work.}. As the choice of 
$\varepsilon_0$ is arbitrary, then $\tau_0$ is arbitrary in $\C^*$ and we can consider the exponential torus as a \emph{generalized monodromy group}. This group is no longer discrete, it is an algebraic group of dimension one. The ``monodromy of a loop between the two infinitely near singularities" can be interpreted as a \emph{``random point"} in this group.
In the unfolding bifurcation $D_\varepsilon$ there is a \emph{breaking of symmetry}, the choice of $\varepsilon$ fix a point
$e^{-2i\pi/\varepsilon}\in \C^*$ and the exponential torus is replaced by the ordinary monodromy group generated by
$f\mapsto e^{-2i\pi/\varepsilon}f$. The \emph{random point} is replaced by a \emph{true point}.

\begin{prop}
\begin{enumerate}
	\item  Let $d_1,\ d_2=d_1+\pi$ be the 2 singular directions of the differential system. We have
$M=\widehat{M}\cdot U_{d_2}\cdot U_{d_{1}}$.
  \item  If $\widehat{Y}_0 =Px^L\exp Q$ is chosen such that $Q$ is diagonal, $U_{d_{1}}$ is a lower triangular unipotent matrix and $U_{d_{2}}$ is a upper triangular unipotent matrix. In the general case, $U_{d_{1}}$ and $U_{d_{2}}$ are in the two Borel components of the Cartan component of the exponential torus.
	\item In the present case (a non ramified case), the formal monodromy commutes with the exponential torus (and belongs to it). This fact is no longer true in the ramified case. 
\end{enumerate}
\end{prop}

A \textsl{Borel-Cartan configuration} is a triple of subgroups $(B^-,C,B^+)$ in $SL_2(\C)$ such that $C$ is a maximal torus, isomorphic to the algebraic group $\C^*$, and $B^-$ and $B^+$ are two maximal solvable subgroups which contains $C$. Any Borel-Cartan configuration is conjugated to the canonical one given by $(T^-,D,T^+)$ where $D$ is the subgroup of diagonal matrices, $T^-$ (resp. $T^+$) the subgroup of lower (resp. upper) triangular matrices. The unipotent subgroups of $B^-$ and $B^+$ 
are respectively denoted by $U^-$ and $U^+$.

The triple $(U_{d_1}, \widehat{M}, U_{d_2})$ and the triples $(U_{d_1}, T_\tau, U_{d_2})$ take their values in a unipotent Borel-Cartan configuration $(U^-,C,U^+)$.

The local wild dynamics of the irregular connection is the subgroup (well defined up to a conjugation) of $SL_2(\C)$ generated by the Stokes operators, the formal monodromy and the exponential torus. According to a result of the second author \cite{Ra5}, the Zariski closure of the local dynamics in $SL_2(\C)$ is the differential Galois group of the local linear differential system.

\bigskip

For the study of isomonodromic deformations, we need a parametric version of Theorem (\ref{sommable}).
They are parametric variants of Gevrey power series and of Gevrey expansions (with analytic parameters). One supposes that the estimates in definition (\ref{gevrey}) are uniform on the parameter space $U\subset \C^p$, or more generally on every compact $K$ of $U$~: one replaces $M_W$ and $A_W$ by $M_{W,K}$ and $A_{W,K}$ (resp. uniform on the parameter space). For an open set $U\subset \C^p$, we will denote $\mathcal{O}(U)[[x]]_{1/k}$ the algebra of Gevrey series of order $1/k$ uniformly on $U$. 
They are also a similar parametric variant of $k$-summability using the uniformly parametrized $1/k$-Gevrey expansions : the definition is the same mutatis mutandis. We will speak of uniform $k$-summability. Then the sums of the uniformly $k$-summable series 
in a direction $d$ are analytic in the variable and in the parameters.\footnote{when these sums exist for all values of the parameters. In general the singular directions move with the parameters and it is necessary to reduce the domain.}  For more information on these subjects cf. \cite{Sib}.

\begin{theorem} \label{sommable_par}
Let $U \subset \C^p$ be an open subset. Let $D\subset \C$ be an open disc centered at $0$. Let $[A]:\ x^2dY/dx=A(x,t)Y$, where $A\in sl_2\left(\mathcal{O}(D\times U)\right)$, be a parametrized meromorphic differential system. We suppose that, for all $t=(t_1,\ldots t_p)\in U$, the eigenvalues of $A(0,t)$ are distinct and that the singularity is irregular, with a Katz rank equal to $1$.
\begin{itemize}
\item[(i)]
Locally on $U$, the system $[A]$ admits a formal fundamental solution analytic in the parameters~:
$\hat F=P\widehat{H}x^Le^Q$, where $P\in SL_2\left(\mathcal{O}(V)\right)$, 
$\widehat{H}\in SL_2\left(\mathcal{O}(V)[[x]]_1\right)$, $\widehat{H}(0,t)=I$ (for every $t\in V$), 
$L\in sl_2(\C)$, $Q=(q,-q)$, with $q\in x^{-1}\mathcal{O}(V)[x^{-1}]$.
\item[(ii)]
Let $t_0\in U$ and $V\subset U$ a neighborhood of $t_0$ such that the system admits on $V$ a formal fundamental solution analytic in the parameters as above. Let 
$t_1\in U$ and $d$ a nonsingular direction for the system 
$x^2dY/dx=A(x,t_1)Y$. Then there exists a neighborhood $W\subset V$ of  
$t_1$ such that $P\widehat{H}$ is \emph{uniformly} $1$-summable in the direction $d$ on $W$.
\end{itemize}
\end{theorem}

\begin{proof}
(i) Let $t_0\in U$. If an open polydisc $V$ centered at 
$t_0$ is sufficiently small, then there exists $P\in GL\left(\mathcal{O}(V)\right)$ conjugating the restriction of $A_0=A(0,t)$ to $V$ to a diagonal matrix $B_0$. Therefore we are reduced to the case where $A_0$ is diagonal. Then we can use a parametrized variant of the splitting lemma of \cite{Balser}. It reduces the system to two parametrized formal one dimensional equations. The result follows easily.

(ii)
\emph{A first proof.} We use a parametrized generalization of the Improved Splitting Lemma of \cite{Balser} (8.1, Lemma $11$, page 124). The proof is the same mutatis mutandis\footnote{It uses the interpretation of $1$-summability in terms of the Borel-Laplace method and delicate estimates in the Borel plane.}.

\emph{A second proof.} We imitate a proof of the unparamatrized version based on Gevrey asymptotics (cf. \cite{Lo3}), replacing the Gevrey Main Asymptotic Existence Theorem by a parametrized version.
\end{proof}

\subsection{The local wild fundamental groupoid }

Following an idea which appears in a correspondence between P. Deligne B. Malgrange and the second author \cite{DMR}, the Stokes operators, formal and geometric monodromy around an irregular point appear as a representation of a ``local wild fundamental groupoid''. We want to encode all these operators by loops, or paths if we allow several base points, in a ``halo'' (see also \cite{Lo3}) around the singular point whose internal boundary describes the formal world (represented by determinations of formal solutions) and its exterior the analytic world (represented by sectoral holomorphic solutions).

We consider a small disc $\Delta_0$ around $x=0$, and we perform a real blowing up $E:\ X_0\rightarrow\Delta_0$ of $x=0$. Let $D=E^{-1}(0)\subset X_0$. We draw a second divisor $\widehat{D}$ inside the disc bounded by $D$.
We mark two opposite directions by drawing two points $p_1$ and $p_2$ in the annulus. Let $\mathcal{A}$ be the punctured annulus. 
We consider the groupoid $\pi_1(X_0,D\cup\widehat{D})$ whose objects are the points of $D\cup\widehat{D}$ and whose morphisms are the paths or loops up to homotopy between two objects inside the punctured annulus $\mathcal{A}$.
Let $S_0=\{s_1,s_2\}$ two opposite directions distinct from $p_1$ and $p_2$.

\begin{defi}  The local wild groupoid $\pi_1^\flat(X_0,S_0)$ is the restriction of $\pi_1(X_0,D\cup\widehat{D})$ over $S_0$.
\end{defi}
We define a presentation of $\pi_1^\flat(X_0,S_0)$ in the following way:
\begin{figure}[H]
\centering
\includegraphics[scale=0.6]{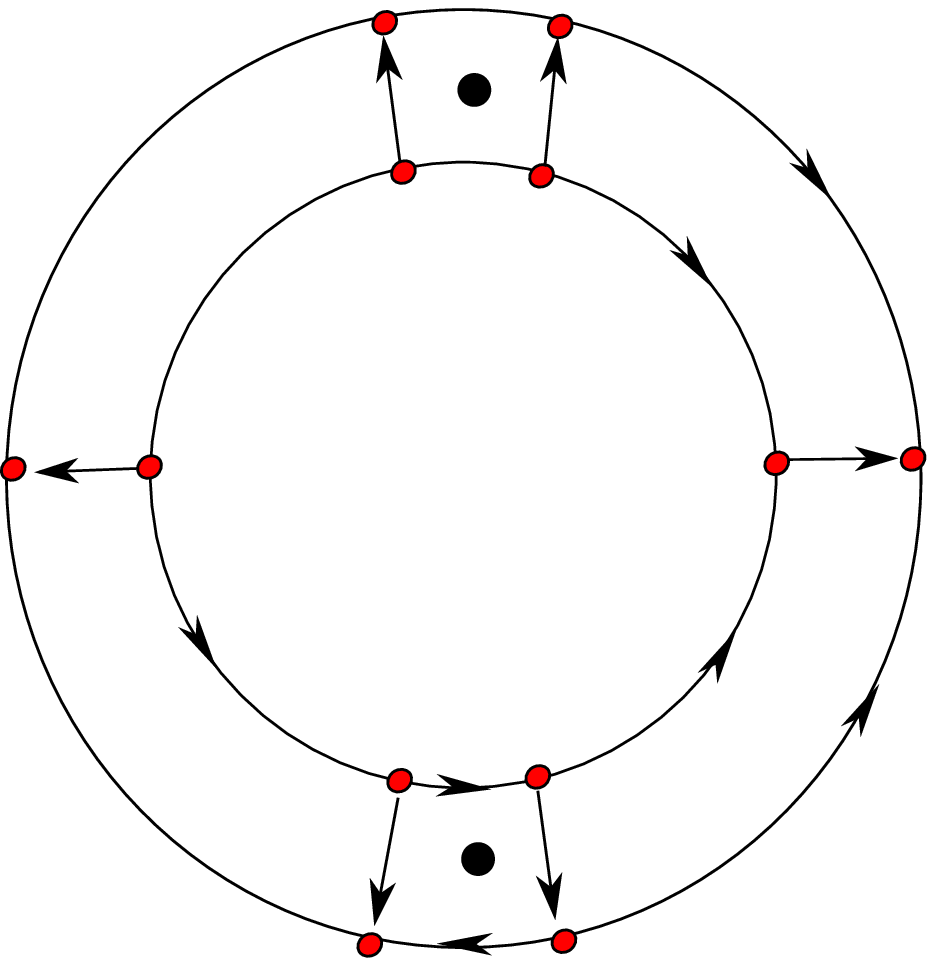}   
\put(-175,82){$s_1$}
\put(3,84){$s_2$}
\put(-84,142){$p_2$}
\put(-84,7){$p_1$}
\put(-107,15){$r_1^+$}
\put(-62,15){$r_1^-$}
\put(-84,-6){$\alpha_1$}
\put(-84,35){$\widehat{\gamma}_1^-$}
\put(-124,55){$\widehat{\gamma}_1^+$}
\put(-152,90){$r_1$}
\put(-18,92){$r_2$}
\put(-57,55){$\widehat{\gamma}_{1,2}^r$}
\put(-10,35){$\gamma_{1,2}^r$}
\put(-55,110){$\widehat{\gamma}_{1,2}^l$}
\put(-15,140){$\gamma_{1,2}^l$}

\caption{A presentation of $\pi_1^\flat(X_0,S_0)$.}
\label{pi1Vlocflat}
\end{figure}
We choose rays $r_i^-$, $r_i^+$ on the left and right side of $p_i$ for $i=1,2$. We choose two opposite base points $s_1$ and $s_2$ on $D$ on the orthogonal direction to $(p_1,p_2)$, endpoints of 2 rays $r_1$ and $r_2$.
Let $\widehat\gamma_i^-$ (resp.$\widehat\gamma_i^+$) be the arc from $s_i$ to the origin of $r_i^-$ (resp. $r_i^+$) in $\widehat{D}$, and $\alpha_i$ the arc on $D$ from the end of $r_i^-$ to the end of $r_i^+$ on $D$. The two Stokes loops are the loops based in $s_i$ (see figure above):
\[\sigma_i =r_i^{-1}\cdot\widehat\gamma_i^-\cdot r_i^-\cdot \alpha_i\cdot (r_i^+)^{-1}\cdot(\widehat\gamma_i^+ )^{-1}r_i,\ i=1,2.\]
Let $\gamma_{1,2}^r$ and $\gamma_{1,2}^l$ the lower and upper arcs from $s_1$ to $s_2$ on $D$, and $\widehat\gamma_{1,2}^r$ and $\widehat\gamma_{1,2}^l$ their analog on $\widehat D$ by using $r_1$, an arc on $\widehat D$, and $r_2^{-1}$.
The local wild groupoid is generated by $\sigma_1$, $\sigma_2$, $\gamma_{1,2}^r$, $\gamma_{1,2}^l$, $\widehat\gamma_{1,2}^r$ and $\widehat\gamma_{1,2}^l$.
Let $\gamma_{1,1}=\gamma_{1,2}^r\cdot \gamma_{1,2}^l$, $\widehat\gamma_{1,1}=\widehat\gamma_{1,2}^r\cdot \widehat\gamma_{1,2}^l$. From the above picture we have the relation
\[\gamma_{1,1}=\sigma_1\cdot\widehat\gamma_{1,2}^r\cdot \sigma_2 \cdot \widehat\gamma_{2,1}^l.\]
In order to include the exponential torus in a linear representation of this groupoid, we introduce a (non discrete) extension $\pi_1(X_0,S_0)$ of $\pi_1^\flat(X_0,S_0)$: we glue a representation of $(C^*,\times)$ by adding two collections of loops $\widehat t_{1,1}(\kappa)$, $\kappa\in\C^*$ based in $\widehat{s}_1$, and $\widehat t_{2,2}(\kappa)$, $\kappa\in\C^*$ based in $\widehat{s}_2$.
We denote:
\[t_{1,1}(\kappa)=r_1^{-1}\cdot \widehat t_{1,1}(\kappa)\cdot r_1 (\mbox{ based in } s_1),\
t_{2,2}(\kappa)=r_2^{-1}\cdot \widehat t_{2,2}(\kappa)\cdot r_2 (\mbox{ based in } s_2).\]
We require the relations $\mathcal{R}_{loc}$:
\begin{equation*}\label{Relations loc}
\begin{aligned}
&(i)\ \forall \kappa \in \C^*,\ \forall \kappa' \in \C^*,\ t_{i,i}(\kappa\kappa')= t_{i,i}(\kappa)\cdot t_{i,i}(\kappa'),\ i=1,2;\\
&(ii)\ \forall \kappa \in \C^*,\ [[\sigma_i, t_{i,i}(\kappa)],\sigma_i]=\star_i,\ i=1,2;\\
&(iii)\ \forall \kappa \in \C^*, \widehat{\gamma}_{i,i}\cdot t_{i,i}(\kappa) \cdot \widehat{\gamma}_{i,i}^{-1} t_{i,i}(\kappa)^{-1}=\star_i,\ i=1,2;\\
&(iv)\ t_{1,1}(\kappa)\cdot\widehat\gamma_{1,2}^l\cdot t_{2,2}(\kappa)\cdot \widehat\gamma_{2,1}^r=\star_1.
\end{aligned}
\end{equation*}
\begin{defi} \label{wlg}The (complete) wild local groupoid is the groupoid $\pi_1(X_0,S_0)$ generated by $\pi_1^\flat(X_0,S_0)$ and the families $T_{1,1}=\{t_{1,1}(\kappa),\ \kappa\in\C^*\}$ and $T_{2,2}=\{t_{2,2}(\kappa),\ \kappa\in\C^*\}$, with the above relations $\mathcal{R}_{loc}$.
\end{defi}
\begin{figure}[H]
\centering
\includegraphics[scale=0.6]{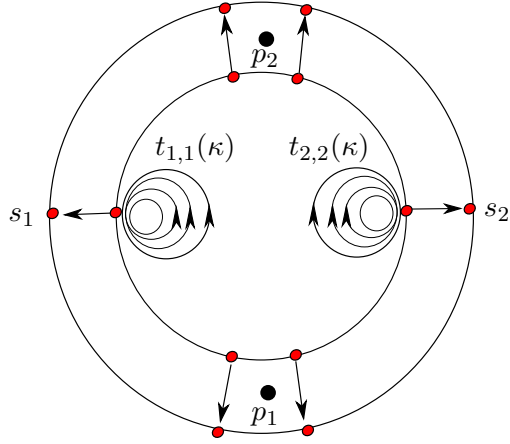}   
\put(-175,82){$s_1$}
\put(3,84){$s_2$}
\put(-84,142){$p_2$}
\put(-84,7){$p_1$}
\put(-120,107){$t_{1,1}(\kappa)$}
\put(-70,107){$t_{2,2}(\kappa)$}
\caption{A presentation of $\pi_1(X_0,S_0)$.}
\label{pi1Vloc}
\end{figure}

\subsection{The global wild fundamental groupoid }

The local wild fundamental groupoid suffices to deal with the Euler equation, or the Kummer equations by adding one base point \cite{MR}. Nevertheless for the Painlevé equations, we have to consider several regular and irregular singularities, with connecting data between them (and in some cases an order two ramification). 

In order to represent the global monodromy data together with the wild local dynamic of a connection in $\mathcal{M}_V$, we consider the following groupoid.
We begin with $B=\p^1\setminus\{p_0,p_1,p_\infty\}$. We perform real blowing up at each of these points, and we obtain a variety $X$ with divisors $D_0$, $D_1$ and $D_\infty$. We choose two opposite base points $s_1$ and $s_2$ on $D_\infty$, a base point $s_3$ on $D_0$, and $s_4$ on $D_1$. Let $S=\{s_1,s_2,s_3,s_4\}$. We consider the groupoid whose morphisms are the paths between the base points up to homotopy. Inside $(D_\infty,s_1,s_2)$, we glue the local wild groupoid $\pi_1^\flat(X_0,S_0)$ (resp. the complete local wild groupoid $\pi_1(X_0,S_0)$). We obtain a new groupoid denoted by $\pi_1^{V,\flat}(X,S)$ (resp. $\pi_1^{V}(X,S)$ for the complete version). A presentation of this groupoid is given by the figure:
\begin{figure}[H]
\centering
\includegraphics[scale=0.6]{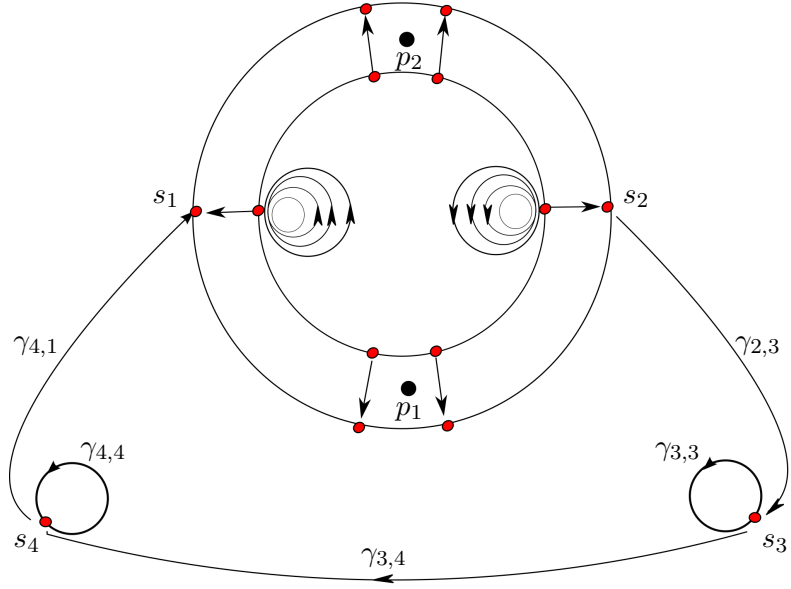}   
\put(-10,25){$s_3$}
\put(-290,25){$s_4$}
\put(-20,100){$\gamma_{2,3}$}
\put(-290,100){$\gamma_{4,1}$}
\put(-160,20){$\gamma_{3,4}$}
\put(-147,75){$p_1$}
\put(-147,207){$p_2$}
\put(-238,155){$s_1$}
\put(-62,155){$s_2$}
\put(-50,60){$\gamma_{3,3}$}
\put(-265,60){$\gamma_{4,4}$}
\caption{A presentation of $\pi_1^V(X,S)$.}
\label{pi1V}
\end{figure}
The relations of this presentation are generated by :

-- the local relations $\mathcal{R}_{loc}$ previously defined; 

--  $\mathcal{R}_{ext}$:\ $\gamma_{1,2}^l\cdot \gamma_{2,3}\cdot \gamma_{3,4}\cdot \gamma_{4,1}=\star_1$;

-- $\mathcal{R}_{int}$:\ $\gamma_{1,2}^r\cdot \gamma_{2,3}\cdot \gamma_{3,3}\cdot\gamma_{3,4}\cdot\gamma_{4,4}\cdot \gamma_{4,1}=\star_1$.

\bigskip
We denote by $\pi_1^{V,\kappa}(X,S)$ the fiber of the morphism $K:\ \pi_1^{V}(X,S)\rightarrow \C^*$ induced by 
$K(t_{1,1}(\kappa))=\kappa$. 

\subsection{A confluent morphism of groupoid}

The comparison of $\pi_1^{V,\kappa}(X,S)$ with $\pi_1^{VI}(X,S)$ is a central point in order to define the confluent process. For this purpose, we introduce another presentation of this groupoid, setting:
\begin{equation*}
\begin{aligned}
& t_{1,2}(\kappa)=t_{1,1}(\kappa^{-1})\cdot \widehat{\gamma}_{1,2}^r=\widehat{\gamma}_{1,2}^l\cdot t_{2,2}(\kappa^{-1}),\\
& \gamma_{2,3}(\kappa)=t_{2,2}(\kappa^{-1})\cdot\sigma_2^{-1}\cdot\gamma_{2,3}.
\end{aligned}
\end{equation*}
We obtain the following figure
\begin{figure}[H]
\hspace{-2cm}
    \begin{minipage}[l]{.45\linewidth}
        \includegraphics[scale=0.4]{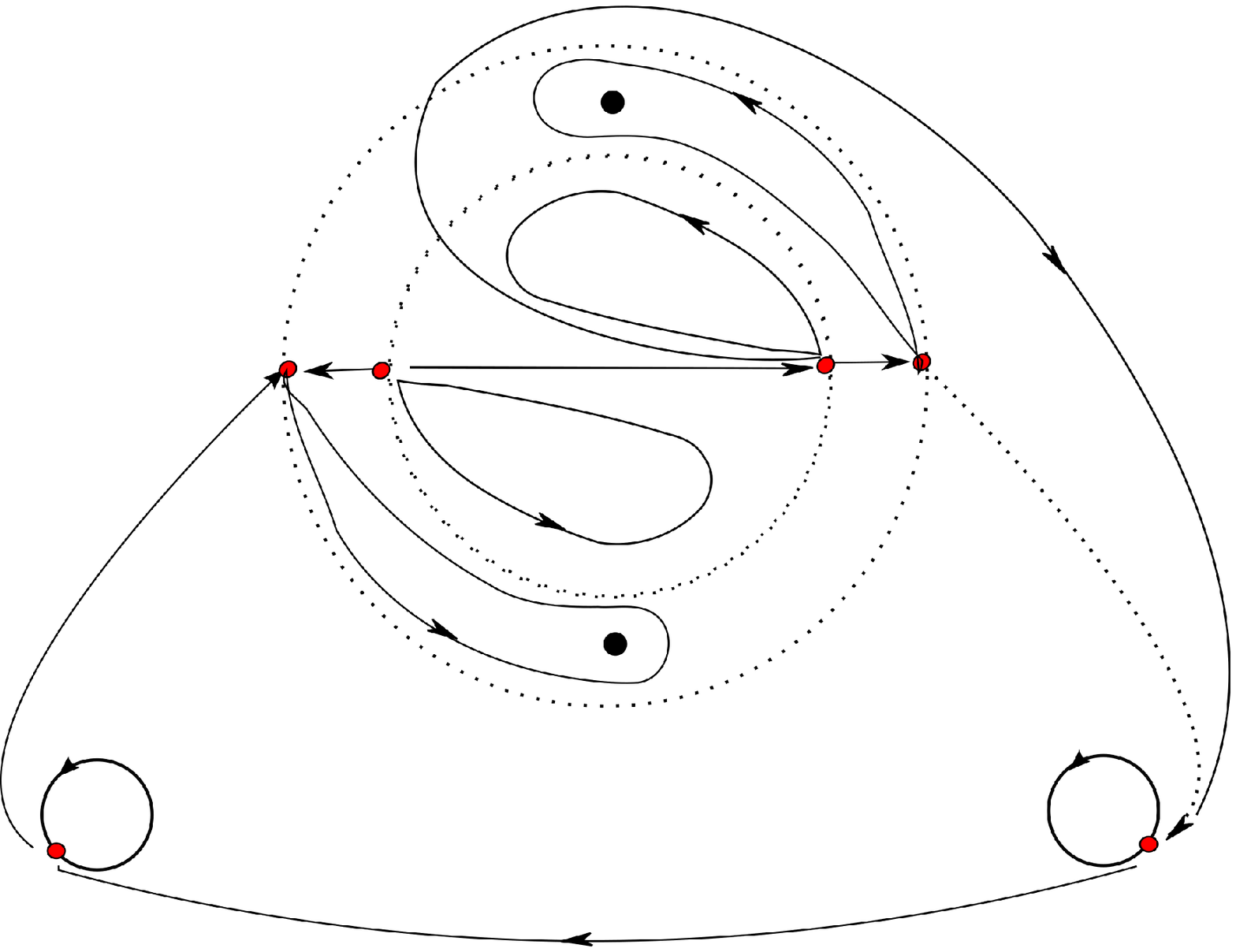}
    \end{minipage}
\put(60,25){$\simeq$}
\put(-100,30){$\scriptstyle{t_{1,2}(\kappa)}$}
\put(-85,10){$\scriptstyle{t_{1,1}(\kappa)}$}
\put(-65,40){$\scriptstyle{t_{2,2}(\kappa)}$}
\put(15,40){$\scriptstyle{\gamma_{2,3}(\kappa)}$}
\put(-50,-20){$\scriptstyle{\sigma_1}$}
\put(-85,70){$\scriptstyle{\sigma_2}$}
    \begin{minipage}[l]{.45\linewidth}
        \includegraphics[scale=0.4]{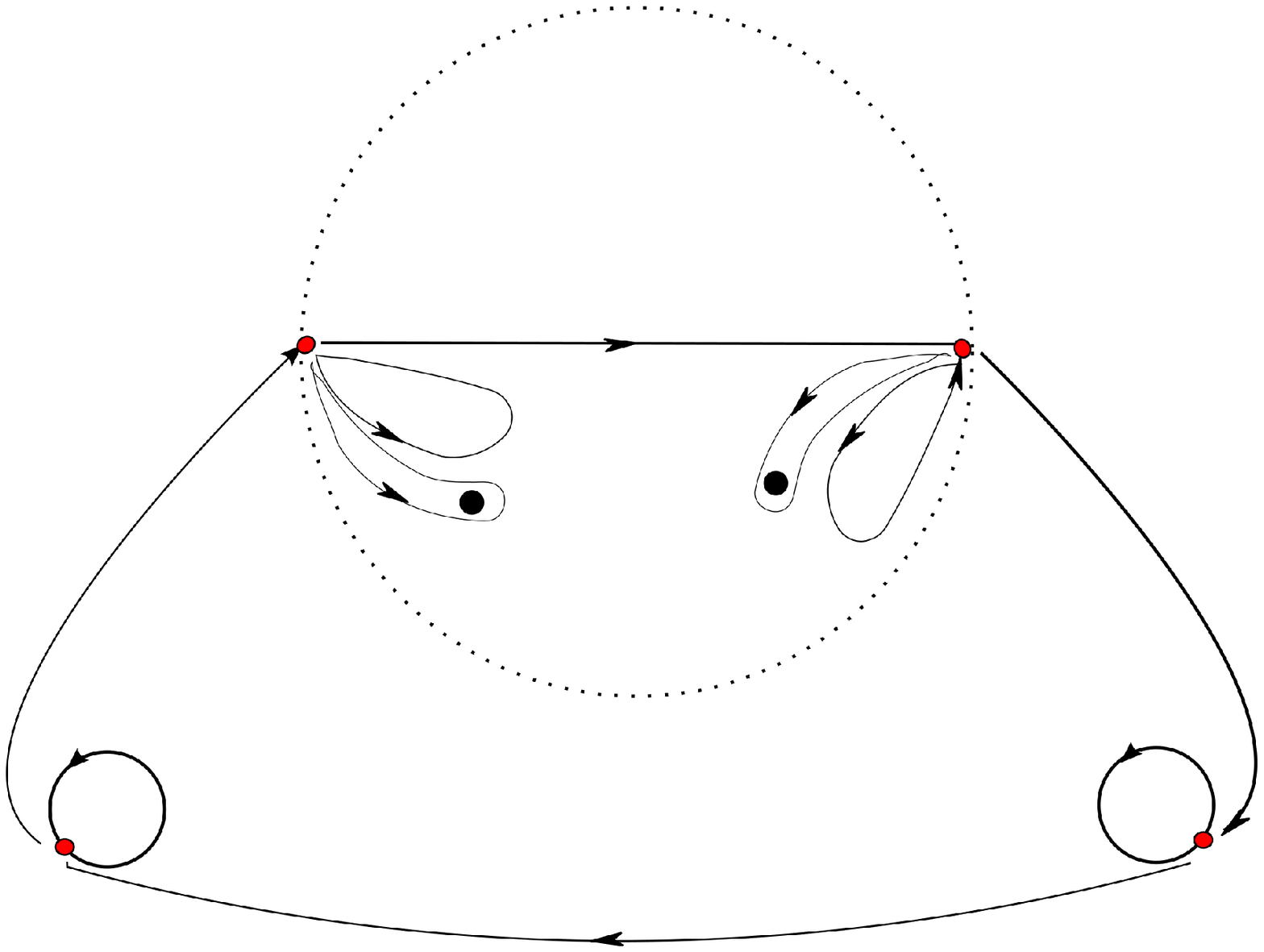}
    \end{minipage}
\put(-70,30){$\scriptstyle{t_{1,2}(\kappa)}$}
\put(-75,10){$\scriptstyle{t_{1,1}(\kappa)}$}
\put(-75,0){$\scriptstyle{\sigma_1}$}
\put(-50,-10){$\scriptstyle{t_{2,2}(\kappa)}$}
\put(-50,0){$\scriptstyle{\sigma_2}$}
\put(15,10){$\scriptstyle{\gamma_{2,3}(\kappa)}$}
    \hfill
		\caption{Another presentation for $\pi_1^{V,\kappa}(X,S)$.}
\end{figure}

This last figure looks like the figure (\ref{pi1VI}) for the groupoid $\pi_1^{VI}(X,S)$. More precisely, we denote by $s'_i$ the objects of $\pi_1^{VI}(X,S)$, and we consider a presentation $\{\gamma'_{i,j},\mathcal{R}'_{int},\mathcal{R}'_{ext}\}$ of $\pi_1^{VI}(X,S)$ given by figure (\ref{pi1VI}). We consider the groupoid morphism $\varphi_\kappa$ from $\pi_1^{VI}(X,S)$ to $\pi_1^{V,\kappa}(X,S)$ which sends $s'_i$ on $s_i$ and which is defined  by:
\begin{itemize}
\item $\varphi_\kappa(\gamma'_{3,3})={\gamma}_{3,3}$, $\varphi_\kappa(\gamma'_{4,4})={\gamma}_{4,4}$, $\varphi_\kappa(\gamma'_{3,4})={\gamma}_{3,4}$,
$\varphi_\kappa(\gamma'_{4,1})={\gamma}_{4,1}$,
\item $\varphi_\kappa(\gamma'_{1,1})=\sigma_1\cdot t_{1,1}{(\kappa)}$,
$\varphi_\kappa(\gamma'_{2,2})=\sigma_2\cdot t_{2,2}{(\kappa)}$,
\item $\varphi_\kappa(\gamma'_{1,2})=t_{1,2}{(\kappa)}$,
$\varphi_\kappa(\gamma'_{2,3})={\gamma}_{2,3}{(\kappa)}$.
\end{itemize}

This morphism $\varphi_\kappa$ is an injective morphim but not a surjective one: the pre-image of a Stokes loop or a loop of an exponential tori are not defined.

\begin{defi}\label{morphism_conf}
The confluent morphism from $\varphi:\ \pi_1^{VI}(X,S)\rightarrow\pi_1^{V}(X,S)$ is defined by the family
$\varphi=(\varphi_\kappa)_{\kappa\in\C^*}$ . 
\end{defi}

There exists a similar result in the context of the hypergeometric equation and the confluent hypergeometric equation in Kummer form.

\section{The character variety $\chi_V$ }

\subsection{Definition of $\chi_V$ }

The class of rank 2 linear representations of a fundamental groupoid in $SL_2(C)$ has been defined in the introduction. Here we are only interested in a subclass of linear representations $\rho$ which satisfy the following property
\begin{defi} A representation $\rho$ of $\pi_1^{V}(X,S)$ satisfies the property ($\star$) if there exists a Borel-Cartan configuration $(B^-,C,B^+)$ such that:
\begin{enumerate}
\item $\rho(t_{1,1}(\kappa),\ \kappa\in\C^*)=C$,
\item $\rho(\sigma_1)\in U^-,\ \rho(\widehat{\gamma}_{1,2}^l\sigma_2\widehat{\gamma}_{2,1}^l)\in U^+.$
\end{enumerate}
\end{defi}
If this property holds for $\rho$ it still holds for $\rho'$ equivalent to $\rho$.
\begin{defi}
The character variety $\chi_V$ is the categorical quotient of the set of linear representations $\pi_1^{V}(X,S)$ which satisfy the property $(\star)$ through the equivalence of representations. 
\end{defi}

A local representation is a representation over the sub-groupoid restricted to the local loops, generated by $\widehat{\gamma}_{1,1}$, $\gamma_{3,3}$, $\gamma_{4,4}$. It is characterized by $a=(a_0,a_3,a_4)$ with $a_0=tr(\rho(\widehat{\gamma}_{1,1}))=tr(\rho(\widehat{\gamma}_{2,2}))$, $a_3=tr(\rho(\gamma_{3,3}))$, $a_4=tr(\rho(\gamma_{4,4}))$. Any representation $\rho$ in $\chi_V$  determines a unique local representation. The fiber of this map is denoted by $\chi_V(a)$.

The inclusion of $\pi_1^{V,\flat}(X,S)$ in $\pi_1^{V}(X,S)$ induces by restriction a map $r$ whose image is denoted by $\chi_V^\flat$.

\begin{prop}\label{ext} For $a$ such that $a_0\neq \pm 2$, $r$ is a (2:1) map over the open set of $\chi_V^\flat$ defined by $\rho(\sigma_1)$ or $\rho(\sigma_2)$ is a non trivial element.\end{prop}

\begin{proof}
If we suppose that $\sigma_1$ is non trivial, one can choose $\widehat{Y_0}$ such that 
\[\rho(\sigma_1)=\left(\begin{array}{cc} 1 & 0 \\  u_1 & 1\end{array}\right).\]
In this case, $\rho(t_{1,1}(\kappa),\kappa\in\C^*)=C$ is the diagonal subgroup $D$ of $SL_2(\C)$, and we have
\[\rho(t_{1,1}(\kappa))=\left(\begin{array}{cc} f(\kappa) & 0 \\ 0 & f(\kappa^{-1})\end{array}\right).\]
From the relation $t_{1,1}(\kappa)t_{1,1}(\kappa')=t_{1,1}(\kappa\kappa')$ we deduce that $f$ belongs to $Aut(\C^*,\times)$, and therefore 
$f(\kappa)=\kappa$ or $f(\kappa)=\kappa^{-1}$ .
\end{proof}

\begin{coro} \label{ext2}
 The character variety $\chi_V(a)=\chi_V^+(a)\cup\chi_V^-(a)$ with 
\begin{equation*}
\begin{aligned}
&\chi_V^+(a)=\{(U_1,M_0,U_2,M_3,M_4,D_\kappa)\}//SL_2(\C),\\ 
&\chi_V^-(a)=\{(U_1,M_0,U_2,M_3,M_4,D_\kappa^{-1})\}//SL_2(\C).
\end{aligned}
\end{equation*} 
The two copies $\chi_V^+(a)$ and $\chi_V^-(a)$ coincide over $a$ such that $e_0=e_0^{-1}$ i.e. such that $a_0=\pm 2$.
\end{coro}%

\subsection{The character variety $\chi_V$ and cubic surfaces}

\begin{defi}
A representation of $\pi_1^{V,\flat}(X,S)$ in $SL_2(\C)$ is normalized if, 
\[\rho(\gamma^l_{1,2})=\rho(\gamma_{2,3})=\rho(\gamma_{3,4})=I,\]
where the $\gamma_{i,j}$ are generators of the presentation introduced in figure \ref{pi1V}.
\end{defi}
There always exists a normalized representation in each class of equivalent representations in $SL_2(\C)$, obtained by choosing a representation of an initial object --say $s_2$-- and by representing the successive others objects by analytic continuation along $\gamma_{2,3}$, $\gamma_{3,4}$ and $\gamma^l_{2,1}$. From the relation $\mathcal{R}_{ext}$, for a normalized representation we also have $\rho(\gamma_{4,1})=I$.

\begin{prop} A normalized representation of $\pi_1^{V,\flat}(X,S)$ is characterized by the data
\[{M}_0=\rho(\widehat{\gamma}_{1,1}),\ U_1=\rho(\sigma_1),\ U_2=\rho(\widehat{\gamma}_{1,2}^l\cdot\sigma_2\cdot\widehat{\gamma}_{2,1}^l),\  M_3=\rho(\gamma_{3,3}), M_4=\rho(\gamma_{4,4})\]
such that $U_1{M}_0U_2M_3M_4=I$, up to a common conjugacy.
\end{prop}
\begin{proof}
We have to prove that, for a normalized representation $\rho$, the images of all the generators of the groupoid by $\rho$ are well defined.
Since $\rho(\gamma^l_{1,2})=I$, we have
\[\rho(\widehat{\gamma}^l_{1,2})=U_2,\ \rho(\widehat{\gamma}^r_{1,2})=M_0U_2,\ \rho(\gamma^r_{1,2})=U_1M_0U_2.\]
The relation $\mathcal{R}_{int}$ gives $U_1{M}_0U_2M_3M_4=I$. A change of representation of the initial object will change this data by a common conjugacy. \end{proof}
\bigskip\\
One can suppose, by using a conjugacy, that the Cartan subgroup $C=\{\rho(t_{1,1}(\kappa)),\ \kappa\in\C^*\}\subset SL_2(\C)$ is the diagonal one $D$. From the relation (iii) of the local groupoid presentation, $\widehat{M_0}$ commutes with each element of $C$. Therefore from the relations (i) (ii) and (iii), we have:
\[U_1=\left(\begin{array}{cc}1 & 0\\ u_1 & 1 \end{array}\right),\
{M}_0=\left(\begin{array}{cc}e_0 & 0\\ 0 & e_0^{-1}\end{array}\right),\
U_2=\left(\begin{array}{cc}1 & u_2\\ 0 & 1 \end{array}\right),\]
\[M_3=\left(\begin{array}{cc}\alpha_3 & \beta_3\\ \gamma_3 & \delta_3 \end{array}\right),\
M_4=\left(\begin{array}{cc}\alpha_4 & \beta_4\\ \gamma_4 & \delta_4 \end{array}\right).\]
This data is now defined up to a conjugacy by an element of $D$. Finally we have
\[\chi_V^\flat=\{(U_1,M_0,U_2,M_3,M_4)\in (U^-,D,U^+)\times SL_2(\C)^2,\ U_1M_0U_2M_3M_4=I\}//D.\]
This algebraic quotient is a 5 dimensional space. The function $(a^+,x^+)=(e_0,a_3,a_4,x_1^+,x_2^+,x_3^+)$:
\begin{equation*}
\begin{aligned}
&e_0=M_0[1,1],\ a_3=tr(M_3),\ a_4=tr(M_4),\\
&x_1^+=M_3[2,2]=\delta_3,\ x_2^+=M_4[2,2]=\delta_4,\ x_3^+=tr(U_1M_0U_2)=e_0+e_0^{-1}+e_0u_1u_2,
\end{aligned}
\end{equation*}
is invariant under the action of $D$.

\begin{prop}\label{equation chiV0}
The coordinates $(a^+,x^+)$ define a map $Tr_V^+$, invertible for a generic $a$, from $\chi_V^\flat(a)$ to the family affine cubic surface $\mathcal{C}_V(\theta^+)$ defined by
\[F_V(\theta^+,x)=x_1x_2x_3+x_1^2+x_2^2-\theta_1^+x_1-\theta_2^+x_2-\theta_3^+x_3+\theta_4^+=0,\]
where $\theta_1^+=a_3+e_0 a_4,\ \theta_2^+=a_4+e_0 a_3,\ \theta_3^+=e_0,\ \theta_4^+=e_0^2+e_0a_3a_4+1.$
\end{prop}

\begin{lemma}\label{coeff-traces} We have:
\begin{equation*}
    \begin{aligned}
e_0u_1u_2&=x_3-e_0-e_0^{-1}\\
e_0\gamma_3s_2&=x_2-e_0a_3+e_0x_1\\
e_0\beta_3s_1&=-x_1x_3+e_0x_1-x_2+a_4
    \end{aligned}
\end{equation*}
\end{lemma}

\begin{proof} We have $x_3=e_0+e_0^{-1}+e_0u_1u_2$. From $U_1M_0U_2M_3=M_4^{-1}$ we obtain
\[\delta_4=\alpha_3e_0+\gamma_3e_0u_2,\ \alpha_4=\beta_3e_0s_1+\delta_3e_0u_1u_2+\delta_3e_0^{-1}.\]
Using $\alpha_i=a_i-\delta_i$, we obtain the two other equalities.
\end{proof}
\bigskip

\emph{Proof of Proposition (\ref{equation chiV0}).}
Since $\alpha_3\delta_3-\beta_3\gamma_3=1$, we have
\[(e_0^2u_1u_2)(\alpha_3\delta_3)-(e_0\beta_3u_1)(e_0\gamma_3u_2)-e_0^2u_1u_2=0.\]
If we replace $e_0s_1s_2$, $e_0\gamma_3s_2$ and $e_0\beta_3s_1$ by the above expressions, we obtain the equation $F_V(x,\theta^+)=0$.
This proves that the map $(a^+,x^+)$ takes its values in the cubic surface.
Given a point in the cubic surface, we recover
$\alpha_3=a_3-x_1$, $\alpha_4=a_4-x_2$, $\beta_3\gamma_3=1-x_1(a_3-x_1)$, $\beta_4\gamma_4=1-x_2(a_4-x_2)$, $\gamma_3s_2=e_0^{-1}x_2-a_3+x_1$, and therefore, for generic values, a unique point of $\chi_V^\flat(a)$.
\carre
\begin{rema}
It might seen more natural to consider the trace coordinates:
\begin{equation*}
\begin{aligned}
&y_1(\rho)=tr(\rho(\widehat\gamma_{1,1}\gamma_{1,3}\gamma_{3,3})),\
y_2(\rho)=tr(\rho(\widehat\gamma_{2,2}\gamma_{2,4}\gamma_{4,4})),\
y_3(\rho)=tr(\rho(\gamma_{3,3}\gamma_{3,4}\gamma_{4,4}))
\end{aligned}
\end{equation*}
which, for a normalized representation, give:
\begin{equation*}
\begin{aligned}
&y_1=tr({M}_0M_3), y_2= tr({M}_0M_4), y_3=tr(M_3M_4).
\end{aligned}
\end{equation*}
But clearly, the coordinates $y_1$ and $y_2$ degenerate if $M_0=\pm I$, since we have $y_1=a_3$ and $y_2=a_4$ in this case. The coordinates
\[x_1^+=\frac{y_1-e_0a_3}{e_0^{-1}-e_0}=\delta_3,\ x_2^+=\frac{y_2-e_0a_4}{e_0^{-1}-e_0}=\delta_4,\ x_3^+=y_3 \]
give an extension to this exceptional case $M_0=\pm I$.
This choice is also the usual one in the litterature: see \cite{VdPSai} or \cite{Kl1}. Since the coordinates $(a^+,x^+)$ are directly related to the above trace coordinates through an affine map, we still denote the corresponding map by $Tr_V^+$, and we call it a trace map.
\end{rema}
The normalization of the representation by $(U_1,M_0,U_2,M_3,M_4)$ used in order to construct $Tr_V^+$ requires that $U_1$ is a lower triangular matrix. If we require that $U_1$ must be an upper triangular matrix, we obtain the data $(U_1^-,M_0^-,U_2^-,M_3^-,M_4^-)$ which is equivalent to the first data by the conjugacy with the matrix \[P=\left(\begin{array}{cc}0 & 1\\ -1 & 0 \end{array}\right). \]
We define the trace map $Tr^-(\rho)=(a^-,x^-)=(e_0^{-1},a_3,a_4,x_1^-,x_2^-,x_3^-)$ with 
\[e_0^{-1}=M_0^-[1,1],\ x_1^-=M_3^-[1,1]=\alpha_3,\ x_2^-=M_4^-[1,1]=\alpha_4,\ x_3^-=tr(U_1^-M_0^-U_2^-)=x_3^+.\]
This map takes its values in the cubic surface $\mathcal{C}_V(\theta^-)$:
\[F_V(x,\theta^-)=x_1x_2x_3+x_1^2+x_2^2-\theta_1^-x_1-\theta_2^-x_2-\theta_3^-x_3+\theta_4^-=0,\]
where $\theta_1^-=a_3+e_0^{-1} a_4$, $\theta_2^-=a_4+e_0^{-1} a_3$, $\theta_3^-=e_0^{-1}$, $\theta_4^-=e_0^{-2}+e_0^{-1}a_3a_4+1$ 
which can also be identified to the categorical quotient $\chi_V^\flat(a)$. The points $p^+\in\mathcal{C}_V(\theta^+)$ and  $p^-=Tr_V^-\circ(Tr_V^+)^{-1}(p^+)$ correspond to the same representation in $\chi_V^\flat(a)$.
\bigskip

From (\ref{ext2}) we know that $\chi_V(a)=\chi_V^+(a)\cup\chi_V^-(a)$ where 
\begin{equation*}
\begin{aligned}
&\chi_V^+(a)=\{(U_1,M_0,U_2,M_3,M_4,D_\kappa)//SL_2(\C)\},\\ 
&\chi_V^-(a)=\{(U_1,M_0,U_2,M_3,M_4,D_\kappa^{-1})//SL_2(\C)\}.
\end{aligned}
\end{equation*}
Using a conjugacy with $P$
this second data is equivalent to $(U_1^-,M_0^-,U_2^-,M_3^-,M_4^-,\rho(t_{1,1}(\kappa))=D_\kappa)$, which define an element in $\mathcal{C}_V(\theta^-)$. Therefore $\chi_V(a)$ is identified through the trace coordinates to the union of the two affine surfaces $\mathcal{C}_V(\theta^-)\cup\mathcal{C}_V(\theta^+)$

We denote by
$Tr^{\pm ,\kappa}$ the restriction of $Tr^\pm$ over $\pi_1^{V,\kappa}(X,S)$.

\subsection{Lines and reducibility locus}

For what follows in this article, excepted (\ref{candyn}), we will suppose that the parameters satisfy the generic conditions:
\[e_0\neq\pm1,\ e_3\neq\pm1,\ e_4\neq\pm1,\ e_0e_3^{\varepsilon_3}e_4^{\varepsilon_4}\neq 1\mbox{ for }\varepsilon_3=\pm 1,\ \varepsilon_4=\pm 1.\]
The description of the lines in cubic surfaces can be found in \cite{Dolg}. In particular
the compactification of a generic element $\mathcal{C}_{VI}(\theta)$ in  $\p^3(\C)$ contains 27 lines.
In the figure below the central triangle is the union of the three lines at infinity:
\begin{figure}[H]
\centering
\hspace{2cm}\includegraphics[scale=0.6]{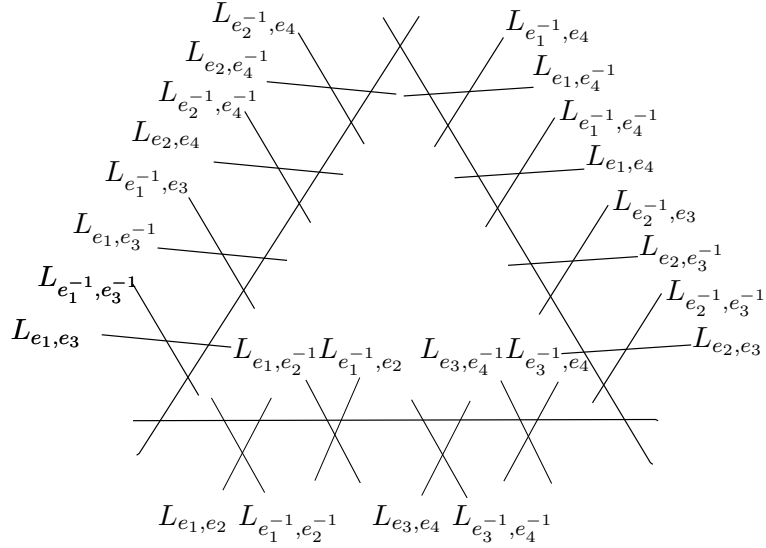}
\put(-200,-10){$L_{e_1,e_2}$}
\put(-170,-10){$L_{e_1^{-1},e_2^{-1}}$}
\put(-120,-10){$L_{e_3,e_4}$}
\put(-90,-10){$L_{e_3^{-1},e_4^{-1}}$}
\put(-172,53){$L_{e_1,e_2^{-1}}$}
\put(-140,53){$L_{e_1^{-1},e_2}$}
\put(-102,53){$L_{e_3,e_4^{-1}}$}
\put(-70,53){$L_{e_3^{-1},e_4}$}
\put(-255,60){$L_{e_1,e_3}$}
\put(0,57){$L_{e_2,e_3}$}
\put(-245,80){$L_{e_1^{-1},e_3^{-1}}$}
\put(-10,75){$L_{e_2^{-1},e_3^{-1}}$}
\put(-255,60){$L_{e_1,e_3}$}
\put(-245,80){$L_{e_1^{-1},e_3^{-1}}$}
\put(-232,100){$L_{e_1,e_3^{-1}}$}
\put(-220,120){$L_{e_1^{-1},e_3}$}
\put(-210,135){$L_{e_2,e_4}$}
\put(-200,150){$L_{e_2^{-1},e_4^{-1}}$}
\put(-190,165){$L_{e_2,e_4^{-1}}$}
\put(-180,180){$L_{e_2^{-1},e_4}$}
\put(-20,92){$L_{e_2,e_3^{-1}}$}
\put(-30,109){$L_{e_2^{-1},e_3}$}
\put(-40,126){$L_{e_1,e_4}$}
\put(-50,143){$L_{e_1^{-1},e_4^{-1}}$}
\put(-60,160){$L_{e_1,e_4^{-1}}$}
\put(-70,177){$L_{e_1^{-1},e_4}$}

\caption{The 24 lines in $\mathcal{C}_{VI}(\theta)$.}\label{droites chiVI}
\end{figure}
\noindent The equations of the lines are obtained from the following decomposition which appears in \cite{Kl1}:
\[F_{VI}(x,\theta)=(x_k-c_{e_i,e_j})(F_{VI, x_k}-x_k+c_{e_i,e_j})+l_{e_i,e_j}l_{e_i^{-1}e_j^{-1}},\]
where $F_{VI,x_k}$ is the partial derivative of $F_{VI}(x,\theta)$ with respect to the variable $x_k$, $c_{\alpha,\beta}=\alpha\beta^{-1}+\alpha^{-1}\beta$ and
\[l_{e_i,e_j}=e_ix_i+e_jx_j-a_ke_ie_j-a_4,\ l_{e_i^{-1},e_j^{-1}}=e_i^{-1}x_i+e_j^{-1}x_j-a_ke_i^{-1}e_j^{-1}-a_4.\]
Therefore the plane $x_k=c_{e_i,e_j}$ intersects the cubic surface along a degenerated conic, union of the two lines
\begin{equation*}
L_{e_i,e_j}:\ x_k=c_{e_i,e_j}=e_ie_j^{-1}+e_i^{-1}e_j \mbox{ and } l_{e_i,e_j}=0,\\
L_{e_i^{-1},e_j^{-1}}:\ x_k=c_{e_i^{-1},e_j^{-1}} \mbox{ and } l_{e_i^{-1},e_j^{-1}}=0.
\end{equation*}
The compactification of $\mathcal{C}_V(\theta)$ in  $\p^3(\C)$ has a singular point of type $A_1$ at infinity, and admits 21 lines, of which 18 are in the affine part:
\begin{figure}[H]
\centering
\hspace{2cm}\includegraphics[scale=0.5]{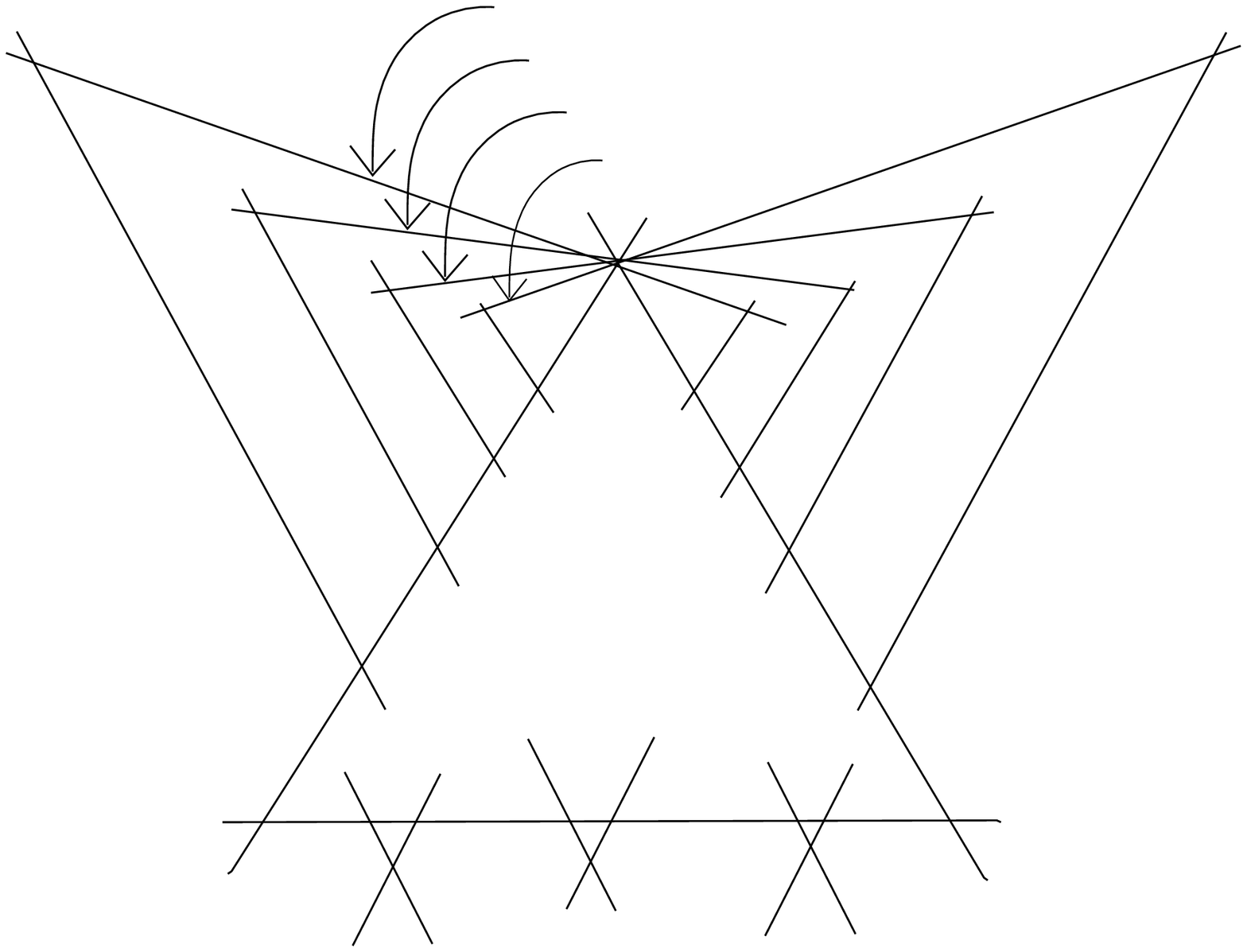}
\put(-153,198){$\Delta_{e_4}$}
\put(-148,188){$\Delta_{e_4^{-1}}$}
\put(-140,174){$\Delta_{e_3}$}
\put(-132,162){$\Delta_{e_3^{-1}}$}
\put(-220,80){$D_{e_4^{-1}}^l$}
\put(-62,75){$D_{e_3}^r$}
\put(-197,100){$D_{e_4}^l$}
\put(-82,95){$D_{e_3^{-1}}^r$}
\put(-181,112){$D_{e_3}^l$}
\put(-97,112){$D_{e_4^{-1}}^r$}
\put(-169,122){$D_{e_3^{-1}}^l$}
\put(-110,122){$D_{e_4}^r$}
\put(-210,0){$D_{e_3,e_4}$}
\put(-170,0){$D_{e_3^{-1},e_4^{-1}}$}
\put(-115,0){$D_{e_3,e_4^{-1}}$}
\put(-80,0){$D_{e_3^{-1},e_4}$}
\put(-170,50){$D_{f_0,f_0^{-1}}$}
\put(-130,50){$D_{f_0^{-1},f_0}$}
\caption{The 18 lines in $\mathcal{C}_{V}(\theta)$.}\label{droites-chiV}
\end{figure}

\begin{rema}\label{equations droites} The equations of these lines are given by
\begin{enumerate}
\item 
$z_0=-x_1^2x_3-x_1x_2+\theta_2x_1-e_0=0$ for $D_{e_3}^l\cup D_{e_3^{-1}}^l
\cup D_{e_4^{-1}}^l\cup D_{e_4}^l$;
\item
$z_1=x_1x_2-e_0=0$ for $\Delta_{e_3}\cup\Delta_{e_3^{-1}}\cup\Delta_{e_4}\cup\Delta_{e_4^{-1}}$;
\item 
$z_2=-x_2^2x_3-x_1x_2+\theta_1x_2-e_0=0$ for $D_{e_3^{-1}}^r\cup D_{e_3}^r
\cup D_{e_4}^r\cup D_{e_4^{-1}}^r$;
\item $x_3=c_{e_3,e_4}$ or $c_{e_3,e_4^{-1}}$ or $c_{e_0,e_0^{-1}}$ for the pairs of lines which cut the basis of the triangle.
\end{enumerate}
\end{rema}
In this section, we want to characterize the representations whose image is a point on a line.\footnote{We thanks Martin Klimes for discussions on this topics (see also \cite{Kl1}).}
The local loops in $\pi_1^{VI}(X,S)$ are the loops $\gamma_{i,i}$ based in $s_i$. The local loops in $\pi_1^{V}(X,S)$ are $\gamma_{3,3}$ in $s_3$, 
$\gamma_{4,4}$ in $s_4$, and a generic element of the torus $t_{1,1}(\kappa)$ in $s_1$ and $t_{2,2}(\kappa)$ in $s_2$. In any cases, the local loops at $s_i$ are denoted below by $\gamma_i$.

\begin{defi}\label{def_red locus} Let $\gamma$ be a morphism of the groupoid $\pi_1^{VI}(X,S)$ or $\pi_1^{V}(X,S)$, defined by some path from $s_i$ to $s_j$, $j\neq i$, $\gamma_i$ the local loop in $s_i$ and $\gamma_j$ the local loop in $s_i$. A linear representation $\rho$ of the groupoid in $SL_2(\C)$ is reducible along $\gamma$ if $\rho(\gamma^{-1}\gamma_{i}\gamma)$ and $\rho(\gamma_{j})$ have a common eigenvector.
\end{defi}

If $\rho$ is the linear representation associated with a regular connection $\nabla$, and $\gamma$ is represented by analytic continuation, it means that there exists some local solution at $s_i$ which is an eigenvector of the local monodromy around $s_i$ and whose analytic continuation along $\gamma$ is also an eigenvector of the local monodromy around $s_j$. This semi-local solution in a neighbourhood of $\gamma$ has an abelian monodromy.

\begin{rema}\label{critere red}
\begin{enumerate}
    \item If $\rho$ is reducible along the path $\gamma$, any equivalent representation is reducible along $\gamma$.
   \item The representation $\rho$ is reducible along $\gamma$ if and only $\rho$ is reducible along $\gamma^{-1}$.
   \item If $\rho(\gamma_{i})=\pm I$, $\rho$ is reducible along any path joining $s_i$ to another base point.
    \item Let $\rho$ be a representation in $SL_2(\C)$ given by a choice of basis of solutions $X_i$ on each $V_i$. We set $\rho(\gamma_{i})=M_i$, $\rho(\gamma)=M_\gamma$. $\rho$ is reducible on $\gamma$ if and only if the matrices $M_\gamma^{-1} M_i M_\gamma$ and $M_j$ have a common eigenvector.
    \item In particular, if $\rho$ is a normalized representation associated to the presentation given by figure (\ref{pi1VI}) of the groupoid $\pi_1^{VI}(X,S)$, with $\rho(\gamma_{i})=M_i$ and
$\rho(\gamma_{i,j})=I$ for $i\neq j$, $\rho$ is reducible on the generator path $\gamma_{i,j}$ if and only if $M_i$ and $M_{j}$ have a common eigenvector.
\end{enumerate}
\end{rema}
\begin{defi} A pair of matrices in $SL_2(\C)$ is reducible if and only if they have a common eigenvector.\end{defi}
\begin{defi}
\begin{enumerate}
\item The reducibility locus associated to a path $\gamma$ is the set $\mathcal{R}(\gamma)$ of linear representations $\rho$ in $\chi_{VI}$ which are reducible on $\gamma$.
\item Given a presentation of $\pi_1^{VI}(X,S)$, the generating reducibility locus is the union of the sets $\mathcal{R}(\gamma_{i,j})$ for all the generators $\gamma_{i,j}$, $i\neq j$.
\item The total reducibility locus is the set $\mathcal{R}$ of linear representations $\rho$ in $\chi_{VI}$ such that there exists a path between two distinct objects on which $\rho$ is reducible.
\end{enumerate}
\end{defi}

\begin{theorem} \label{reducibility locusVI} Let $\{\gamma_{i,j}),\mathcal{R}_i\}$ be the presentation of $\pi_1^{VI}(X,S)$ given by the figure (\ref{pi1VI}), and $Tr_{VI}:\ \chi_{VI}(a)\rightarrow \mathcal{C}_{VI}(\theta)$ the trace map associated to this presentation. The 24 lines in $\mathcal{C}_{VI}(\theta)$ are the reducibility locus of the 6 paths: $\gamma_{i,i+1}$, $i=1,..,4$ and $\gamma_{1,3}=\gamma_{1,2}\cdot\gamma_{2,3}$, $\gamma_{2,4}=\gamma_{2,3}\cdot\gamma_{3,4}$.
\end{theorem}

\begin{lemma}\label{paire reductible}
 Let $M_1$ and $M_2$ be two matrices in $SL_2(\C)$ distinct from $\pm I$, $e_1$, $e_1^{-1}$, and $e_2$, $e_2^{-1}$ their eigenvalues (we may have $e_1=e_1^{-1}$ or $e_2=e_2^{-1}$). Let
 \[c_{\alpha,\beta}=\alpha\beta^{-1}+\alpha^{-1}\beta.\]
 The pair $(M_1,M_2)$ is reducible if and only if
\[tr(M_1M_2)=c_{e_1,e_2},\mbox{ or }tr(M_1M_2)=c_{e_1,e_2^{-1}}.\]

\end{lemma}
\begin{proof}
We choose a "mixed basis" taking an eigenvector of $M_2$, and an eigenvector of $M_1$. In such a basis, we have:
\begin{equation*}\label{base mixte} M_1=\left(\begin{array}{cc}e_1&0\\f_1&e_1^{-1}\end{array}\right),\ \
M_2=\left(\begin{array}{cc}e_2&f_2\\0&e_2^{-1}\end{array}\right)
\end{equation*}
or any similar writing obtained by changing $e_1$ with $e_1^{-1}$ or $e_2$ with $e_2^{-1}$. Let $D_i=diag(e_i,e_i^{-1})$. We have
\[tr(M_1M_2)=tr(D_1D_2)+f_1f_2.\]
Such a pair has a common eigenvector if and only if $f_1=0$ or $f_2=0$, i.e. if and only if $tr(M_1M_2)=tr(D_1D_2)$. We have
$tr(D_1D_2)=c_{e_1,e_2}$ or $tr(D_1D_2)=c_{e_1,e_2^{-1}}$. 
\end{proof}
\bigskip

\emph{Proof of Theorem (\ref{reducibility locusVI}).}
We consider a normalized representation $\rho$ and the six (non oriented) generating paths $\gamma_{i,j}$, $i\neq j$. For a normalized representation $\rho$, these paths satisfy: $\rho(\gamma_{i,j})=I$. Therefore $\rho$ is reducible over $\gamma_{i,j}$ if and only if the pair of matrices $(M_i,M_j)$ is reducible. For $\{i,j\}$ in \{1,2,3\}, according to Lemma (\ref{paire reductible}), $\rho$ is reducible on the path $\gamma_{i,j}$ if and only if
$x_k=c_{e_i,e_j}$ or $x_k=c_{e_i,e_j^{-1}}$.
Therefore the reducibility locus $\mathcal{R}({\gamma_{i,j}})$ is given by the four lines $L_{e_i^{\pm 1},e_j^{\pm 1}}$. \carre

\begin{defi} The 12 Kaneko points are the intersections
\[p_{e_i,e_j}=L_{e_i,e_j}\cap L_{e_i^{-1},e_j^{-1}},\ p_{e_i,e_j^{-1}}=L_{e_i,e_j^{-1}}\cap L_{e_i^{-1},e_j},\ \{i,j\}\subset\{1,2,3,4\}.\]
\end{defi}
\begin{prop}\label{monodromie commmutante} The Kaneko points $p_{e_i,e_j}$ and $p_{e_i,e_j^{-1}}$ correspond to a normalized monodromy representation such that  $M_i$ and $M_j$ commute.
\end{prop}
\begin{proof} In a "mixed basis",
\begin{equation}\label{base mixte} M_i=\left(\begin{array}{cc}e_i&0\\f_i&e_i^{-1}\end{array}\right),\ \
M_j=\left(\begin{array}{cc}e_j&f_j\\0&e_j^{-1}\end{array}\right)
\end{equation}
or the same writing changing $e_j$ in $e_j^{-1}$.
The reducibility locus is given by $f_if_j=0$, which defines two components $L_{e_i,e_j}$ and $L_{e_i^{-1},e_j^{-1}}$ (or $L_{e_i,e_j^{-1}}$ and $L_{e_i^{-1},e_j}$). Therefore $p_{e_i,e_j}$ (or $p_{e_i,e_j^{-1}}$) is defined by $f_i=0$ and $f_j=0$.\end{proof}

\bigskip

The Painlevé VI foliation admits 12 singular points, 4 over each singular fiber. The local description of the foliation in a neighborhood of these non linear singularities of ``regular type'' can be found in the chapter 4 of \cite{IKSY}. The Painlevé V foliation also admits 3 non linear regular singular points over 0. The germ of Painlevé equation admits a unique meromorphic solution around such a singular point, defining an analytic leave for the germ of foliation: we call them the \textit{central solutions}. These solutions have been studied by K. Kaneko \cite{Ka4}. Furthermore, they appears as the intersection of two codim 1 germs of invariant analytic surfaces \cite{IKSY}.

\begin{theorem}\label{correspondence}
\begin{enumerate}
  \item The Riemann-Hilbert correspondence $RH_{VI}$ sends the 12 central solutions to the 12 Kaneko points.
  \item The Riemann-Hilbert correspondence $RH_{VI}$ sends the local invariant varieties near a singular point to the 2 germs of lines intersecting at the corresponding Kaneko point.
\end{enumerate}
\end{theorem}

\begin{proof}
We consider the 4 singular points over 0. The non linear monodromy generated by a simple loop around 0 will keep invariant each meromorphic solutions and the analytic invariant surfaces which intersect on it, around this singular point.
Through the Riemann-Hilbert correspondence this non linear monodromy is conjugated to one of the polynomial dynamics $h_{i,j}$ given by (\ref{dynamicPVI}). The central solutions are sent on fixed points of $h_{i,j}$.
Since each $h_{i,j}$ fixes 4 central points, and has no more than 4 fixed points, this proves the first point.

Now we search for invariant curves under the action of $h_{i,j}$ through the central point $p_{e_i,e_j}$. Clearly the two lines
$L_{e_i,e_j}$ and $L_{e_i^{-1},e_j^{-1}}$ are invariant since the union of the two lines is $(x_k=c_{e_i,e_j})\cap(F_{VI}(x,\theta)=0)$ which is invariant. Suppose that there exists another local analytic invariant curve transverse to $dx_k=0$. It cut the plane $(x_k=c)$ on a finite set of points, for any value $c$ near from $c_{e_i,e_j}$. This would create a periodic orbit on each $(x_k=c)$. The restriction of $h_{i,j}$ on each $(x_k=c)$ is a family of affine maps, and for a generic affine map, there doesn't exist such a periodic orbit. Therefore the local invariant surfaces are sent on the germ of two lines around each central point.
\end{proof}

\begin{rema} Using Proposition (\ref{monodromie commmutante}) and Theorem (\ref{correspondence}), we recover here a result of K. Kaneko \cite{Ka4}: the linear monodromy data preserved along a Kaneko solution is generated by 3 matrices with a pair of commuting matrices. In particular, it generates a solvable group.
Nevertheless, the theorem of K. Kaneko is much more precise: it describes \emph{explicitely} this linear solvable monodromy data.
\end{rema}

\noindent We have a result similar to Theorem (\ref{reducibility locusVI}) for the Painlevé V foliation.

\begin{theorem} \label{reducibility locusV} Let $\{\gamma_{i,j},\mathcal{R}\}$ be the presentation of $\pi_1^V(X,S)$ given by figure (\ref{pi1V}), and $Tr_V^+:\ \chi_{V}^+(a)\rightarrow \mathcal{C}_V(\theta^+)$ the trace map associated to this presentation.
Any line in $\mathcal{C}_V(\theta^+)$ is the image of a reducibility locus in $\chi_{V}(a)$ for some path in $\pi_1^V(X,S)$.

\end{theorem}
The proof is similar to the one of Theorem (\ref{reducibility locusVI}, but one needs to investigate different cases. We just mention here an example:
\begin{lemma}
The reducibility locus $\mathcal{R}(\gamma_{1,2}^r)$ is given by a pair of lines defined by $(x_3=a_0)\cap\mathcal{C}_{V}$. 
\end{lemma}

\textit{Proof.} For a normalized representation $\rho$ we have
\[\rho(\gamma_{1,2}^r)=U_1M_0U_2,\ \rho(\widehat\gamma_{1,1})=\rho(\widehat\gamma_{2,2})=M_0.\]
Therefore, according to Remark (\ref{critere red}), we have:
\begin{equation*}
\begin{aligned}
\rho\in\mathcal{R}(\gamma_{1,2}^r)&\Leftrightarrow ((U_1M_0U_2)^{-1}M_0(U_1M_0U_2),M_0) \mbox{ is a reducible pair}\\
&\Leftrightarrow (M_0^{-1}U_1^{-1}M_0U_1,U_2M_0U_2^{-1}M_0^{-1}) \mbox{ is a reducible pair}\\
&\Leftrightarrow u_1u_2=0\\
&\Leftrightarrow tr(U_1M_0U_2)=tr(M_0),\\
&\Leftrightarrow x_3=a_0=f_0^2+f_0^{-2}\mbox{ where }f_0^2=e_0.
\end{aligned}
\end{equation*}
The intersection $(x_3=a_0)\cap \mathcal{C}_{V}$ is given by one pair of lines. Indeed we have:  
\[F_V(x,\theta)=(x_3-e_0-e_0^{-1})(x_1x_2-e_0) +d_{f_0,f_0^{-1}}d_{f_0^{-1},f_0},\]
with
\[d_{f_0,f_0^{-1}}=f_0x_1+f_0^{-1}x_2-f_0a_3,\ d_{f_0^{-1},f_0}=f_0^{-1}x_1+f_0x_2-f_0a_4.\]
Therefore the reducibility locus $\mathcal{R}(\gamma_{1,2}^r)$ is given by the pair of lines:
\begin{equation*}
\begin{aligned}
D_{f_0,f_0^{-1}}:&\ x_3=e_0+e_0^{-1} \mbox{ and } d_{f_0,f_0^{-1}}=0,\ \
D_{f_0^{-1},f_0}:&\ x_3=e_0+e_0^{-1} \mbox{ and } d_{f_0^{-1},f_0}=0.\ \carre
\end{aligned}
\end{equation*} 
The others cases are treated with similar computations. Here we also have a correspondence between the 3 Kaneko solutions described by
K. Kaneko and Y. Ohyama in  and \cite{KaOh}, and the 3 Kaneko points $p_{e_3,e_4}$, $p_{e_3,e_4^{-1}}$ and $p_{e_0,e_0^{-1}}$ through $RH_V$.

The complete reducibility locus, defined by all the paths $\gamma$ of the fondamental groupoid, contains these lines and their images by the tame dynamics. We don't know if the lines and the tame action generate the complete reducibility locus (maybe we have to add some parabolas).

\subsection{Log-canonical coordinates on $\mathcal{C}_V(\theta)$ and cluster sequences}\label{logcan}
Let $\mathcal{C}_V(\theta)$ be the family of symplectic cubic affine surfaces in $\C^3$ defined by $F_V(x,\theta)=0$.
The cubic surfaces $\mathcal{C}_V(\theta)$ are birational to $\C^2$. Indeed, since the equation $F_V(\theta)$ is affine in $x_3$, the map $\pi_3:\ \mathcal{C}_V(\theta)\rightarrow \C^2$ induced by
$(x_1,x_2,x_3)\rightarrow (x_1,x_2)$
is rationally invertible and therefore birational. The polar locus of $\pi_3^{-1}$ is given by $F_{V,x_3}=x_1x_2-e_0=0$. Recall that $\mathcal{C}_V(\theta)$ has a symplectic structure defined by
\[\omega_V(\theta)=\frac{dx_i\wedge dx_j}{\partial F_{V,\theta}/\partial x_k},\]
for any ${i,j,k}={1,2,3}$. We want here to introduce \textit{symplectic} birational maps:
\bigskip\\
\textbf{Notations} (see Appendix):
\begin{description}
  \item[-] $\omega_{log}=\frac{du}{u}\wedge\frac{dv}{v}$ : the log-canonical symplectic form on $\C^2$;
  \item[-] $Bir(\C^2)$ : the group of the birational automorphisms of the complex plane;
  \item[-] $Symp(\C^2)=Symp^+(\C^2)$ : the subgroup of $Bir(\C^2)$ of the automorphisms which preserve $\omega_{log}$;
  \item[-] $Symp^-(\C^2)$ : the subset of $Bir(\C^2)$ of the automorphisms $\varphi$ such that
  $\varphi^*\omega_{log}=-\omega_{log}$;
  \item[-] $Symp^\pm(\C^2)$ : the subgroup of $Bir(\C^2)$ generated by $Symp^+(\C^2)\cup Symp^-(\C^2)$.
\end{description}
\begin{defi}
\begin{enumerate}
  \item A log-canonical system of coordinates on $\mathcal{C}_{V}(\theta)$ is a birational symplectic morphism from
  $\left(\mathcal{C}_{V}(\theta),\omega_{V}(\theta)\right)$ to $(\C^2,\omega_{log})$.
  \item A log-canonical function on $\mathcal{C}_{V}(\theta)$ is a component of some log-canonical system.
  \item A log-canonical sequence of coordinates is a sequence of coordinates such that two consecutive elements define a log-canonical system of coordinates.
  \item A log-anti-canonical system of coordinates on $\mathcal{C}_{V}(\theta)$ is a birational symplectic morphism from
  $\left(\mathcal{C}_{V}(\theta),\omega_{V}(\theta)\right)$ to $(\C^2,-\omega_{log})$.
\end{enumerate}
\end{defi}
\begin{rema}
\begin{enumerate}
\item
If $(x,y)$ is a log-anti-canonical system of coordinates, $(y,x)$ and $(x^{-1},y)$ are log-canonical systems of coordinates.
\item If $(x,y)$ is a log-canonical system of coordinates, for any $\lambda$, $\mu$ in $\C^*$, $(\lambda x,\mu y)$ is a log-canonical system of coordinates.
\end{enumerate}
\end{rema}
The cubic surface $\mathcal{C}_{V}(\theta)$ is symplectic birationally equivalent to $(\C^2,\omega_{log})$. Indeed
we have two first examples of log-canonical systems of coordinates on $\mathcal{C}_{V}(\theta)$:
\begin{prop}\label{z1} Let $y_1=x_1$, $y_2=x_2$ and $z_1=F_{V,x_3}=x_1x_2-e_0$. The pairs
$(y_1,z_1)$, and $(z_1,y_2)$
are log-canonical systems of coordinates on $\mathcal{C}_{V}(\theta)$. A map $z_1$ such that both  $(y_1,z_1)$, and $(z_1,y_2)$
are log-canonical systems of coordinates is unique up to a multiplicative constant.
\end{prop}

\begin{proof}
We have:
\[\frac{dy_1}{y_1}\wedge\frac{dz_1}{z_1}=\frac{dx_1}{x_1}\wedge\frac{x_1dx_2+x_2dx_1}{z_1}=\frac{dx_1\wedge dx_2}{x_1x_2-e_0}=\omega_V(\theta),\]
and we have a similar computation for $(z_2,y_2)$. All the maps such that $(y_1,z)$
is a log-canonical system of coordinates write $z=c(y_1)z_1$. Therefore the only maps such that $(y_1,z)$, and $(z,y_2)$
are log-canonical systems of coordinates are $z=cz_1$ for some $c$ in $\C^*$.
\end{proof}
\bigskip\\
Therefore $(y_1,z_1,y_2)$ is a log-canonical triple, which satisfies: $z_1=y_1y_2-e_0$.
In order to construct new log-canonical or log-anti-canonical systems of coordinates we consider the two maps:
\begin{equation*}
\begin{aligned}
&\sigma_1(x_1,x_2,x_3)=(x_1-F_{V,x_1},x_2,x_3)=(-x_1-x_2x_3+\theta_1,x_2,x_3)\\
&\sigma_2(x_1,x_2,x_3)=(x_1,x_2-F_{V,x_2},x_3)=(x_1,-x_2-x_1x_3+\theta_2,x_3).
\end{aligned}
\end{equation*}
\begin{lemma}\label{sigma}
$\sigma_1$ and $\sigma_2$ define polynomial involutive automorphims of $\mathcal{C}_V(\theta)$. Furthermore they are anti-symplectic automorphisms, i.e. $\sigma_i^*\omega_V(\theta)=-\omega_V(\theta)$.
\end{lemma}
\begin{proof} The involutive property is a direct computation. We have $F_V\circ\sigma_i=F_V$, and therefore the $\sigma_i$'s are polynomial automorphims. Finally, by using
$\omega_V(\theta)= \frac{dx_2\wedge dx_3}{x_2x_3+2x_1-\theta_1}$ (resp. $\omega_V(\theta)= \frac{dx_3\wedge dx_1}{x_3x_1+2x_2-\theta_2}$), 
we obtain $\sigma_1^*\omega_V(\theta)=-\omega_V(\theta)$ (resp. $\sigma_2^*\omega_V(\theta)=-\omega_V(\theta)$).
\end{proof}
\bigskip\\
We set:
\[g=\sigma_1\circ\sigma_2,\ \  g^{-1}=\sigma_2\circ\sigma_1.\]
From Lemma (\ref{sigma}), $g$ is a polynomial symplectic automorphisms of $(\mathcal{C}_V(\theta),\omega_V(\theta))$.
Therefore, starting from the log-canonical triple $(y_1,z_1,y_2)$, we obtain two other log-canonical triples by applying $\sigma_i^*$, $i=1,2$ and reversing the order of the triple. We set:
\begin{equation*}
\begin{aligned}
& (y_2,z_2,y_3):=(\sigma_1^*y_2,\sigma_1^*z_1,\sigma_1^*y_1);\\
& (y_0,z_0,y_1):=(\sigma_2^*y_2,\sigma_2^*z_1,\sigma_2^*y_1).
\end{aligned}
\end{equation*}
Therefore we obtain a log-canonical sequence of length seven:
\[\mathcal{H}:\ (y_0,z_0,y_1,z_1,y_2,z_2,y_3).\]
We call it the fundamental log-canonical heptuple. Note that we have:
\[(y_2,z_2,y_3)=g^{-1 *}(y_0,z_0,y_1).\]
We extend the fundamental log-canonical sequence to an infinite log-canonical sequence by setting for any $k$ in \Z:
\[(y_{2k},z_{2k},y_{2k+1},z_{2k+1},y_{2k+2},z_{2k+3},y_{2k+3}):=(g^{-k*})(y_0,z_0,y_1,z_1,y_2,z_2,y_3).\] 
This sequence satisfies the following relations:
\begin{prop}\label{exchange relations}[exchange relations] let $A_V=\Z[e_0^\pm,e_3^\pm,e_4^\pm]$. Let $P$, $Q_1$ and $Q_2$ in
$A_V[t]$ defined by
\begin{equation*}
\begin{aligned}
&P(t)=t+e_0,\\
&Q_1(t)=(t-e_0e_4^{-1})(t-e_0e_4)(t-e_3^{-1})(t-e_3),\\
&Q_2(t)=(t-e_0e_3^{-1})(t-e_0e_3)(t-e_4^{-1})(t-e_4).
\end{aligned}
\end{equation*}
For all $k$ in $\Z$ we have
\begin{equation*}
\begin{aligned}
&y_ky_{k+1}=P(z_k)\\
&z_{2k}z_{2k+1}=Q_1(y_{2k+1})\\
&z_{2k+1}z_{2k+2}=Q_2(y_{2k+2}).
\end{aligned}
\end{equation*}
\end{prop}

\begin{lemma}
\label{lemmasigmaq}
We have~:
\begin{equation}
\label{equasigmaq}
\sigma_1^*x_1=x_2^{-1}\left(e_0+z_1^{-1}Q_2(x_2)\right), ~~~ \text{and} ~~ \sigma_2^*x_2=x_1^{-1}\left(e_0+z_1^{-1}Q_1(x_1)\right).
\end{equation}
with~:
\begin{equation}
\label{equaq2}
Q_2(t)=t^4-\theta_2t^3+\theta_4t^2-e_0\theta_1t+e_0^2=(x-e_0e_3^{-1})(x-e_0e_3)(x-e_4^{-1})(x-e_4),
\end{equation}
\begin{equation}
\label{equaq1}
Q_1(t)=t^4-\theta_1t^3+\theta_4t^2-e_0\theta_2t+e_0^2=(x-e_0e_4^{-1})(x-e_0e_4)(x-e_3^{-1})(x-e_3).
\end{equation}
\end{lemma}

\textsl{Proof.}
 \begin{enumerate}
\item 
We have~:
\begin{align*}
x_2\sigma_1^*x_1 &=-x_1x_2-x_2^2x_3+\theta_1x_2  \\
&=-e_0-z_1+x_2^2z_1^{-1}(x_1^2+x_2^2-\theta_1x_1-\theta_2x_2+\theta_4)+\theta_1x_2 \\
&=-e_0-z_1+z_1^{-1}\left((e_0+z_1)^2+x_2^4-\theta_1x_2(e_0+z_1)-\theta_2x_2^3+\theta_4x_2^2\right)+\theta_1x_2 \\
&=e_0+z_1^{-1}\left(e_0^2+x_2^4-\theta_1x_2-\theta_2x_2^3 +\theta_4x_2^2\right)=e_0+z_1^{-1}Q_2(x_2).
\end{align*}
Hence $\sigma_1^*x_1=e_0x_2^{-1}+x_2^{-1}z_1^{-1}Q_2(x_2)$.

The proof of the other equality is similar.
\item
If we develop the right term of (\ref{equaq2}), we get~:
\begin{align*}
& t^4-(a_4+e_0a_3)t^3+(e_0a_3a_4+e_0^2+1)t^2-e_0(a_3+e_0a_4)t+e_0^2 \\
&= t^4-\theta_2t^3+\theta_4t^2-e_0\theta_1t+e_0^2.
\end{align*}
If we develop the right term of (\ref{equaq1}), we get~:
\begin{align*}
& t^4-(a_3+e_0a_4)t^3+(e_0a_3a_4+e_0^2+1)t^2-e_0(a_4+e_0a_3)t+e_0^2  \\
& =t^4-\theta_1t^3+\theta_4t^2-e_0\theta_2t+e_0^2.\ \carre
\end{align*}
\end{enumerate}

\textsl{Proof of Proposition (\ref{exchange relations}).}
From lemma \ref{lemmasigmaq}, we get~:
$z_2=x_2\sigma_1^*x_1-e_0=z_1^{-1}Q_2(x_2)$. Hence
$z_1z_2=Q_2(x_2)$.
Similarly~:
$z_0=\sigma_2^*x_2-e_0=z_1^{-1}Q_1(x_1)$. Hence
$z_0z_1=Q_1(x_1)$.
The others relations are obtained by applying $g^{-k*}$ to the previous one.
\carre

\begin{rema}\label{Q1Q2}
\begin{enumerate}
	\item We have $Q_1=e_0^{-2}t^4Q_2(e_0t^{-1})$. 
	\item According to Remark (\ref{equations droites}), $z_{2k}=0$ are the equations of the lines $D_{\cdot}^r$, $D_{\cdot}^l$, or their images by the dynamic $<g>$, and $z_{2k+1}=0$ are the equations of the lines $\Delta_{\cdot}$ or their images by the dynamic $<g>$.
\end{enumerate}
\end{rema}

\subsection{The Laurent property}

The exchange relations obtained at Proposition (\ref{exchange relations}) suggests that the sequence of log-canonical systems of coordinates is related to a structure of generalized cluster algebra. Without going into the theory of cluster algebras (for details see \cite{Kell}), we just recall here the fact that under some hypothesis they satisfy the ``Laurent property'' (see also \cite{ChMR}):

\begin{defi} 
\begin{enumerate}
	\item A rational map $r\in\C(x,y)$ satifies the Laurent property if its polar set is included in $xy=0$. 
	\item A birational map $r$ satifies the Laurent property if both $r$ and $r^{-1}$ have the Laurent property.
	\item Let $X$ be an affine surface, and let $(y_n,z_n)$ be a sequence of algebraic morphisms from $X$ to $\C^2$. This sequence satisfies the Laurent property, if given an element $(y_n,z_n)$, any other regular function $y_m$ (or $z_m$) $=r(x_n,y_n)$ satisfies the Laurent property. 
\end{enumerate}
\end{defi}

We must remark that the Laurent property is not stable by composition or inversion. It turns out that in a cluster sequence some simplications arise from the exchange relations and give this property. 
A direct proof of this fact needs heavy computations.
We can prove here the Laurent property for the canonical sequence by an argument which avoid these computations with the exchange relations.
\begin{prop}\label{Laurent_property} The log-canonical sequence satisfies the Laurent property.
\end{prop}

\begin{proof} We first prove that all the log-canonical variables are polynomials in $(x_1,x_2,x_3)$. This is true for $(y_1,z_1,y_2)=(x_1,x_1x_2-e_0,x_2)$. Recall that $\sigma_1$, $\sigma_2$ (and therefore $g$) are polynomials in $(x_1,x_2,x_3)$. All the other canonical variables are obtained by the action of a word in $\sigma_1^*$, $\sigma_2^*$ from a variable in this triple. Therefore they are still polynomials in $(x_1,x_2,x_3)$. We set:
\begin{equation*}
\begin{aligned}
&p_n:\ \mathcal{C}_{V}(\theta)\rightarrow\C^2, u=y_n(x), v=z_n(x),\\
&q_n:\ \mathcal{C}_{V}(\theta)\rightarrow\C^2, u=z_n(x), v=y_{n+1}(x).
\end{aligned}
\end{equation*}                                                           
We claim that $p_n$ and $q_n$ satisfy the Laurent property. Since $p_n$ is polynomial, we only have to prove that the polar set of $p_n^{-1}$ is included in $uv=0$. This is true for $p_1=(y_1,z_1)$:
\begin{equation*}
\begin{aligned}
&x_1=y_1\\
&x_2=(z_1+e_0)y_1^{-1}\\
&x_3=z_1^{-1}(-y_1^2-(z_1+e_0)y_1^{-2}+\theta_1y_1+\theta_2(z_1+e_0)y_1^{-1}-\theta_4.
\end{aligned}
\end{equation*}
This is also true for $q_1=(z_1,y_2)$ by using a similar computation. Let
\[\iota:\ (u,v)\rightarrow (v,u).\]
We have :
\[q_0=\iota\circ p_1\circ \sigma_2,\ p_2=\iota\circ q_1\circ \sigma_1.\]
Therefore, $q_0$ and $p_2$ also satisfy the Laurent property. Since the property is true for $p_1$ and $p_2$, and since the tame dynamic $g$ is a polynomial automorphism, it remains true for any $p_n=g^{-1*}p_{n-2}$. Since the property is true for $q_0$ and $q_1$ it remains true for any $q_n=g^{-1*}q_{n-2}$.
Finally, since $p_n$ satisfies the Laurent property and $p_m$ is polynomial, the map
\[p_m\circ p_n^{-1}:\ (y_m,z_m)=(r(y_n,z_n),s(y_m,z_m))\]
satisfies the Laurent property.\end{proof}\bigskip\\
From the above proof, one can remark that the Laurent property comes from the fact that the antisymplectic tame dynamics ($\sigma_1$ and $\sigma_2$) and the symplectic dynamics ($g$) are automorphisms of the cubic surface.

\subsection{Families of confluent and diffluent morphisms}

These families were first discovered by M. Klimes in \cite{Kl1} through an analytic confluent process, both on Painlevé equations and on character varieties. We recover them by using a groupoid point of view.
In definition (\ref{morphism_conf}) we have introduced a family of morphisms $\varphi_\kappa:\ \pi_1^{VI}(X,S)\rightarrow\pi_1^{V,\kappa}(X,S)$.  Therefore we obtain a family 
\begin{equation*}
\begin{aligned}
\varphi_\kappa^*:\ &\chi_V^\kappa(a)\rightarrow \chi_{VI}(a_\kappa)\\
&\rho\rightarrow \rho\circ\varphi_\kappa.
\end{aligned}
\end{equation*}
On the restriction over $\chi_V^\pm$ we have 2 families $\varphi_\kappa^{\pm *}$.

\begin{prop}\label{Invertible} 
The morphisms $\varphi_\kappa^{\pm *}$ are invertible on the open set in $\chi_{VI}(a_\kappa)$ defined by the elements $(M_1,M_2,M_3,M_4)$ such that $(M_1,M_2)$ is a non reducible pair.
\end{prop}

\begin{proof}
The matrix expression of $\varphi_\kappa^{+*}$ is
\[[U_1,M_0,U_2,M_3,M_4]_{D}\rightarrow 
[U_1D_\kappa, D_{\kappa^{-1}}M_0U_2,M_3,M_4)]_{SL_2(\C)}.\]
Given $(M_1,M_2,M_3,M_4)$, and a parameter $\kappa$, we search for matrices $M_0$, $U_1$, $U_2$, $M'_3$, $M'_4$, $(U_1,M_0,U_2)$ in $U^-\times D\times U^+$ and a matrix $Q$ in $SL_2(\C)$ such that
\begin{equation}\label{preimage}
\left\{
\begin{aligned}
&U_1D_\kappa=Q^{-1}M_{1}Q\\
&D_\kappa^{-1}M_0U_2=Q^{-1}M_{2}Q\\
&M'_3=Q^{-1}M_{3}Q\\
&M'_4=Q^{-1}M_{4}Q
\end{aligned}\right.
\end{equation}
For a given local data $(a_1,a_2,a_3,a_4)$, and a value $\kappa$, we have only two choices for the eigenvalues $(e_0,e_0^{-1})$ of $M_0$, solutions of
$\kappa^{-1} e_0+\kappa e_0^{-1}=a_2$, one centered around $a_2\kappa$ and the other around $a_2\kappa^{-1}$, induced by the change $\kappa\rightarrow \kappa^{-1}$. Let $e_1^\pm1$, $e_2^\pm$ be the eigenvalues of $M_1$ and $M_2$. The two solutions for $M_0$ are:
$M_0=diag(e_{1}e_{2},e_{1}^{-1}e_2^{-1})$, or $M_0=diag(e_{1}^{-1}e_{2}^{-1},e_{1}e_{2})$. We choose the first one (the second will give a pre-image for $\varphi_\kappa^{-*}$).

We recall the LDU decomposition in $SL_2(\C)$:
\begin{lemma}\label{LDU}
Let $M$ be an element of $SL_2(\C)$ such that $M[1,1]\neq 0$. There exists a unique triple $(L,D,U)$ in $U^-\times D\times U^+$ such that $M=L\times  D\times U$.
\end{lemma}
\begin{proof} We have
\begin{equation*}
\left(\begin{array}{cc}a & b\\ c & d \end{array}\right)=
\left(\begin{array}{cc}1 & 0\\ l & 1 \end{array}\right)\times
\left(\begin{array}{cc}e & 0\\ 0 & e^{-1} \end{array}\right)\times
\left(\begin{array}{cc}1 & u\\ 0 & 1 \end{array}\right)
\Leftrightarrow
\left\{\begin{aligned}
&a=e\\&b=eu\\&c=el\\&d=e^{-1}+elu
\end{aligned}\right.
\end{equation*}
which system has the unique solution $(l,e,u)=(c/a,a,b/a)$ if $a\neq 0$. \end{proof}\bigskip\\
We call the above decomposition the $LDU$-decomposition of $M$ and we denote:
\[LDU(M)=(L,D,U).\]
\begin{lemma}\label{central lemma} Let $M_1$ and $M_2$ be two non vanishing matrices in $SL_2(\C)$, $M_1\neq\pm I$, $M_2\neq\pm I$, with eigenvalues $(e_1,e_1^{-1})$ and $(e_2,e_2^{-1})$. Suppose that the eigenvectors related to $e_2^{-1}$ and to $e_1$ are independent. There exists a matrix $Q$ in $SL_2(\C)$ such that the diagonal component of the decomposition $LDU$ of $Q^{-1}M_1M_2Q$ is $diag(e_1e_2,e_1^{-1}e_2^{-1})$.
\end{lemma}
\begin{proof}
The hypothesis of the Lemma gives the existence of a "mixed basis" $(u,v)$ given by an eigenvector $u$ of $M_2$ related to the eigenvalue $e_2^{-1}$, and an eigenvector $v$ of $M_1$ related to the eigenvalue $e_1$. Let $Q$ be the matrix of this change of basis. We have:
\[Q^{-1}M_1Q=\left(\begin{array}{cc}e_1 & 0\\ c_1 & e_1^{-1} \end{array}\right),\ Q^{-1}M_2Q=\left(\begin{array}{cc}e_2 & b_2\\ 0 & e_2^{-1} \end{array}\right).\]
Therefore
\[Q^{-1}M_1M_2Q=\left(\begin{array}{cc}e_1e_2 & e_1b_2\\ e_2c_1 & e_1^{-1}e_2^{-1}+b_2c_1 \end{array}\right)=
\left(\begin{array}{cc}1 & 0\\ c_1 & 1 \end{array}\right)\cdot
\left(\begin{array}{cc}e_1e_2 & 0\\ 0 & e_1^{-1}e_2^{-1}\end{array}\right)\cdot
\left(\begin{array}{cc}1 & b_2\\ 0 & 1 \end{array}\right)
.\]
\end{proof}
\bigskip\\
\emph{End of the proof of Proposition (\ref{Invertible}).}
Suppose that the pair of matrices $(M_{1,\kappa},M_{2,\kappa})$ satisfies the hypothesis of Lemma (\ref{central lemma}). We choose the matrix $Q$ given by this Lemma and we obtain a solution of (\ref{preimage}):
\[(U_1,M_0,U_2)=LDU(Q^{-1}M_{1,\kappa}M_{2,\kappa}Q), M'_3= Q^{-1}M_3Q, M'_4=Q^{-1}M_4Q.\]
A change of mixed basis will modify $Q$ by a multiplication on the righta side with a diagonal matrix. Therefore the class $[(U_1,M_0,U_2,M'_3,M'_4)]_D$ is unique.
The representations $\rho$ for which the hypothesis of Lemma (\ref{central lemma})
is not satisfied are the representations for which the eigen-directions related to $e_{1,\kappa}$ and $e_{2,\kappa}^{-1}$ are the same. This defines a component --here the line $L_{e_1,e_2}$-- of the reducibility locus along the path $\gamma_{1,2}$ in $\pi_1^{VI}(X,S)$. Therefore $\Phi_\kappa$ is invertible outside $L_{e_1,e_2}$.
Note that for the other choices of $\kappa$ or $e_0$ 
we shall find one of the other components of $\mathcal{R}(\gamma_{1,2})$.\ \end{proof}

\bigskip

Using the trace maps, $\varphi_\kappa^{\pm *}$ induces two families of morphisms of cubic surfaces:
\[\Phi_\kappa^\pm=(Tr_V^{\pm\kappa})^{-1}\circ \varphi_\kappa^{\pm *}\circ  Tr_{VI}:\  \mathcal{C}_V(\theta^\pm)\rightarrow \mathcal{C}_{VI}(\kappa(\theta^\pm))\]
where $\kappa(\theta^+)=\kappa(e_0,a_3,a_4)=(\kappa+\kappa^{-1},e_0\kappa^{-1}+e_0^{-1}\kappa,a_3,a_4)$ and $\kappa(a^-)=\kappa(e_0^{-1},a_3,a_4)=(\kappa+\kappa^{-1},e_0^{-1}\kappa^{-1}+e_0\kappa,a_3,a_4)$.

\begin{theorem}[Confluent and diffluent morphisms]\label{birational map} The morphisms $\Phi_\kappa^\pm$ are symplectic birational morphisms from $\mathcal{C}_V(\theta^\pm)$ to $\mathcal{C}_{VI}(\kappa(\theta^\pm)$.  
\end{theorem}

\textit{Proof.} We have to prove that $\Phi_\kappa^+$ has rational expression in trace coordinates and is invertible in the class of birational transformations.

\begin{lemma}\label{VI-V} The map $\Phi_\kappa^+:\ \mathcal{C}_V(\theta^+)$ to $\mathcal{C}_{VI}(\kappa(\theta^+)$  is defined by
\begin{equation*}
  \left\{
    \begin{aligned}
&x_{1,\kappa}=e_0^{-1}\kappa x_1+\kappa^{-1}x_2\\
&x_{2,\kappa}=-e_0^{-1}\kappa x_1x_3 +\kappa^{-1}x_1-e_0^{-1}\kappa x_2+a_3\kappa+a_4e_0^{-1}\kappa\\
&x_{3,\kappa}= x_3.
    \end{aligned}
  \right.
\end{equation*}
\end{lemma}

\begin{proof}
We set $M_{1,\kappa}=U_1D_\kappa$, $M_{2,\kappa}=D_{\kappa^{-1}}M_0U_2$. We have
\begin{equation*}
  \left\{
    \begin{aligned}
&x_{1,\kappa}=tr(M_{2,\kappa}M_3)=\alpha_3e_0\kappa^{-1}+e_0\gamma_3s_2\kappa^{-1}+\delta_3e_0^{-1}\kappa\\
&x_{2,\kappa}=tr(M_3M_{1,\kappa})=\alpha_3\kappa+\beta_3s_1\kappa+\delta_3\kappa^{-1}\\
&x_{3,\kappa}=tr(M_{1,\kappa}M_{2,\kappa})=e_0s_1s_2+e_0+e_0^{-1}.
    \end{aligned}
  \right.
\end{equation*}
Using Lemma (\ref{coeff-traces}), we obtain the expressions of $\Phi_\kappa$ in trace coordinates.
\end{proof}

\begin{lemma} This map is invertible outside the lines $L_{e_{1,\kappa},e_{2,\kappa}}\cup L_{e_{1,\kappa}^{-1},e_{2,\kappa}^{-1}}$ in $\chi_{VI}(a_\kappa)$ 
and $\Phi_\kappa^{-1}$ is given by
\begin{equation*}
  \left\{
    \begin{aligned}
&x_1=(-\kappa x_{1,\kappa}-e_0\kappa^{-1}x_{2,\kappa}+a_3e_0+a_4)(x_{3,\kappa}-c_{e_{1,\kappa},e_{2,\kappa}})^{-1}\\
&x_2=(\kappa x_{1,\kappa}x_{3,\kappa} -e_0\kappa^{-1}x_{1,\kappa} + \kappa x_{2,\kappa} -a_3\kappa^2-a_4\kappa^2e_0^{-1}) (x_{3,\kappa}-c_{e_{1,\kappa},e_{2,\kappa}})^{-1}\\
&x_3=x_{3,\kappa}.
    \end{aligned}
  \right.
\end{equation*}
\end{lemma}

\begin{proof}
In order to find $\Phi_\kappa^{-1}$ we solve the system
\begin{equation*}
  \left\{
    \begin{aligned}
&e_0^{-1}\kappa x_1+\kappa^{-1}x_2=x_{1,\kappa}\\
&(-e_0^{-1}\kappa x_{3,\kappa}) +\kappa^{-1})x_1-e_0^{-1}\kappa x_2=x_{2,\kappa}-a_3\kappa-a_4e_0^{-1}\kappa
    \end{aligned}
  \right.
\end{equation*}
For a fixed value of $x_3=x_{3,\kappa}$ this system is linear and invertible outside the set
\[e_0^{-1}(x_{3,\kappa}-e_0^{-1}\kappa^2-e_0\kappa^{-2})=0.\]
This set is the pair of lines $x_{3,\kappa}=e_{1,\kappa}e_{2,\kappa}^{-1}+e_{1,\kappa}e_{2,\kappa}^{-1}=c_{e_{1,\kappa},e_{2,\kappa}}$, i.e. the pair of lines $L_{e_{1,\kappa},e_{2,\kappa}} \cup L_{e_{1,\kappa}^{-1},e_{2,\kappa}^{-1}}$ in $\chi_{VI}(a_\kappa)$.
Solving this system, we obtain the expression of $\Phi_\kappa^{-1}$ given in the statement of the Lemma.\end{proof}

\bigskip

Finally one can prove by using the above expressions that $\Phi^+_\kappa$ is a symplectic morphism (see also \cite{Kl1}) with respects to the symplectcic forms $\omega_V^\kappa(\theta^+)$ and $\omega_{VI}(\theta_\kappa)$. We have a similar result for $\Phi^-_\kappa$.

\section{The Painlevé V vector field on $\mathcal{M}_V$.}

\subsection{The Riemann-Hilbert map $RH_V$}

\begin{prop} \label{representation}
\begin{enumerate}
\item Any local connection $\nabla$, defined by a germ of irregular $sl_2$-system of Katz rank 1 induces a representation $\rho_\nabla$ of the local wild groupoid $\pi_1(X_0,S_0)$ in $SL_2(\C)$.
\item Any connection $A$ in $\mathcal{C}_V$ induces a representation $\rho_A$ of $\pi_1^{V}(X,S)$ in $SL_2(\C)$.
\end{enumerate}
\end{prop}
\begin{proof} The representation of the local wild fundamental groupoid induced by a local irregular system is defined in the following way. We first consider an extension of the set of base points to any point of $D$ or $\widehat{D}$ which is an extremity of the rays $r_i$, $r_i^\pm$, $i=1,2$. The points $p_1$ and $p_2$ are defined by the singular directions of the connection. Any point of $\widehat{D}$ is arbitrarily represented by a determination in the direction $d$ of a formal fundamental system of solutions $\widehat{Y}_d$ of the connection and any point of $D$ is represented by an actual holomorphic fundamental system of solutions $Y_d$ on a small sector around $d$. Any arc $\gamma$ on $D$  with origin $a$ and end point $b$ is represented, as usually by the comparison between the analytic continuation $\widetilde {Y_a}^{\gamma}$ of $Y_a$ along $\gamma$ with of $Y_b$, i.e. by the matrix $M_\gamma$ such that 
\[Y_b=\widetilde {Y_a}^{\gamma}\cdot M_\gamma.\]
Any arc on $\widehat{D}$ is represented in the same way by using the formal representations of its extremities, and the analytic continuation of the logarithm.
Let $r_d$ be a regular ray whose origin and end-point are represented by $\widehat{Y_d}$ and $Y_d$. The morphism $r_d$ is represented by the comparison between the summation of $\widehat{Y_d}$ (the summation replace here he analytic continuation) with $Y_d$ : \[Y_d=S_d(\widehat{Y_d})\cdot M_{r_d}.\] 
We represent the loops $t_{i,i}(\kappa)$ in the following way. One can suppose that the matrix $\widehat{Y}_{s_i}$ is a diagonal matrix.
Then the action of $\widehat{t}_{i,i}(\kappa)$ on $\widehat{Y}_{s_i}$ is given by the product with the diagonal matrix $diag(\kappa^{-1},\kappa)$.Therefore $\rho_A\left(t_{i,i}(\kappa)\right)=D_\kappa$. We have defined the representation $\rho_\nabla$ on all the generators of $\pi_1(X_0,S_0)$ and therefore on $\pi_1(X_0,S_0)$.

Given a global system $A$ in $\mathcal{S}_V$, we complete the previous local wild representation to a representation of $\pi_1^{V}(X,S)$ by defining $\rho_\nabla(\gamma_{i,j})$, $(i,j)=(2,3),\ (3,4),\ (4,1),\ (3,3),\ (4,4)$ in the usual monodromic way, by analytic continuation along these paths or loops.
A change of representations of the objects, or a gauge action on the linear system, induce equivalent representations, which end the proof. \end{proof}
\bigskip\\
We have defined a Riemann-Hilbert map
\[RH_V:\ \mathcal{M}_V(\alpha)\rightarrow \chi_V(a).\]
The general Riemann-Hilbert problem discusses about the surjectivity of $RH$: in the present context, for any \textit{irreducible} representation $\rho$ in $\chi_V$, there exists a unique connection in $\mathcal{M}_V$ on the trivial bundle whose linear representation is $\rho$: for details see \cite{VdPSai}, \cite{BoMaMi}.

\subsection{Isomonodromic families on $\mathcal{M}_V$}

Any isomonodromic family on $\mathcal{M}_V(\alpha)$ is a fiber of the map $RH_V$, parametrized by a variable $t$, i.e. a family of connections defined by  
\[\frac{dY}{dx}=A(x,t)\cdot Y\]
such that the monodromy and the Stokes operators are locally constant in $\chi_V(a)$. 

\begin{theorem}
There exists an open Zariski set  in $\mathcal{M}_V(\alpha)$ and coordinates $p,q$ on the fiber $\mathcal{M}_V(\alpha)$ such that the isomonodromic families are solutions the Painlevé V hamiltonian differential system (\ref{HV}).
\end{theorem}

\textit{A sketch of proof.} 
We follow here an argument of \cite{VdPSai} wich extends the argument of Schlesinger for fuchsian connections. 
Given a fundamental system of solutions $Y(x,t)$, by isomonodromy, $\frac{d}{dt}Y(x,t)$ and 
$Y(x,t)$ have the same behaviour by analytic continuation or under a Stokes operator. Therefore the quotient 
\[B(x,t)=\frac{d}{dt}Y(x,t)\cdot Y(x,t)^{-1}\]
is univalued outside the fixed singular points.

\begin{lemma} $B(x,t)$ extends to the singular set in a meromorphic way.
\end{lemma} 
\begin{proof} Let $U \subset \C$ be an open subset, and let $D$ be an open disc centered at $\infty$. Let 
$dY/dx=A(x^{-1},t)Y$, where $A\in sl_2(\mathcal{O}(D\times U)$, be a parametrized meromorphic differential system. We suppose that, for all $t\in U$, the eigenvalues of $A(0,t)$ are distinct and that the singularity is irregular, with a Katz rank equal to $1$.
From Theorem (\ref{sommable_par}), the system $dY/dx=AY$ admits, in a small neighborhood of each $t_0\in U$ a formal fundamental solution \emph{analytic in  $t$}~: 
$\widehat{Y}=\widehat{\Phi} x^Le^Q$. The polynomials $Q$ and the matrix $L$ are analytic in $t$ and the entries of $\widehat{\Phi}$ belongs to $\mathcal{O}(V)[[x^{-1}]]$, where $V$ is a small open disc centered at $t_0$. Moreover these coefficients are $1$-summable \emph{uniformly} in  $t$. For a direction $d$ non-singular at $t_0$, if $\vert t-t_0\vert$ is small, then $d$ remains non-singular and the $1$-sum is analytic in 
$t$. Multiplying on the right by an invertible diagonal matrix analytic in $t$, we get a fundamental solution with similar properties.

If the system is isomonodromic, then the matrix $L$ is constant and it is possible to choose 
$\widehat{Y}$ in such a way that the Stokes multipliers does not depend on $t$.

By a simple calculation, we see that 
$\widehat{B}(x,t)=\frac{d}{dt}\widehat{Y}(x,t)\cdot \widehat{Y}(x,t)^{-1}$ does not contain exponentials, and is invariant by the formal monodromy and by the (constant) Stokes multipliers. Therefore its entries belongs to $\mathcal{O}(D)((x^{-1}))$.
These entries are $1$-summable \emph{uniformly} in  $t$ and are invariant by the Stokes multipliers, therefore they are meromorphic in $x$ (with a pole at $\infty$) and analytic in $t$. \end{proof}\bigskip

Therefore $Y(x,t)$ satisfies
two rational linear systems $\frac{dY}{dx}=A(x,t)\cdot Y$ and $\frac{dY}{dt}=B(x,t)\cdot Y$. The compatibility condition requires that the 
two operators $d/dx-A$ and $d/dt-B$ have to commute :
\begin{equation}\label{compatibility}\frac{dA}{dt}-\frac{dB}{dz}+[B,A]=0.\end{equation}
The pair $(A,B)$ is called an isomonodromic Lax pair. If $A$ is irreducible, $B$ is unique. In \cite{VdPSai}, the computation of $B$ and the equivalence of the compatibility condition (\ref{compatibility}) with the hamiltonian system of $P_V$ have been performed. \carre

\subsection{Singularities of the Painlevé V vector field in the Okamoto's compactification of $\mathcal{M}_V(\alpha)$}\label{SingPV}

\textbf{The Okamoto's compactification} \cite{Ok79}, \cite{InIwSai2}. All the Painlevé foliations satisfy the Painlevé property: any solution admits an extension along any path is the basis given by the time variable punctured by the fixed singular points, as \textsl{a meromorphic} function.

In a first step, we search for a compactification on which appear the ``polar'' singular set, which corresponds to the poles of the meromorphic solutions. Here, following H. Chiba \cite{Chiba1} and \cite{Chiba2}, we have to distinguish the equations $P_{IV}$, $P_{II}$ and $P_I$ from $P_{VI}$, $P_V$ and $P_{III}$. For the first triple we obtain this compactification by a convenient weighted projective space. For the second one, the natural weighted projective space presents a vanishing weight and we can't catch all the polar set. In this case we need to glue two copies of this weighted projective space by a convenient B\"acklund transformation.

In a second step one can remove the polar singular set by a convenient sequence of blowing up's. In the version of H. Chiba, it suffices to perform only one weighted blowing up whose weights are given by the eigenvalues of the polar singular point, which are positive integers.

The first step creates new singular points outside the polar singular set, over $z=\infty$, which are saddle-nodes. We describe here this set for the $P_V$ foliation. In this case,  the weights are $\omega=(1,0,1,1)$. Since the second weight vanishes, this space $\p^3_{(1,0,1,1)}$ is not compact and is isomorphic to $\p^2\times\C$. This space is here a manifold (an orbifold for other Painlevé equations: [Chiba1-2-4]), covered by 3 charts:
\begin{equation*}
    \begin{aligned}
    &U_1=\{(X_1:X_2:X_3:X_4),\ X_1\neq 0\},\ (X_1:X_2:X_3:X_4)= (1:u_1:u_2:x);\\
    &U_3=\{(X_1:X_2:X_3:X_4),\ X_3\neq 0\},\ (X_1:X_2:X_3:X_4)=(v_1:v_2:1:y);\\
    &U_4=\{(X_1:X_2:X_3:X_4),\ X_4\neq 0\},\ (X_1:X_2:X_3:X_4)=(w_1:w_2:z:1).
    \end{aligned}
\end{equation*}
The change of charts are given by
\[w_1=v_1y^{-1}=x^{-1},\ w_2=v_2=u_1,\ z=y^{-1}=u_2x^{-1}.\]
\textbf{Singular points. }
The chart $U_4$ is the "initial chart", before compactification. In this chart the vector field corresponding to the hamiltonian $zH_V$ is defined by
\begin{equation*}
  \left\{
    \begin{aligned}
    z\dot{w_1}&=-2w_1^2w_2+w_1^2+w_1z-2w_1w_2z+(\alpha_1+\alpha_3)w_1-\alpha_2z\\
    z\dot{w_2}&=2w_1w_2^2-2w_1w_2-w_2z+w_2^2z-(\alpha_1+\alpha_3)w_2+\alpha_1\\
    \dot{z}&=z
    \end{aligned}
  \right.
\end{equation*}
For each vector field,
the plane $z=0$ is invariant and the singular points of the restriction of the vector field to this plane are given by
\begin{equation*}
  \left\{
    \begin{aligned}
    &(-2w_1w_2+w_1+(\alpha_1+\alpha_3))w_1=0\\
    &2w_1w_2^2-2w_1w_2-(\alpha_1+\alpha_3)w_2+\alpha_1=0\\
    &z=0
    \end{aligned}
  \right.
\end{equation*}
For $\alpha_1\neq \pm\alpha_3$, we obtain 2 singular points $r_1$ and $r_2$:
\[(w_1,w_2,z)=\left(0,\frac{\alpha_1}{\alpha_1+\alpha_3},0\right),\ \  (w_1,w_2,z)=\left(\alpha_1-\alpha_3,\frac{\alpha_1}{\alpha_1-\alpha_3},0\right).\]
\noindent Now we search for singular points over $z=\infty$. In the first chart the Painlevé vector field is (see \cite{Chiba2}):
\begin{equation*}
  \left\{
    \begin{aligned}
    \dot{u_1}&=-2u_1+2u_1^2-u_1u_2+u_1^2u_2+\alpha_1x-(\alpha_1+\alpha_3)u_1x\\
    \dot{u_2}&=-u_2+2u_1u_2-u_2^2+u_2x+2u_1u_2^2-(\alpha_1+\alpha_3)u_2x+\alpha_2u_2^2x\\
    \dot{x}&=x(-1+2u_1-u_2+2u_1u_2-(\alpha_1+\alpha_3)x+\alpha_2u_2x)
    \end{aligned}
  \right.
\end{equation*}
The plane $(x=0)$ is invariant and $P_V|_{(x=0)}$ is the autonomous hamiltonian vector field
\begin{equation*}
  \left\{
    \begin{aligned}
    \dot{u_1}&=-2u_1+2u_1^2-u_1u_2+u_1^2u_2=u_1(u_1-1)(u_2+2)\\
    \dot{u_2}&=-u_2+2u_1u_2-u_2^2+2u_1u_2^2=u_2(2u_1-1)(u_2+1)
    \end{aligned}
  \right.
\end{equation*}
We have 5 singularities $p_1$, $p_2$, $s_1$, $s_2$, $s_3$ in this first chart, which do not depend on the parameters $\alpha_i$:
\[(u_1,u_2,x)=(0,0,0),\ (1,0,0),\ (0,-1,0),\ (\frac{1}{2},-2,0),\ (1,-1,0).\]
The singularities $p_1$ and $p_2$ are \textsl{polar} singular points: H. Chiba has proved in \cite{Chiba2} that they correspond to solutions given by Laurent series at a movable pole $z=z_0$. One can remove these singularities by a weighted blowing up. The eigenvalues of linear part of $P_V$ are  $(-2,-1,-1)$ and $(2,1,1)$ respectively. They are analytically linearizable. 
\bigskip\\
In the third chart --the \emph{Boutroux chart}--, the Painlevé vector field is given by
\begin{equation*}
  \left\{
    \begin{aligned}
    \dot{v_1}&=v_1(1+v_1-2v_2-2v_1v_2+(\alpha_1+\alpha_3-1)y)-\alpha_2y\\
    \dot{v_2}&=v_2(-1+v_2-2v_1+2v_1v_2-(\alpha_1+\alpha_3)y)+\alpha_1y\\
    \dot{y}&=-y^2
    \end{aligned}
  \right.
\end{equation*}
The plane $y=0$ is invariant and the restriction of the vector field at infinity is
\begin{equation*}
  \left\{
    \begin{aligned}
    \dot{v_1}&=v_1(1+v_1-2v_2-2v_1v_2)=v_1(v_1+1)(1-2v_2)\\
    \dot{v_2}&=v_2(-1+v_2-2v_1+2v_1v_2)=v_2(v_2-1)(2v_1+1)
    \end{aligned}
  \right.
\end{equation*}
There are five singular points $s_1$, $s_2$, $s_3$, $s_4$ and $s_5$ in $y=0$ given by    
\[(v_1,v_2,y)=(-1,0,0),\ (-\frac{1}{2},\frac{1}{2},0),\ (-1,1,0),\ (0,0,0),\ (0,1,0).\]
The polar singular set of $P_V$ contains 4 points. In order to catch the two missing polar singular points, H. Chiba introduces a second copy of $\widetilde{\p^3}_{(1,0,1,1)}$ glued to the first one by the B\"acklund transformation $\pi$, given in each chart by
\begin{equation*}
\begin{aligned}
		\pi: &(\tilde{w}_1,\tilde{w}_2,\tilde{z},\tilde{\alpha}_0,\tilde{\alpha}_1,\tilde{\alpha}_2,\tilde{\alpha}_3)=
(z(w_2-1),-w_1z^{-1},z,\alpha_1,\alpha_2,\alpha_3,\alpha_0)\\
& (\tilde{v}_1,\tilde{v}_2,\tilde{y},\tilde{\alpha}_0,\tilde{\alpha}_1,\tilde{\alpha}_2,\tilde{\alpha}_3)=
(v_2-1,-v_1,y,\alpha_1,\alpha_2,\alpha_3,\alpha_0)\\
& (\tilde{u}_1,\tilde{u}_2,\tilde{x},\tilde{\alpha}_0,\tilde{\alpha}_1,\tilde{\alpha}_2,\tilde{\alpha}_3)=
(-u_2^{-1},(u_1-1)^{-1},(u_1-1)^{-1}u_2^{-1}x,\alpha_1,\alpha_2,\alpha_3,\alpha_0)\\
\end{aligned}
\end{equation*}
By computing the singularities of the vector field in the chart $(\tilde{w}_1,\tilde{w}_2,\tilde{z})$, we find two singular points $r_3$ and $r_4$ given by:
\[(\tilde{w}_1,\tilde{w}_2,\tilde{z})=(\tilde{\alpha}_1-\tilde{\alpha}_3+1,\frac{\tilde{\alpha}_1}{\tilde{\alpha}_1-\tilde{\alpha}_3+1},0),\ (\tilde{w}_1,\tilde{w}_2,\tilde{z})=(0,-\frac{\tilde{\alpha}_1}{\tilde{\alpha}_0+\tilde{\alpha}_2},0).\]
We have $r_4=r_1$, and therefore we obtain 3 singular points of regular type.
In the chart $(\tilde{u}_1,\tilde{u}_2,\tilde{x})$, we find five singular points:
\begin{equation*}
\begin{aligned}
&p_3:\ (\tilde{u}_1,\tilde{u}_2,\tilde{x})=(0,0,0),\\
&p_4:\ (\tilde{u}_1,\tilde{u}_2,\tilde{x})=(1,0,0),\\
&s_1:\ (\tilde{u}_1,\tilde{u}_2,\tilde{x})=(1,-1,0),\\
&s_2:\ (\tilde{u}_1,\tilde{u}_2,\tilde{x})=(\frac{1}{2},-2,0),\\
&s_4:\ (\tilde{u}_1,\tilde{u}_2,\tilde{x})=(0,-1,0).
\end{aligned}
\end{equation*}
According to  \cite{Chiba2}, the singular points $p_3$ and $p_4$ are the two missing polar singular points. The singularities $s_1$, $s_2$ and $s_4$ was yet detected in the first copy.
The singular points in the Boutroux chart $(\tilde{v}_1,\tilde{v}_2,y)=(v_2-1,-v_1,y)$ are $s_1$, $s_2$, $s_3$, $s_4$ and $s_5$ without new singular point (the change of chart is polynomial).
We have obtained:
\begin{prop}
The Painlevé vector field admits 3 singular points $r_1,\ r_2,\ r_3$ of regular type over $z=0$, 4 polar singular points $p_i$, $i=1...4$ over $z=\infty$, which can be removed by a weighted blowing-up, and 5 singular points $s_i$ over $z=\infty$, of saddle-node type.
\end{prop} 
By a local study around each saddle-node singularity, one can recover the 5 solutions of A. Parusnikova in Laurent series around $z=\infty$ for the Painlevé V equation: \cite{Parus}.

\section{Dynamics on $\chi_V$}

\subsection{The tame dynamics}

As mentionned in the introduction, the dynamic of the Painlevé VI equation on $\chi_{VI}$ is induced by the automorphisms of the groupoid, given by braids over 3 points. Here since the singular set reduces to 3 points, the braid group over two points has only one generator, which can be represented by the pure braid $b$ over $s_3$ and $s_4$. Furthermore one can consider the ``half-monodromy'' defined the (non pure) braid $b_{3,4}$ which permutes the positions of $s_3$ and $s_4$ and induces an involution on the local parameters. These actions induced by the trace coordinates on $\mathcal{C}_V(\theta^+)$ has been computed by the authors in \cite{PR}
and M. Klimes in \cite{Kl1}.

\begin{prop} The action of the braid $b_{3,4}:$ $\mathcal{C}_V(\theta_1^+,\theta_2^+,e_0)\rightarrow\mathcal{C}_V(e_0^{-1}\theta_2^+,e_0^{-1}\theta_1+,e_0^{-1})$ is defined by:
\begin{equation*}
\left\{
    \begin{aligned}
   & x'_1=e_0^{-1}(-x_2-x_1x_3+\theta_2^+)\\
   & x'_2=e_0^{-1}x_1\\
   & x'_3=x_3
   \end{aligned}
    \right.
		\mbox{ and }\left\{
    \begin{aligned}
   & e'_0=e_0^{-1}\\
   & \theta_1^{+}\,'=e_0^{-1}\theta_2^+\\
   & \theta_2^{+}\,'=e_0^{-1}\theta_1^+
   \end{aligned}
    \right.
\end{equation*}
The action of the pure braid $b_{3,4}\circ b_{3,4}:$ $\mathcal{C}_V(\theta^+)\rightarrow\mathcal{C}_V(\theta^+)$ is defined by:
\begin{equation*}
  g_{3,4}:\ \left\{
    \begin{aligned}
   & x'_1=x_1x_3^2+x_2x_3-x_1-\theta_2^+x_3+\theta_1^+\\
   & x'_2=-x_1x_3-x_2+\theta_2^+\\
   & x'_3=x_3
    \end{aligned}
    \right.
\end{equation*}
 \end{prop}

As for the dynamics on $\mathcal{C}_{VI}(\theta)$, this dynamics and its inverse are \textsl{polynomial} dynamics. This is the only part of the dynamics on $\chi_{VI}$ which do not degenerate through the confluent morphisms from $\chi_{VI}$ to $\chi_{V}$  (see next paragraph).

\begin{defi} The dynamics generated by $g_{3,4}$ is called the tame dynamics, and denoted by $Tame(\mathcal{C}_V)$ .\end{defi}

\subsection{The confluent dynamic}\label{confluent dynamics}

Using the birational map $\Phi_\kappa=\Phi_\kappa^+$, each dynamic $h_{i,j}$ on $\mathcal{C}_{VI}(\theta_\kappa)$ defined by the braids induces a family of birational dynamics $g_{i,j}(\kappa)=\phi_\kappa^{-1}\circ h_{i,j}\circ\phi_\kappa$ on $\chi_V^+(a)$.
This one was previously obtained by M. Klimes by  using an analytic confluent process in the spaces of connections.
\begin{defi}
The confluent dynamic $Conf(P_V)$ is the dynamic generated by $g_{1,2}(\kappa)$, $g_{2,3}(\kappa)$ and $g_{3,1}(\kappa)$ for $\kappa$ is $\C^*$.
\end{defi}
We still have the relation $g_{1,2}(\kappa)\circ g_{2,3}(\kappa)\circ g_{3,1}(\kappa)=id.$ Therefore it suffices to compute $g_{1,2}(\kappa)$ and $g_{2,3}(\kappa)$. Notice that, from the braid group relations, $g_{1,2}(\kappa)=g_{3,4}(\kappa)^{-1}$.
By a direct computation using proposition (\ref{VI-V}), it can be checked that
\begin{prop}\label{tame}
The dynamic $g_{3,4}(\kappa)=\phi_\kappa^{-1}\circ h_{3,4}\circ\phi_\kappa$ on $\chi_V(a)$ do not depend on $\kappa$ and generates the tame dynamic.
\end{prop}

\begin{rema}\label{corr-dyn} Since the braid between $s_1$ and $s_2$ (or between $s_3$ and $s_4$) corresponds to a loop around $0$ in the time space, $g_{3,4}$ is conjugated through the Riemann-Hilbert correspondence to the non linear monodromy of the foliation $\mathcal{F}_V$ over $z=0$.
\end{rema}

Contrary to the previous case, the dynamics $g_{2,3}(\kappa)$ and $g_{3,1}(\kappa)$ are not generated by an automorphism of $\pi_1^{V}(Y,S)$: indeed on the groupoid level, $\varphi_\kappa$ is not an isomorphism. Nevertheless they induce families of \textit{birational} dynamics on $\mathcal{C}_V(\theta)$:

\begin{prop}\label{dyn conf}
The family of dynamics $g_{2,3}(\kappa)$ is defined  by:
\begin{equation*}
\begin{aligned}
&X_1= \frac{e_0}{x_2}\\
&X_2=x_2-\frac{\kappa^2}{x_2}+e_0^{-1}\kappa^2x_1\\
&X_3= \kappa^{-2}x_2^2x_3-(e_0^{-2}\kappa^2-\kappa^{-2})x_1x_2-2e_0^{-1}x_2^2+(e_0^{-1}\theta_2+\kappa^{-2}\theta_1)x_2+(e_0^{-1}\kappa^2+e_0\kappa^{-2})
\end{aligned}
\end{equation*}
The second family $g_{3,1}(\kappa)$ is given by $g_{1,3}(\kappa)=g_{1,2}\circ g_{2,3}(\kappa).$
\end{prop}
We first search for the matrix expressions of the confluent dynamics :
\begin{prop} The confluent dynamic $(\varphi_\kappa^*)^{-1}\circ\ h_{2,3}\circ\varphi_\kappa^*:\ \chi_V^\kappa(a) \rightarrow \chi_V^\kappa(a)$ is defined by $([U_1,M_0,U_2,M_3]_D\rightarrow [V_1,M_0,V_2,N_3]_D $ with
\begin{equation*}
\begin{aligned}
&V_1=\left(\begin{array}{cc}
1 & 0 \\
\kappa^{-1}e_0s_1\theta_3+\kappa^{-1}e_0^{-1}\gamma_3-\kappa e_0^{-1}\gamma_3+\kappa^{-1}e_0s_1s_2\gamma_3 & 1
\end{array}\right),  \\
&V_2=\left(\begin{array}{cc}
1 & \kappa^{-1}e_0s_2\theta_3+\kappa^{-1}e_0s_2^2\gamma_3-\kappa^{-1}e_0s_2\delta_3-\kappa^{-1}e_0\beta_3+\kappa e_0^{-1}\beta_3+\kappa e_0^{-1} s_2\delta_3  \\  0 & 1 \end{array}\right), \\
&N_3=\frac{1}{\theta_3+s_2\gamma_3}\left(\begin{array}{cc}
\theta_3^2+\theta_3\gamma_3s_2+\beta_3\gamma_3+\gamma_3\delta_3s_2  & -e_0\kappa^{-1}(\theta_3s_2+\gamma_3s_2^2\beta_3-\delta_3s_2) \\
\kappa e_0^{-1}\gamma_3 & 1
\end{array}\right).
\end{aligned}
\end{equation*}
\end{prop}
\begin{proof}
We consider an element $\rho$ of $\chi_V(a)$ given by its normalized representation in the canonical Cartan-Borel decomposition:
\[\rho=[U_1,M_0,U_2,M_3,M4],\ (U_1,M_0,U_2,M_3)\in U^-\times D\times U^+\times SL_2(\C),\ M_4=(U_1M_0U_2M_3)^{-1}.\]
We set:
\[U_1=\left(\begin{array}{cc}1 & 0\\ u_1 & 1 \end{array}\right),\
{M}_0=\left(\begin{array}{cc}e_0 & 0\\ 0 & e_0^{-1}\end{array}\right),\
U_2=\left(\begin{array}{cc}1 & u_2\\ 0 & 1 \end{array}\right),\
M_3=\left(\begin{array}{cc}\theta_3 & \beta_3\\ \gamma_3 & \delta_3 \end{array}\right).\]
The map $\Phi_\kappa$ is defined by:
\[\Phi_\kappa([U_1,M_0,U_2,M_3,M_4])=[M_1,M_2,M_3,M_4]\mbox{ with }M_1=U_1D_\kappa,\ M_2=D_\kappa^{-1}M_0U_2.\]
Therefore
\[M_1=\left(\begin{array}{cc}\kappa & 0\\ \kappa u_1 & \kappa^{-1} \end{array}\right),\
M_2=\left(\begin{array}{cc}\kappa^{-1}e_0 & \kappa^{-1}e_0u_2\\ 0 & \kappa e_0^{-1} \end{array}\right),\
M_3=\left(\begin{array}{cc}\theta_3 & \beta_3\\ \gamma_3 & \delta_3 \end{array}\right).\]
According to \cite{PR}, the dynamic of $h_{2,3}$ is given by
\begin{equation*}
\begin{aligned}
&M_1\rightarrow (M_2M_3)^{-1}M_1(M_2M_3)\\
&M_2\rightarrow M_2\\
&M_3\rightarrow M_3\\
&M_4\rightarrow (M_2M_3)^{-1}M_4(M_2M_3)
\end{aligned}
\end{equation*}
Since each data is given up to a common conjugacy, we can also use:
\begin{equation*}
\begin{aligned}
&M_1\rightarrow M'_1={M_2}^{-1}M_1M_2\\
&M_2\rightarrow M'_2=M_3M_2{M_3}^{-1}\\
&M_3\rightarrow M'_3=M_3\\
&M_4\rightarrow M'_4={M_2}^{-1}M_4M_2
\end{aligned}
\end{equation*}
Now, in order to compute $\phi_\kappa^{-1}(M'_1,M'_2,M'_3,M'_4)$,
we make use of the matrix $P$ given by Lemma (\ref{central lemma}), defined by a mixed basis for the pair $(M'_1,M'_2)$.
Let $(u,v)$ a basis of the representation of the initial object such that the representation $\Phi_\kappa(\rho)$ is given by the matrices $M_i$: $u$ is an eigenvector of $M_2$ for the eigenvalue $e_2=\kappa^{-1} e_0$ and $v$ is an eigenvector of $M_1$ for the eigenvalue ${e_1}^{-1}=\kappa^{-1}$. The vectors
\[u'=M_3\cdot u',\ v'={M_2}^{-1}\cdot v\]
are eigenvectors for $M'_2$ (with eigenvalue $\kappa^{-1}e_0$) and for $M'_1$ (with eigenvalue $\kappa^{-1}$). The change of basis is given by the matrix
\[P=\left(\begin{array}{cc}\theta_3 & -\kappa^{-1}e_0u_2\\ \gamma_3 & \kappa^{-1}e_0 \end{array}\right)\]
which is invertible outside
\[\kappa^{-1}e_0(b_3+u_2\gamma_3)=0.\]
Using Lemma (\ref{coeff-traces}), this locus is given by $x_2=0$. Therefore the dynamic $g_{2,3}(\kappa)$ is a \textsl{rational} dynamic with polar set $x_2=0$. We have:
\[P^{-1}(M'_1M'_2)P=
\left(\begin{array}{l}
e_0\ \ \ \ \kappa^{-1}e_0^2u_2b_3+\kappa^{-1}e_0^2u_2^2\gamma_3-\kappa^{-1}e_0^2u_2\delta_3-\kappa^{-1}e_0^2\beta_3+\kappa\beta_3+\kappa u_2\delta_3\\
\kappa^{-1}e_0^2u_1b_3+\kappa^{-1}\gamma_3-\kappa\gamma_3+\kappa^{-1}e_0^2u_1u_2\gamma_3 \hspace{2cm} \star
\end{array}\right).\]
where the coefficient $\star$ is not specified.
The LDU-decomposition of this matrix is given by $(V_1,M_0,V_2)$ with
\begin{equation*}
\begin{aligned}
&V_1=\left(\begin{array}{cc}
1 & 0 \\
\kappa^{-1}e_0u_1b_3+\kappa^{-1}e_0^{-1}\gamma_3-\kappa e_0^{-1}\gamma_3+\kappa^{-1}e_0u_1u_2\gamma_3 & 1
\end{array}\right), \\
&V_2=\left(\begin{array}{cc}
1 & \kappa^{-1}e_0u_2b_3+\kappa^{-1}e_0u_2^2\gamma_3-\kappa^{-1}e_0u_2\delta_3-\kappa^{-1}e_0\beta_3+\kappa e_0^{-1}\beta_3+\kappa e_0^{-1} u_2\delta_3  \\  0 & 1 \end{array}\right).
\end{aligned}
\end{equation*}
we have:
\[N_3=P^{-1}M'_3P=\frac{1}{b_3+u_2\gamma_3}\left(\begin{array}{cc}
b_3^2+b_3\gamma_3u_2+\beta_3\gamma_3+\gamma_3\delta_3u_2  & -e_0\kappa^{-1}(b_3u_2+\gamma_3u_2^2\beta_3-\delta_3u_2) \\
\kappa e_0^{-1}\gamma_3 & 1
\end{array}\right)\]
which proves the proposition. \end{proof}
\bigskip\\
\emph{Proof of Proposition (\ref{dyn conf}).}
The image of $x_1=\delta_3$ is given by $X_1=N_3[2,2]=\frac{1}{b_3+u_2\gamma_3}$. According to Lemma (\ref{coeff-traces}), we have:
\[N_3[2,2]=\frac{1}{b_3+u_2\gamma_3}=\frac{1}{(a_3-x_1)+e_0^{-1}(x_2-e_0a_3+e_0x_1)}=\frac{e_0}{x_2}.\]
The image of $x_2=\delta_4$ is given by $N_4[2,2]$ where $N_4=P^{-1}M'_4P$. We have
\[N_4[2,2]=\frac{-e_0^{-2}\kappa^2\gamma_3\beta_4+b_3\delta_4+u_2\gamma_3\delta_4}{b_3+u_2\gamma_3}.\]
From $M_4=(U_1M_0U_2M_3)^{-1}$ we obtain that $\beta_4=-e_0(\beta_3+u_2\delta_3)$. Therefore,
\begin{equation*}
\begin{aligned}
\gamma_3\beta_4&=-e_0\gamma_3\beta_3-e_0u_2\gamma_3\delta_3\\
&=-e_0(b_3\delta_3-1)-e_0u_2\gamma_3\delta_3\\
&=-e_0((a_3-x_1)x_1-1)-x_1(x_2+e_0x_1-e_0a_3)\\
&=e_0-x_1x_2.
\end{aligned}
\end{equation*}
From Lemma (\ref{coeff-traces}) and the above equality we obtain
\[X_2=N_4[2,2]=x_2-\kappa^2x_2^{-1}+e_0^{-1}\kappa^2x_1.\]
We can obtain the last component by computing
$X_3=tr(M''_1M''_2)=tr({M'_2}^{-1}M'_1M'_2M'_3M'_2{M'_3}^{-1})$, or by using the equation of the cubic surface:
\[X_3=\frac{-X_1^2-X_2^2+\theta_1X_1+\theta_2X_2-\theta_4}{X_1X_2-e_0}.\]
\carre

\subsection{The canonical dynamics on $\mathcal{C}_V(\theta)$}\label{candyn}

In \cite{Kl1}, Martin Klimes has computed the dynamics obtained on $\mathcal{C}_V(\theta)$ from the wild dynamics (non linear stokes operators and non linear exponential tori of one irregular singularity of the Painlevé foliation) through the Riemann-Hilbert morphism. It turns out that this dynamics is a rational one on $\mathcal{C}_V(\theta)$. We enhance here that this dynamics was already present on the cubic surface in a canonical way, using the log-canonical coordinates introduced in subsection \ref{logcan}. 
This section is completely independent of the other ones excepted \ref{logcan}, and only deals with dynamics on a singular cubic surface. Nevertheless the terminology (canonical Stokes operators etc...) is deeply influenced by the confluent dynamics. We do not need here any restriction on the parameters. 
 
\bigskip

Let $y$ be a log-canonical function on $\mathcal{C}_V(\theta)$: there exists $z$ such that $(y,z)$ --or $(z,y)$-- is a log-canonical system of coordinates.
We consider the logarithmic hamiltonian function $H_y$ on $\mathcal{C}_V(a)$ such that \[dH_y=\frac{dy}{y}.\]
Let $X_y$ be the hamiltonian vector field related to $H_y$ for the symplectic form $\omega_V$:
\[\omega_V(X_y,\cdot)=-\frac{dy}{y}.\]
Through the symplectic morphism $(y,z)$, the image of this vector field is $z\frac{\partial}{\partial z}$ (or $-z\frac{\partial}{\partial z}$ if $y$ is completed in log-canonical system of coordinates by the left side). Therefore its flow $z(t)=z_0e^{\pm t}$ is globally defined on $\C$, and it can be factorized through a multiplicative action of $\C^*$ by setting $\lambda=e^{t}$. Let $T_y$ be this multiplicative family of (rational symplectic) automorphisms on $\mathcal{C}_V$.
\begin{defi} 
$T_y=\{t_\lambda:\ (z',y)\mapsto(\lambda^{-1} z',y)\}=\{t_\lambda:\ (y,z)\mapsto(y,\lambda z)\}$ is the exponential torus related to the log canonical function $y$. 
\end{defi}
\begin{rema}
\begin{enumerate}
	\item Each element of $T_y$ keeps invariant $y$, and the set of rational invariant functions of $T_y$ is $\C(y)$.
	\item Consider the log-canonical pentuple  $(z_0,y_1,z_1,y_2,z_2)$. The elements of $T_{y_1}$ keep invariant the reducible locus $z_0=0$ and $z_1=0$, and the elements of $T_{y_2}$ keep invariant the reducible locus $z_1=0$ and $z_2=0$.    
\end{enumerate}
\end{rema}

The subgroups of $Symp=Symp(\C^2,\omega_{log})$: $B_i$, $B^\natural_i$, $U_i$, $i=1,2$ are defined in the Appendix.
They induce subgroups of $Symp(\mathcal{C}_V,\omega_V)$ by using a system of log-canonical coordinates. Given a log canonical function $z=z_k$ in $\mathcal{S}$, one can complete $z$ into two log-canonical systems, either by the left side: $(y,z)$ or by the right side: $(z,y')$.
\begin{lemma} \label{BzBy}
\begin{enumerate}
	\item We fix a canonical triple $(y,z,y')$. The pull-back of $B_2$ (resp. $B^\natural_2$, $U_2$) by $(y,z)$ and the pull-back 
of $B_1$ (resp. $B^\natural_1$, $U_1$) by $(z,y')$ define the same subgroup of $Sym(\mathcal{C}_V,\omega_V)$. 
  \item We fix a canonical triple $(z,y,z')$. The pull-back of $B_2$ (resp. $B^\natural_2$, $U_2$) by $(z,y)$ and the pull-back 
of $B_1$ (resp. $B^\natural_1$, $U_1$) by $(y,z')$ define the same subgroup of $Sym(\mathcal{C}_V,\omega_V)$.
\end{enumerate}
\end{lemma}
The subgroups obtained at the first point of Lemma (\ref{BzBy}) are denoted by 
$B_z$, $B^\natural_z$, $U_z$, and the subgroups obtained at the second point are denoted by
$B_y$, $B^\natural_y$, $U_y$. 
\begin{defi} Given a canonical system of coordinates $(y,z)$ we will say that $B_y$ and $B_z$ are the opposite Borel subgroups associated to $(y,z)$.\end{defi}

\textit{Proof of Lemma (\ref{BzBy}).}
1- The elements of the pullback of $B_2$ by $(y,z)$ are given by
\[(y,z)\rightarrow (r(z)y,\lambda z),\ r\in\C(z)^*,\ \lambda\in\C^*,\]
and the elements of the pullback of $B_1$ by $(z,y')$ are given by
\[(z,y')\rightarrow (\lambda z,r(z)y'),\ r\in\C(z)^*,\ \lambda\in\C^*.\]
Therefore using $yy'=z+e_0$, the elements of the first group also write
\[(z,y')\rightarrow (\lambda z,r'(z)y'),\mbox{  with }r'(z)=\frac{\lambda z+e_0}{(z+e_0)r(z)},\]
which proves the first point.

2- The elements of the pullback of $B_2$ by $(z,y)$ are given by
\[(z,y)\rightarrow (r(y)z,\lambda y),\ r\in\C(y)^*,\ \lambda\in\C^*,\]
and the elements of the pullback of $B_1$ by $(y,z')$ are given by
\[(y,z')\rightarrow (\lambda y,r(y)z'),\ r\in\C(y)^*,\ \lambda\in\C^*.\]
Therefore using $zz'=Q(y)$, where $Q=Q_1$ or $Q_2$ is a polynomial, the elements of the first group also write
\[(y,z')\rightarrow (\lambda y,r'(y)z'),\mbox{  with }r'(y)=\frac{Q(\lambda y)}{Q(y)r(y)},\]
which proves the second point.
\carre

\begin{prop}\label{centralisateur} We have $Z(T_y)=B_y$, $N(Y_y)=B_y\cup B_y^-$, where $B_y^-=\{(y,z')\rightarrow (\lambda y,r(y)z'^{-1}),\ r\in\C(y)^*,\ \lambda\in\C^*\}$.   \end{prop}
\begin{proof} Let $t_\lambda:\ (y,z')\rightarrow (y,\lambda z').$ An element $\rho$ of $Symp(\mathcal{C}_V)$, given  by
\[\rho:\ (y,z')\rightarrow (Y(y,z'),Z'(y,z'))\] 
commutes with $t_\lambda$ if and only if, for any $\lambda$ in $\C^*$,
\begin{align*}
&Y(y,\lambda z')=Y(y,z'),\\
&Z'(y,\lambda z')=\lambda Z'(y,z').
\end{align*}
We claim that any rational function $r$ in $K(x)$ over a field $K\supset \C$ which satisfies $r(\lambda x)=\lambda r(x)$ for any $\lambda$ in $\C^*$ is a constant function. Indeed if it was not a constant function it would admit an infinite number of zeroes or poles in $K$. By applying this remark to 
$z'\mapsto Y(y,z')$ and to $z'\mapsto Z'(y,\lambda z')z'^{-1}$, we conclude that $Y$ do not depend on $z'$ and that $Z'$ is linear in $z'$. There exists two rational functions $q(y)$ and $r(y)$ such that
\begin{align*}
&Y(y,\lambda z')=q(y),\\
&Z'(y,\lambda z')=r(y) z'.
\end{align*}
We have
\[\frac{dY}{Y}\wedge\frac{dZ'}{Z'}=q'(y)\frac{dy}{y}\wedge(\frac{r'(y)dy}{y}+\frac{dz}{z})=q'(y)\frac{dy}{y}\wedge\frac{dz}{z}.\]
Therefore $q'(y)=1$, and $q(y)=\mu$ in $\C^*$. The elements of $Z(T_y)$ write
\[(y,z')\rightarrow (\mu y,r(y)z')\] 
which prove that $Z(T_y)=B_y$.\\
Suppose now that an element $\rho$ is in the normalizer of $T_y$. Then either it commutes whith each element of $T_y$, or it anti-commutes with each element of $T_y$. In this last case a similar computation proves that 
$\rho:\  (y,z')\rightarrow (\mu y,r(y)z'^{-1}).$
\end{proof}

\begin{prop}\label{stokes alg} Let $\mathfrak{T}=(y,z,y')$ be a log-canonical triple in the cluster sequence $\mathfrak{S}$.
 There exists a unique element $s_\mathfrak{T}$ of $U_z$ such that $s_\mathfrak{T}T_ys_\mathfrak{T}^{-1}=T_{y'}$.
 This operator is given by:
\begin{equation*}
\begin{aligned}
 s_\mathfrak{T}:\ &(y,z)\rightarrow (y(1+e_0^{-1}z),z)
\mbox{   or by }\\ &(z,y')\rightarrow (z, y'(1+e_0^{-1}z)^{-1})=(z,e_0y^{-1})
\end{aligned}
\end{equation*}
and we have:
$s_\mathfrak{T}^{-1}:\ (y,z)\rightarrow (y(1+e_0^{-1}z)^{-1},z)=(e_0y'^{-1},z).$
\end{prop}
\begin{proof} 
We have:
\begin{equation*}
\begin{aligned}
T_y&=\{t_y(\lambda):\ (y,z)\rightarrow(y,\lambda z)\},\\
T_{y'}&=\{t_{y'}(\lambda):\ (z,y')\rightarrow(\lambda^{-1} z,y')\}\\
&=\{t_{y'}(\lambda):\ (y,z)\rightarrow(y\frac{1+e_0^{-1}\lambda^{-1}z}{1+e_0^{-1}z},\lambda^{-1}z)\},\\
U_z&=\{u_r:\ (y,z)\rightarrow(r(z)y,z), r\in\C(z)^*, r(0)=1\}.
\end{aligned}
\end{equation*}
Therefore the element $s$ of $U_z$:
$(y,z)\rightarrow (y(1+e_0^{-1}z),z)$ satisfies
 for all $\lambda$ in $\C^*$
\[s\circ t_y(\lambda)\circ s^{-1}=t_{y'}(\lambda^{-1}).\]
In order to prove the unicity of $s$, suppose that there exist another $s'$ in $U_z$ which conjugates $T_y$ and $T_{y'}$. We have : 
$s^{-1}\circ s'(T_y)=T_y$ i.e. $s^{-1}\circ s'$ belongs to the normalizer $N(T_y)$ of $T_y$. We have 
$N(T_y)\cap U_z=\{id\}:$
this is a direct consequence of Proposition (\ref{centralisateur}).\end{proof}

\begin{defi} The automorphism $s_\mathfrak{T}$ is the canonical Stokes operator induced by the triple $\mathfrak{T}=(y,z,y')$.
\end{defi}
\begin{prop}\label{properties stokes}
\begin{enumerate}
\item The automorphism $s_\mathfrak{T}$ writes in the coordinates $h=\log y, l=\log (1+e_0^{-1}z)$:
$(h,l)\rightarrow (h+l,l).$
\item $s_\mathfrak{T}$ is a \textit{pseudo-generator} of $U_z$, that is to say the family $\{t_y(\lambda)\circ s_\mathfrak{T}\circ t_y(\lambda)^{-1},\ \lambda\in\C^*\}$ generates $U_z$. Consequently, $B_z^\natural:=<U_z,T_y>=<s_\mathfrak{T},T_y>$.
\item Let $s_k$ be the canonical Stokes operator related to the triple $(y_k,z_k,y_{k+1})$. We have:
	$\sigma_1\circ s_1 \circ \sigma_1^{-1}=s_2^{-1}$; $\ g\circ s_k \circ g^{-1}=s_{k+2}.$
\item For $i=1,2$,\ $s_i$ keeps invariant the lines $z_i=0$, and the restriction of $s_i$ on each line is a translation. More generally, $s_k$ keeps invariant $z_k=0$, and the restriction of $s_k$ to the set of rational curves $z_k=0$ has no fixed point.
\end{enumerate}
\end{prop}

\textit{Proof.}
\begin{enumerate}
	\item This first point is trivial.
	\item We have:
	\[\{t_{y}(\lambda)\circ s_\mathfrak{T}\circ t_{y}(\lambda)^{-1},\ \lambda\in\C^*\}=\{(y,z)\rightarrow (y(1+\nu z),z),\ \nu\in\C^*\}\]
	The statement comes from the fact that any rational function $r$ such that $r(0)=1$ writes
	$r(z)=\prod_i(1-\mu_i z)\prod_j(1-\nu_j z)^{-1}$ where the $\mu_i^{-1}$ are the zeroes of $r$ and the $\nu_j^{-1}$ are the poles of $r$.
	\item We have:
	\begin{equation*}
	\begin{aligned}
	&s_1:\ (y_1,z_1)\rightarrow (y_1(1+e_0^{-1}z_1,z_1)\\
	&s_2:\ (z_2,y_3)\rightarrow (z_2,y_3(1+e_0^{-1}z_2)^{-1}).
	\end{aligned}
	\end{equation*}
	Since $\sigma_1:\ (y_1,z_1)\rightarrow (y_3,z_2)$, we have $\sigma_1\circ s_1 \circ \sigma_1^{-1}=s_2^{-1}$.\\
	The automorphism $g^{-1}$ send the triple $(y_k,z_k,y_{k+1})$ on $(y_{k+2},z_{k+2},y_{k+3})$. Therefore
	$s_k=g^{-1}\circ s_{k+2}\circ g$.
	\item From the previous point, it suffices to prove the result for $s_1$.
We lift the expression of $s_1$ in the coordinate system $(y_1,z_1,x_3)$:
\[s_1:\ (y_1,z_1,x_3)\rightarrow (e_0^{-1}y_1(z_1+e_0),z_1,X_3).\]
The equation of the character variety $\mathcal{C}_V(\theta)$ is given by
\[x_3z_1+y_1^2+(z_1+e_0)^2y_1^{-2}-\theta_1y_1-\theta_2(z_1+e_0)y_1^{-1}+\theta_4=0,\]
which is equivalent to 
\[x_3z_1+y_1^{-2}z_1^2+2e_0y_1^{-2}z_1-\theta_2y_1^{-1}z_1=-y_1^2-e_0^2y_1^{-2}+\theta_1y_1+\theta_2e_0y_1^{-1}-\theta_4.\]
We have:
\[X_3z_1+e_0^{-2}y_1^2(z_1+e_0)^2+e_0^2y_1^{-2}-\theta_1e_0^{-1}y_1(z_1+e_0)-\theta_2e_0y_1^{-1}+\theta_4=0.\]
Therefore,
\begin{equation*}
\begin{aligned}
X_3z_1&=-e_0^{-2}y_1^2z_1^2-2e_0^{-1}y_1^2z_1+\theta_1e_0^{-1}y_1z_1
-y_1^2-e_0^2y_1^{-2}+\theta_1y_1+\theta_2e_0y_1^{-1}-\theta_4\\
&=-e_0^{-2}y_1^2z_1^2-2e_0^{-1}y_1^2z_1+\theta_1e_0^{-1}y_1z_1
+x_3z_1+y_1^{-2}z_1^2+2e_0y_1^{-2}z_1-\theta_2y_1^{-1}z_1.
\end{aligned}
\end{equation*}
We obtain:
\[X_3=-e_0^{-2}y_1^2z_1-2e_0^{-1}y_1^2+\theta_1e_0^{-1}y_1
+x_3+y_1^{-2}z_1+2e_0y_1^{-2}-\theta_2y_1^{-1}.\]
Therefore the restriction of $s_1$ to $z_1=0$ is given by
\[x_3\rightarrow x_3+\theta_1e_0^{-1}y_1-\theta_2y_1^{-1}.\]
On each component of $z_1=0$, $y_1$ is constant ($=e_3^{\pm }$ or $e_0e_4^{\pm 1}$) and this map is a translation. \carre
\end{enumerate}

\begin{defi} The canonical dynamics induced by the canonical triple $\mathfrak{T}=(y,z,y')$ is the subgroup $Dyn(\mathfrak T)=<T_y, s_\mathfrak{T},T_{y'}>$ of $Symp(\mathcal{C}_V)$ generated by $T_y$, $s_\mathfrak{T}$, and $T_{y'}$.
\end{defi}

\begin{prop}  We have : $Dyn(\mathfrak T)=<T_y, s_\mathfrak{T}>=B_z^\natural$.\end{prop}
\begin{proof} The first equality is a consequence of the relation $s_\mathfrak{T}T_ys_\mathfrak{T}^{-1}=T_{y'}$ obtained in Proposition (\ref{stokes alg}), and the second one was obtained at the second point of Proposition (\ref{properties stokes}).
\end{proof}

\begin{defi}\label{can_dyn}
The canonical dynamics $Dyn(\mathcal{C}_V)$ induced by the fundamental canonical sequence $\{y_0,z_0,y_1,z_1,y_2,z_2,y_3\}$
is the rational symplectic dynamic generated by:
\[g:\ (y_2,z_2)\rightarrow (y_0,z_0),\ T_{y_1},\ s_1=s_{(y_1,z_1,y_2)}, s_2=s_{(y_2,z_2,y_3)}.\]
The subgroup $Tame(\mathcal{C}_V)$ of $Dyn(\mathcal{C}_V)$ generated by $g$ is the tame canonical dynamics.
\end{defi}
\begin{rema}
\begin{enumerate}
	\item From Proposition (\ref{z1}), the canonical sequence is almost unique, that is we can only replace $z_k$ with $cz_k$. This action is an element of $T_{y_k}$ and therefore do not change the dynamics. This dynamics only depends on the polynomial $F_V$, which justify the notation.
	\item $Dyn(\mathcal{C}_V)$ also contains $T_{y_0}=g^*T_{y_2}$ and $s_0=s(y_0,z_0,y_1)=g^*s_2$, and more generally all the $T_{y_k}$ and all the $s_k$ for $k$ in $\Z$.
	\item $Dyn(\mathcal{C}_V)=<B_{z_1}^\natural,B_{z_2}^\natural,g>.$
\end{enumerate}
\end{rema}

\begin{prop}\label{mon_can_formelle} The element $\widehat{m}$ of $Dyn(\mathcal{C}_V)$ such that $g=s_2\circ\widehat{m}\circ s_1$ is defined by
\[(y_2,z_2)\rightarrow(y_2,e_0^2z_2y_2^{-4}).\]
\end{prop}
\begin{proof}
We have
\[s_1^*z_2=s_1^*Q_2(y_2)z_1^{-1}=Q_2(e_0y_1^{-1})z_1^{-1}=e_0^2y_1^{-4}Q_1(y_1)z_1^{-1}=e_0^2y_1^{-4}z_0.\]
Therefore,
\begin{equation*}
\begin{aligned}
&g^{-1*}s_1^*z_2=e_0^2y_3^{-4}z_2,\\
&s_2^*g^{-1*}s_1^*z_2=e_0^{-2}y_2^4z_2.
\end{aligned}  
\end{equation*}
We also have
\[s_2^*g^{-1*}s_1^*y_2=s_2^*g^{-1*}(e_0y_1^{-1})=s_2^*(e_0y_3^{-1})=y_2,\]
which proves that $\widehat{m}^{-1}=s_1\circ g^{-1}\circ s_2$ is defined by
$(y_2,z_2)\rightarrow(y_2,e_0^{-2}z_2y_2^{4}).$
\end{proof}

\begin{defi} The element $\widehat{m}$ in $Dyn(\mathcal{C}_V)$ is the canonical formal monodromy.\end{defi}

Note that $\widehat{m}=t_{2}(e_0^{2}y_{2}^{-4})$. In particular, $\widehat{m}$ is in the centralizer $Z(T_{y_2})=B_{y_2}$ of $T_{y_2}$, and preserves $y_{2}$.

\subsection{Comparison between the confluent and the canonical dynamics}

We have previously found that the confluent dynamics and the canonical dynamics are both extensions of the tame dynamics.
In order to compare $Dyn(\mathcal{C}_V)$ and $Conf(\mathcal{P}_V)$, we write the confluent dynamic in the canonical coordinates.
The confluent dynamic is generated by $g_{1,2}=g_{3,4}^{-1}$ given by Proposition (\ref{tame}),  the family $g_{2,3}(\kappa)$ 
given by Proposition (\ref{dyn conf}) and the family $g_{3,1}(\kappa)$, which satisfies the relation $g_{1,2}\circ g_{2,3}(\kappa)\circ g_{3,1}(\kappa)=id.$
\begin{prop}\label{conf_canonique}We have:
\begin{itemize}
\item[(i)] We have $g_{1,2}=g$. Hence \ $g_{1,2}:\ (y_n,z_n )\rightarrow (y_{n-2},z_{n-2}).$
\item[(ii)] The automorphism $g_{2,3}(\kappa)$ and its inverse are given in canonical coordinate systems by
\begin{equation*}
\begin{aligned}
g_{2,3}(\kappa):\ (z_1,y_2)\rightarrow (\kappa^2z_1y_2^{-2},(1+\kappa^2e_0^{-1}z_1y_2^{-2}) y_2),\\
g_{3,2}(\kappa):\ (y_1,z_1)\rightarrow (y_1+e_0\kappa^{-2}z_1y_1^{-1},e_0^2\kappa^{-2}z_1y_1^{-2}).
\end{aligned}
\end{equation*}
\item[(iii)] The automorphism $g_{3,1}(\kappa)$ and its inverse are given in canonical coordinate systems by
\begin{align*}
g_{3,1}(\kappa):\ (z_2,y_3)\rightarrow (\kappa^2 z_2y_3^{-2},(1+\kappa^2e_0^{-1}z_2y_3^{-2}) y_3),\\
g_{1,3}(\kappa):\ (y_2,z_2)\rightarrow (y_2+e_0\kappa^{-2}z_2y_2^{-1},e_0^2\kappa^{-2}z_2y_2^{-2}).
\end{align*}
\end{itemize}
\end{prop}  

\begin{proof}\\
(i) We have $g=\sigma_1\circ \sigma_2$, with
\begin{equation*}
\begin{aligned}
&\sigma_1(x_1,x_2,x_3)=(x_1-F_{V,x_1},x_2,x_3)=(-x_1-x_2x_3+\theta_1,x_2,x_3)\\
&\sigma_2(x_1,x_2,x_3)=(x_1,x_2-F_{V,x_2},x_3)=(x_1,-x_2-x_1x_3+\theta_2,x_3).
\end{aligned}
\end{equation*}
Therefore,
\begin{equation*}
\begin{aligned}
(\sigma_1\circ \sigma_2)^*x_1&=\sigma_2^*(-x_1-x_2x_3+\theta_1)\\
&=-x_1+x_2x_3+x_1x_3^2-\theta_2x_3+\theta_1=(g_{1,2})^*x_1.   \\
(\sigma_1\circ \sigma_2)^*x_2&=\sigma_2^*(x_2)\\
&=-x_2-x_1x_3+\theta_2=(g_{1,2})^*x_2. 
\end{aligned}
\end{equation*}
Hence, $g=g_{1,2}$. \bigskip\\
(ii) From Proposition (\ref{dyn conf}) we have:
\begin{align*}
g_{2,3}(\kappa)^*z_1&=e_0x_2^{-1}(x_2-\kappa^2x_2^{-1}+e_0^{-1}\kappa^2x_1)-e_0\\
&=e_0+\kappa^2x_2^{-2}(x_1x_2-e_0)-e_0\\
&=\kappa^2z_1y_2^{-2}.\\
g_{2,3}(\kappa)^*y_2&=g_{2,3}(\kappa)^*((z_1+e_0)y_1^{-1})\\
&=(\kappa^2z_1y_2^{-2}+e_0)e_0^{-1}y_2\\
&=y_2(1+e_0^{-1}\kappa^2z_1y_2^{-2}).
\end{align*}
We obtain $g_{3,2}(\kappa)$ by reversing the map $(z_1,y_2)\rightarrow (Z_1,Y_2)$:
\[z_1=\kappa^{-2}Y_2^2(1+e_0^{-1}Z_1)^{-2},\ y_2=Y_2(1+e_0^{-1}Z_1)^{-1},\]
and by using the change of canonical coordinates $(z_1,y_2)\rightarrow (y_1=(z_1+e_0)y_2^{-1},z_1).$
\bigskip\\
(iii) We have $g_{1,3}(\kappa)=g\circ g_{2,3}(\kappa)$.
We set $Z_2:=\kappa^2z_2y_3^{-2}$, 
$Y_3:=y_3\left(1+\kappa^2e_0^{-1}z_2y_3^{-2}\right)$. We have to prove that  $g_{1,3}(\kappa)^*Z_2=z_2$ and $g_{1,3}(\kappa)^*Y_3=y_3.$
\begin{align*}
g_{1,3}(\kappa)^*Z_2&=g_{2,3}(\kappa)^*\circ g^*(\kappa^2z_2y_3^{-2})=g_{2,3}(\kappa)^* (\kappa^2z_0y_1^{-2}).\\
g_{1,3}(\kappa)^*Y_3&=g_{2,3}(\kappa)^*\circ g^* Y_3=g_{2,3}(\kappa)^*\left(y_1(1+\kappa^2e_0^{-1}z_0y_1^{-2})\right).
\end{align*}
We compute $g_{2,3}(\kappa)^*z_0$ and $g_{2,3}(\kappa)^*y_1$. We have $z_0=z_1^{-1}Q_1(y_1)$, therefore~:
\[g_{2,3}(\kappa)^*z_0=\left(g_{2,3}(\kappa)^*z_0\right)^{-1}
Q_1\left(g_{2,3}(\kappa)^*y_1\right)
=\kappa^{-2}z_1^{-1}y_2^2Q_1\left(e_0y_2^{-1}\right).
\]
We recall   (cf. the remark \ref{Q1Q2}) that 
$Q_1(e_0t^{-1})=e_0^{-2}t^4Q_2(t)$. We obtain~:
\[
g_{2,3}(\kappa)^*z_0=
\kappa^{-2}e_0^2y_2^{-2}z_1^{-1}
e_0^2y_2^4Q_1\left(e_0y_2^{-1}\right)
=\kappa^{-2}e_0^2y_2^{-2}z_1^{-1}Q_2(y_2)
=e_0^2\kappa^{-2}y_2^{-2}z_2.
\]
Since $y_1=(z_1+e_0)y_2^{-1}$, by using (ii) we have
\[g_{2,3}(\kappa)^*y_1= (\kappa^2z_1y_2^{-2}+e_0)(y_2+\kappa^2e_0^{-1}z_1y_2^{-1})^{-1}=e_0y_2^{-1}.\]
Finally,
\begin{align*}
g_{1,3}(\kappa)^*Z_2&=g_{2,3}(\kappa)^* (\kappa^2z_0y_1^{-2})\\
&=\kappa^2(e_0^2\kappa^{-2}y_2^{-2}z_2)(e_0y_2^{-1})^{-2}=z_2.\\
g_{1,3}(\kappa)^*Y_3&=g_{2,3}(\kappa)^*\left(y_1(1+\kappa^2e_0^{-1}z_0y_1^{-2})\right)\\
&=e_0y_2^{-1}\left(1+\kappa^2e_0^{-1}(e_0^2\kappa^{-2}z_2y_2^{-2})(e_0y_2^{-1})^{-2}\right) \\
&=e_0y_2^{-1}(1+e_0^{-1}z_2)
=e_0y_2^{-1}\left(1+e_0^{-1}(y_2y_3-e_0)\right)=y_3.
\end{align*}
Therefore, $g_{3,1}(\kappa):\ (z_2,y_3)\rightarrow (Z_2,Y_3).$ The computation of its inverse $g_{3,1}(\kappa)$ in the chart $(y_2,z_2)$
is similar to the computation of the inverse of the $g_{2,3}(\kappa)$ in point (ii).\end{proof}

\begin{rema} The transformations $g_{2,3}(\kappa)$ and $g_{3,1}(\kappa)$ are conjugated by the map $\xi:\ (y_1,z_2)\rightarrow (y_2,z_3)$.
Nevertheless, this conjugation $\xi$ does not belong to the canonical dynamics. \end{rema}

The following equalities are consequences of the points (ii) and (iii) in Proposition (\ref{conf_canonique}):

\begin{coro}\label{tores_et_conf} For all $\kappa$ in $\C^*$ we have:
\begin{enumerate}
\item $g_{2,3}(\kappa)=g_{2,3}(1)\circ t_{y_2}(\kappa^2)=t_{y_1}(\kappa^{-2})\circ g_{2,3}(1);$
\item $g_{1,3}(\kappa)=g_{1,3}(1)\circ t_{y_2}(\kappa^2)=t_{y_3}(\kappa^{-2})\circ g_{1,3}(1).$
\end{enumerate}
In particular, $g_{2,3}(\kappa)$ conjugates $T_{y_1}$ and $T_{y_2}$: 
$g_{2,3}(\kappa)\circ t_{y_2}(\kappa^2)\circ g_{2,3}(\kappa) =t_{y_1}(\kappa^{-2}),$
and $g_{3,1}(\kappa)$ conjugates $T_{y_2}$ and $T_{y_3}$: 
$g_{3,1}(\kappa)\circ t_{y_3}(\kappa^2)\circ g_{1,3}(\kappa) =t_{y_1}(\kappa^{-2}).$
\end{coro}

From these preliminary results we immediately obtain:

\begin{prop} 
\begin{enumerate}
\item The tame dynamic generated by $g=\sigma_1\circ\sigma_2$ is included in $Conf(P_V)$.
\item For any $n$ in \Z,\ $T_{y_n}\subset Conf(P_V)$.
\end{enumerate}
\end{prop}
\begin{proof}
The first point is a direct consequence of $g=g_{1,2}$.\\
The second one is obtained for $n=1$ or $n=2$ from the first item of Corollary (\ref{tores_et_conf}). The others tori are obtained from 
$T_{y_1}$ or $T_{y_2}$ by a conjugation with an element of the tame dynamic. \end{proof}

\bigskip

Nevertheless, the canonical Stokes operators do not belong to $Conf(P_V)$. The reason is the following one: both $g_{2,3}(\kappa)$ and $s_1$ keep invariant $\Delta:\ z_1=0$. The restriction of $s_1$ to each component of $\Delta$ is a translation: see Proposition (\ref{properties stokes}). One can compute the restriction of $g_{2,3}(\kappa)$ on each each line of $\Delta$. 
\[g_{2,3}(\kappa)^*x_3=\kappa^{-2}x_2^2x_3+\beta(\kappa,x_1,x_2),\]
where $\beta(\kappa,x_1,x_2)$ only depends on $\kappa$, $x_1$, and $x_2$. On $\Delta$, $x_1=e_3^{\pm 1}, e_0e_4^{\pm 1}$, $x_2=e_0e_3^{\pm 1}, e_4^{\pm 1}$, therefore the restriction of the $g_{2,3}(\kappa)$ are affine transformations on each line. The same property holds for $g_{3,1}(\kappa)$. Nevertheless, we cannot find $\kappa$ in $\C^*$ such that $\kappa^{-2}x_2^2=1$ simultaneously on each component of $\Delta$.

This remark suggests to introduce an extension of the confluent dynamics. The only way to obtain from the family $g_{2,3}(\kappa)$ an element which is a translation on $\Delta$ is to introduce a \textsl{functional time} for the elements of $T_{t_2}$, setting
$\kappa=x_2$. Therefore if we expect to obtain the canonical stokes operators from the confluent dynamic, we have to extend it by the element $t_{y_2}(y_2^{-2})$. 

Another motivation in order to introduce this element is purely algebraic and makes use of the structure induced by the symplectic Cremona group. We recall the subgroups of $Bir(\mathcal{C}_V)$ generated from the Borel-Cartan structure of $Symp(\C^2)$ by the canonical triple $(y_1,z_1,y_2)$~:
\begin{align*}
&B_{y_1}=\{b_{y_1}(\lambda,r):\ (y_1,z_1)\rightarrow(\lambda y_1,z_1r(y_1)),\ r\in\C(y_1)^*,\ \lambda\in\C^*\} \\
&B_{y_2}=\{b_{y_2}(\lambda,r):\ (z_1,y_2)\rightarrow(z_1r(y_2),\lambda^{-1} y_2),\ r\in\C(y_2)^*,\ \lambda\in\C^*\} \\
&T_{y_1}=\{t_{y_1}(\lambda):\ (y_1,z_1)\rightarrow(y_1,\lambda z_1),\ \lambda\in\C^*\}\\
&T_{y_2}=\{t_{y_2}(\lambda):\ (z_1,y_2)\rightarrow(\lambda^{-1} z_1,y_2),\ \lambda\in\C^*\}\\
&U_{z_1}=\{u_{z_1}(r):\ (z_1,y_2)\rightarrow(z_1,y_2r(z_1)),\ r\in\C(z_1)^*,\ r(0)=1\}.
\end{align*}
and we set~: 
\begin{align*}
T_{y_2}\left(\C(y_2)\right)&=\{(z_1,y_2)\rightarrow(z_1r(y_2),y_2),\ r\in\C(y_2)^*\}\\
&=\{b_{y_2}(1,r),\ r\in\C(y_2)^*\}\subset B_{y_2};\\
T_{y_1}\left(\C(y_1)\right)&=\{(z_1,y_2)\rightarrow(y_1,z_1r(y_1)),\ r\in\C(y_1)^*\}\\
&=\{b_{y_1}(1,r),\ r\in\C(y_1)^*\}\subset B_{y_1}.
\end{align*}
\begin{prop}\label{decomposition} \ 
\begin{itemize}
\item[(i)]
There exists a unique pair $(b_{y_2},u_{z_1})$ in $T_{y_2}\left(\C(y_2)\right)\times U_{z_1}$, such that~: \\
 $g_{2,3}(1)=u_{z_1}\circ b_{y_2}$. Moreover
$b_{y_2}=t_{y_2}(y_2^{-2})\in T_{y_2}\left(\C(y_2)\right)$ and $u_{z_1}=s_1^{-1}$.
\item[(ii)]
There exists a unique pair $(b_{y_1},v_{z_1})$ in $T_{y_1}\left(\C(y_1)\right)\times U_{z_1}$, such that~: \\
 $g_{3,2}(1)=v_{z_1}\circ b_{y_1}$. Moreover
$b_{y_1}=t_{y_1}(e_0^{-2}y_1^2)
\in T_{y_1}\left(\C(y_1)\right)$ and $v_{z_1}=s_1$.
\end{itemize}
\end{prop}

\begin{proof}
(i) We set~:
\begin{align*}
&b_{y_2}:\ (z_1,y_2) \mapsto (z_1y_2^{-2},y_2),\\
&u_{z_1}:\ (z_1,y_2) \mapsto (z_1,(1+e_0^{-1}z_1)y_2).
\end{align*}
We have $g_{2,3}(1)=u_{z_1}\circ b_{y_2}$. The unicity of this decomposition is a consequence of $T_{y_2}\left(\C(y_2)\right)\cap U_{z_1}=\{id\}.$
One recognizes here that $b_{y_2}=t_{y_2}(y_2^{-2})$ and $u_{z_1}=s_1^{-1}$.
\bigskip\\
(ii) We set~:
\begin{align*}
&b_{y_1}:\ 
(y_1,z_1) \mapsto  (y_1,e_0^2z_1y_1^{-2})\\
&v_{z_1}:\ (y_1,z_1) \mapsto  \left((1+e_0^{-1}z_1)y_1,z_1\right).
\end{align*}
We have $g_{3,1}(1)=v_{z_1}\circ b_{y_1}$. The unicity of this decomposition is a consequence of $T_{y_1}\left(\C(y_1)\right)\cap U_{z_1}=\{id\}.$
One recognizes here that $b_{y_1}=t_{y_1}(e_0^{-2}y_1^2)$ and $v_{z_1}=s_1$.
\end{proof}
\bigskip\\
By using $g_{2,3}(\kappa)=g_{2,3}(1)\circ t_{y_2}(\kappa^2)$, and $Z(T_{y_2})=B_{y_2}$, we obtain:
\[g_{2,3}(\kappa)=u_{z_1}\circ b_{y_2}\circ t_{y_2}(\kappa^2)=u_{z_1}\circ t_{y_2}(\kappa^2)\circ b_{y_2}.\]
Since $g_{3,1}(\kappa)$ is conjugated to $g_{2,3}(\kappa)$ by $(z_1,y_2)\rightarrow (z_2,y_3)$ we have a similar decomposition of the family $g_{3,1}(\kappa)$ in $U_{z_2}\times B_{y_3}^1\times T_{y_3}$:
\[g_{3,1}(\kappa)=u_{z_2}\circ b_{y_3}\circ t_{y_3}(\kappa^2)=u_{z_2}\circ t_{y_3}(\kappa^2)\circ b_{y_3}, \ u_{z_2}=s_2^{-1},\ b_{y_3}=t_{y_3}(y_3^{-2}).\]
Notice that $b_{y_2}^{-1}=t_{y_2}(y_2^2):\ (z_1,y_2)\rightarrow (z_1y_2^2,y_2)$ is a ramified blowing-up in the canonical chart $(z_1,y_2)$. We extend the confluent dynamic by this element:
\begin{defi} The extended confluent dynamic is defined by
\[Conf^\sharp(P_V)=<Conf(P_V),t_{y_2}(y_2^2)>.\]
\end{defi}
We recall that, according to Proposition (\ref{mon_can_formelle}), the element $\widehat{m}$ of $Dyn(\mathcal{C}_V)$ such that $g=s_2\circ\widehat{m}\circ s_1$ is defined by $\widehat{m}=t_{y_2}(e_0^{-2}y_2^4)$. We consider a square root of $\widehat{m}$ defined by:
\[\widehat{m}^{1/2}:\ (y_2,z_2)\rightarrow (y_2,e_0^{-1}y_2^{2}z_2).\]
\begin{defi} The extended confluent canonical dynamic is defined by
\[Dyn^\sharp(\mathcal{C}_V)=<Dyn^\sharp(\mathcal{C}_V), \widehat{m}^{1/2}>.\]
\end{defi}
\begin{theorem} $Conf^\sharp(P_V)=Dyn^\sharp(\mathcal{C}_V).$ \end{theorem}
\begin{proof}. We first prove that $Dyn(\mathcal{C}_V)^\sharp\subset Conf^\sharp(P_V)$. 
$Dyn^\sharp(\mathcal{C}_V)$ is generated by $g$, $T_{y_1}$, $s_1$, $s_2$ and $\widehat{m}^{1/2}$. From Proposition (\ref{decomposition}), since $s_1^{-1}=g_{2,3}(1)\circ t_{y_2}(y_2^2)$, $Conf^\sharp(P_V)$ contains $s_1$.
We have:
\[\widehat{m}=t_{y_2}(e_0^{-2}y_2^4)=\left(t_{y_2}(e_0^{-1})\circ t_{y_2}(y_2^2)\right)^{\circ 2}.\]
Hence, $Conf^\sharp(P_V)$ also contains $\widehat{m}$. According to the relation $g=s_2\circ\widehat{m}\circ s_1$, $s_2$ also belongs to 
$Conf^\sharp(P_V)$. Finally since $\widehat{m}^{1/2}=t_{y_2}(e_0)\circ t_{y_2}(y_2^2)$, 
we obtain the inclusion $Dyn(\mathcal{C}_V)^\sharp\subset Conf^\sharp(P_V)$.\\
Now we prove that $Conf^\sharp(P_V)\subset Dyn(\mathcal{C}_V)^\sharp$. The relation
\[g_{2,3}(\kappa)=s_1^{-1}\circ t_{y_2}(y_2^{-2})\circ t_{y_2}(\kappa^2)\]
proves that $g_{2,3}(\kappa)$ belongs to $Dyn(\mathcal{C}_V)$ for any $\kappa$ in $\C^*$. Since $g_{1,2}=g$ belongs to $Dyn(\mathcal{C}_V)$, from the relation $g_{1,2}\circ g_{2,3}(\kappa)\circ g_{3,1}(\kappa)=id$, $g_{3,1}(\kappa)$ also belongs to $Dyn(\mathcal{C}_V)$ for any $\kappa$ in $\C^*$. Finally $t_{y_2}(y_2^2)=\widehat{m}^{1/2}\circ t_{t_2}(e_0^{-1})$ belongs to $Dyn(\mathcal{C}_V)^\sharp$.\end{proof}

\subsection{Comparison with the wild dynamics}
The wild dynamics is a pseudogroup of non linear dynamics induced by the Painlevé V foliation $P_V(\kappa)$ in a neighborhood of each singular point of saddle-node type. We will recall here the construction of its generators: the non linear stokes operators, non linear tori, and formal and analytic non linear monodromy. Through the Riemann-Hilbert map $RH_V$ it induces an (a priori) local dynamics on $\mathcal{C}_V(\theta)$. The second author formulates the Wuhan conjecture which claims that this dynamics extends into a symplectic rational dynamics. This conjecture has been proved for the Painlevé V equation by M. Klimes in \cite{Kl1} (there is also a very clear presentation of this work in Klimes's lecture \cite{Kl3}). His method uses a description of the confluence on the foliations of the Hamiltonian systems on the left side and on the linear isomonodromic systems and on the associated character variety on the right side. An essential tool is the \emph{discretization} of the confluence as indicated in section \ref{irregular sing} on the baby model of confluence $x(x-\varepsilon) \rightarrow x^2y'+y=0$, first introduced by C. Zhang for  the confluence for the hypergeometric equations \cite{Zha}. We will present here this result of M. Klimes, and compare this dynamics with the canonical dynamics $Dyn(\mathcal{C}_V)$ on the cubic surfaces.
\bigskip\\
\textbf{Definition of the wild dynamics.} The Painlevé V foliation is given under its hamiltonian form by (\ref{HV}). From section \ref{SingPV}, the irregular singular points $s_i$ all appear in the Boutroux chart. Around each irregular singular point of Painlevé V of saddle-node type $s_i$ the formal and sectoral normal forms are given by \cite{Kl1}:

\begin{theorem}\label{thformalnormalformp5} Let $\alpha_0=2\alpha_1+\alpha_2-1$.

1- In a neighborhood of a saddle-node $s_i$, the hamiltonian system (\ref{HV}) can be brought to a formal normal form~:
\begin{equation}
\label{hamsym5infty2}
x^2\frac{du}{dx}=(1-\alpha_0x+4xu_1u_2)\begin{pmatrix}
1 & 0 \\
0 & -1
\end{pmatrix}u, \quad
u=\begin{pmatrix}
u_1 \\
u_2
\end{pmatrix}
\end{equation}
by means of a formal transversally symplectic change of coordinates~:
\[
\begin{pmatrix}
q \\
p
\end{pmatrix}
:=\widehat{\Psi}(u,x)=\sum_{k\geq 0} \psi^{(k)}(u)x^k
\]
where the $\psi^{(k)}$ are analytic on a polydisc $P=\{\vert u_1\vert,\vert u_2\vert<\delta\}$
($\delta>0$). 

2- The formal normal form (\ref{hamsym5infty2}) is integrable in closed form~:
\begin{equation}
\label{equaclosedformp5}
\begin{cases}
u_1(x,c_1)=c_1e^{-1/x}x^{-\alpha_0+4c_1c_2}\\
u_2(x,c_2)=c_2e^{1/x}x^{\alpha_0-4c_1c_2}.
\end{cases}
\end{equation}
and this local hamiltonian vector field (for $du_1\wedge du_2$ and $h=x^{-2}(1-\alpha_0 x)u_1u_2$, also admits the analytic first integral $u_1u_2$.

3- The formal series 
$\widehat{\Psi}\in\mathcal{O}(P)[[x]]$ is \emph{uniformly} $1$-summable with a pair of $1$-sums 
$\Psi^\uparrow(u,x)$, $\Psi^\downarrow(u,x)$, defined respectively above the sectors
\[
U^\uparrow:=\{\vert \arg x -\pi/2\vert < \pi -\eta, \vert x\vert < \delta'\} ~~\text{and}  ~~ 
U^\downarrow:=\{\vert \arg x -\pi/2\vert < \pi -\eta, \vert x\vert < \delta'\} 
\]
for some $0<\eta<\pi/2$ arbitrary small and some $\delta'>0$ (depending on $\eta$), and 
$u\in P$.
\end{theorem}

\begin{rema} The formal Takano's normal form \cite{Taka83} used here to define the non linear Stokes operators will not suffice for the other Painlevé equations in order to obtain overlapping open sectors. We will have to make use of the Bittman's normal forms, which are no longer polynomials: see \cite{Bitt}, \cite{Bitt1} and \cite{Bitt3}.
\end{rema}

\begin{coro} On each sector $U^\bullet$ ($\bullet=\uparrow$ or $\downarrow$),
\begin{enumerate}
\item the system (\ref{HV}) has a unique analytic bounded solution $f^\bullet$
which is the $1$-sum of the formal solution $\widehat{f}:\ u_1(x,0)=u_2(x,0)=0$.
We call these solutions the sectoral center solutions. They are pole-free on the corresponding sector and they are characterized by this property. 
\item the system (\ref{HV}) has a 2-parameters family of solutions $\widehat{f}_{c_1,c_2}$:
\[(q^\bullet,p^\bullet)(x,c_1,c_2)=\Psi^\bullet(u(x,c_1,c_2),x).\]
\end{enumerate}
\end{coro}

We consider the left and right intersection sectors of the overlapping sectors $U^\uparrow$ and $U^\downarrow$
\[V^l:=U^\uparrow \cap U^\downarrow \cap \{ \Re x<0\} ~~\text{and}  ~~ V^r:=U^\uparrow \cap U^\downarrow \cap \{ \Re x>0\}.\]
The non linear Stokes multipliers are defined by 
\[S_2=(\Psi^\uparrow)^{-1}\circ \Psi^\downarrow\mbox{ on }V^r, \ S_1=(\Psi^\uparrow)^{-1}\circ \Psi^\downarrow\mbox{ on }V^l.\]
The \textsl{formal non linear monodromy} is defined by the action induced by $x\mapsto e^{2i\pi}x$ on the space of formal solutions $\widehat{f}_{c_1,c_2}$: 
\[\widehat{N}:\ (c_1,c_2)\mapsto(e^{2i\pi(-\alpha_0+4c_1c_2)}c_1,e^{2i\pi(\alpha_0-4c_1c_2)}c_2).\]
The \textsl{formal non linear exponential torus} is defined by the analytic symplectic symmetries of the formal normal forms (\ref{equaclosedformp5}):
\[t_\alpha: (c_1,c_2) \mapsto \left(e^{\alpha(c_1c_2)}c_1,e^{-\alpha(c_1c_2)}c_2\right), \quad \alpha\in \mathcal{O}(\C,0).\]
Using $\Psi^\uparrow$ and $\Psi^\downarrow$, the formal exponential torus induces two sectoral exponential tori $T^\uparrow$ and $T^\downarrow$ and the formal first integral $h=c_1c_2$ induces 2 sectoral first integrals $h^\uparrow$ and $h^\downarrow$.
\begin{defi}
The wild dynamics $Wild(P_V,f^\uparrow)$ based in the central solution $f^\uparrow$ is the pseudo-group generated by $S_1$, $S_2$, $\widehat{N}$, and $T=\{t_\alpha\}$. 
\end{defi}
$Wild(P_V,f^\uparrow)$ also contains the actual monodromy $N$ generated by a loop around $x=\infty$, according to the relation $S_2\circ \widehat{N}^\uparrow\circ S_1=\widehat{N}$, where $\widehat{N}^\uparrow=\Psi^\uparrow\circ \widehat{N}\circ (\Psi^\uparrow)^{-1}$.
\bigskip\\
\textbf{The wild dynamics through the Riemann-Hilbert map }$RH_V$. We fix a value $x_0$ in a neighborhood of $\infty$, and we consider the Okamoto space of initial condition $\mathcal{V}_{x_0}$ over $x_0$. Let 
\[m^\bullet=(q^\bullet,p^\bullet)(x,0,0)\]
be the two points on $\mathcal{V}_{x_0}$ corresponding to the two central varieties in a neighborhood of a singular point $s$. We denote by
$(c_1^\bullet,c_2^\bullet)(q,p)$ the inverse maps of $(q^\bullet,p^\bullet)(x_0,c_1,c_2)$. The equations $c_2^\bullet=0$ define two germs of curves $\delta^\bullet$ in $m^\bullet$. The equations $c_1^\bullet=0$ define two germs of curves $d^\bullet$ in $m^\bullet$. 
On the rightside, we consider the 12 lines $\Delta_\cdot$, $D^r_\cdot$, $D_l\cdot$ in the configuration of lines in $\chi_V$: see figure \ref{droites-chiV}. We can group them in four triples $(\Delta_i,D_{r_i},D_{l_i})$ such that $\Delta_i\cap D_{r_i}\neq\emptyset$ and $\Delta_i\cap D_{l_i}\neq\emptyset$. We set $p_{r_i}=\Delta_i\cap D_{r_i}$, $p_{l_i}=\Delta_i\cap D_{l_i}$.

Now we consider the confluent process between $P_{VI}$ and $P_V$ defined by
\begin{equation}
\label{equavariablesconf}
 \quad t_{VI}=1+\varepsilon t_V  , \quad \alpha_3=1/\varepsilon, \quad
 \alpha_{2,VI}=\tilde\alpha_2=-\frac{1}{\varepsilon}+ \alpha_{2,V},
 \quad x=\frac{1}{t_V}+\varepsilon, \alpha_{1,VI}=\alpha_{1,V}
\end{equation}
which sends the three fixed singularities to $t_V=-1/\varepsilon,~0,\infty$. 
This change of variables transforms $\varepsilon H_{VI}$ into $H_{VI}^{\varepsilon}$ and the $P_{VI}$ Hamiltonian system into~:
\[
\frac{dq}{dt}=\frac{\partial H_{VI}^{\varepsilon}}{\partial p}, ~~  
\frac{dp}{dt}=-\frac{\partial H_{VI}^{\varepsilon}}{\partial q},
\]
When $\varepsilon \rightarrow 0$, the simple singular points $-1/\varepsilon$ and $\infty$ merge into a double singular point, an irregular singularity at infinity, and the limit of $H_{VI}^{\varepsilon}$ is $H_V$. The four pairs of singularities over $-1/\varepsilon$ and $\infty$ merges to 4 saddle-nodes (the confluent saddle-nodes) among the five saddle nodes $s_i$ over $\infty$. In what follows, we suppose that $s$ is one of these 4 singularities. We summarize the results of Martin Klimes \cite{Kl1} by the following theorem:

\begin{theorem} We consider the Riemann-Hilbert map $RH_V$ between an open set in the space of initial condition $\mathcal{V}_{x_0}$ induced by a neighborhood of a confluent singularity $s$ and the character variety $\chi_V$. The choice of $s$ determines a triple of lines $(\Delta,D_r,D_l)$  among the four triples $(\Delta_i,D^r_i,D^l_i)$ such that
\begin{enumerate}
\item $RH_V(m^\uparrow)=\Delta\cap D_r=p_r$, $RH_V(m^\downarrow)=\Delta\cap D_l=p_l$.
\item $RH_V(\delta^\uparrow)=(\Delta,p_r)$ (the germ of line at $p_r$), $RH_V(\delta^\downarrow)=(\Delta,p_l)$. Furthermore, the germs
$\delta^\uparrow$ and $\delta^\downarrow$ extend to a same analytic curve $\delta$ in $\mathcal{V}_{x_0}$ isomorphic to $\C$, such that $RH_V(\delta)=\Delta$.
\item $RH_V(d^\uparrow)=(D_r,p_r)$, $RH_V(d^\downarrow)=(D_l,p_l)$. Furthermore the germs of curves $d^\uparrow$ and $d^\downarrow$ extends to 2 curves $d_r$ and $d_l$ isomorphic to $\C$ in $\mathcal{V}_{x_0}$, such that $RH_V(d_r)=D_r$ and $RH_V(d_l)=D_l$.
\item Let $(y_1,z_1,y_2)$ be the triple of canonical coordinates (\ref{z1}). We have:
\[RH_V^*y_1= y_1(p_r)e^{h^\uparrow}, \ RH_V^*y_2=y_2(p_r)e^{-h^\downarrow},\ RH_V^*z_1=e_0(e^{(h^\uparrow-h^\downarrow)}-1).\]
\item Let $g$, $s_1$, $s_2$, be the generators of the canonical dynamics defined in \ref{can_dyn}, let $\widehat{m}$ be the canonical formal monodromy defined by \ref{mon_can_formelle} and $T_{y_1}$, $T_{y_2}$ the canonical exponential tori. We have:
\[RH_V^*g=N,\ RH_V^*S_1=s_1,\ RH_V^*S_2=s_2,\ RH_V^*\widehat{m}=\widehat{N},\ RH_V^*T^\downarrow=T_{y_1},\ RH_V^*T^\uparrow=T_{y_2}.\]
\end{enumerate}
\end{theorem}
Therefore the dynamics induced by $Wild(P_V,f^\uparrow)$ through the Riemann-Hilbert correspondence $RH_V$ extends to a rational symplectic dynamics $Wild(P_V)$ on $\chi_V$, and we have $Wild(P_V)=Dyn(\mathcal{C}_V)$. This proves the Wuhan conjecture of the second author for the $P_V$ foliation.
\bigskip\\
The main ingredients of the proof of this result in \cite{Kl1} are:

- an unfolded version of Theorem \ref{thformalnormalformp5} for the hamiltonian system related to $H_{VI}^{\varepsilon}$;

- a discretization of the confluent parameter by setting $\varepsilon^{-1}=\varepsilon_0^{-1}+n$, $n\in\Z^+$, along which the monodromies of the unfolded system, a priori divergent, converge. The choice of this sequence depends on a parameter $\kappa=e^{2i\pi/\varepsilon}$.

- a confluence on the rightside between $\chi_{VI}$ and $\chi_V$ which allows to compare these limits to the wild dynamics $Wild(P_V,f^\uparrow)$.

\section{Conclusion and open questions}

\textbf{Conclusion. }Several dynamics related to the Painlevé V foliation through the Riemann-Hilbert map $RH_V$ has been described here on the wild character variety $\chi_V(a)$ identified to the cubic symplectic surface $\mathcal{C}_V(\theta)$ for generic parameters $\theta$. The tame dynamics is induced here by only one braid. We first extended the tame dynamics by using a confluent dynamics obtained from the Painlevé VI dynamics by a birational confluent morphism between $\chi_{VI}$ and $\chi_V$. This one, first discovered by M. Klimes, is obtained here by a confluent process between the corresponding groupoids (after an extension with families of exponential loops). 

We have constructed the canonical symplectic birational dynamics defined on the cubic symplectic surface $\mathcal{C}_V(\theta)$. This one is the pullback of symplectic dynamics on $\C^2$ by a canonical sequence of coordinates satisfying cluster-type relations. 
This birational dynamics coincides with the dynamics obtained by M. Klimes from the wild dynamics defined by a confluent irregular singular point of the foliation. After and extension by one element, confluent dynamics and canonical dynamics also coincide.

The cubic symplectic surface $\mathcal{C}_V(\theta)$ contains a \textit{skeleton} generated by 18 lines and their images through the action of the tame dynamics. We have characterized this set by a condition of reducibility of the corresponding linear representations. Some intersections between these lines correspond to particular solutions of the Painlevé V equation : Kaneko solutions, or central solutions.
The restriction of the dynamics on this skeleton is an important tool in the comparisons between these dynamics.
\bigskip\\
\textbf{Open problems.} The equations  of the character varieties $\chi_J$ for the other Painlevé equations $P_J$, $J=V_{deg}$, $IV$, $III(D_6)$,  $III(D_7)$, $III(D_8)$, $II(JM)$, $II(FN)$, $I$ (with the notations of \cite{VdPSai}) are known. We already know that each of them can be obtained as a quotient of a space of linear representations of a wild groupoid, and that there exists canonical cluster sequences of coordinates on each $\chi_J$, inducing a canonical birational symplectic dynamic $Dyn_J$ which can be explicitely computed. We conjecture that for all $J$:
\medskip\\
\textsl{
- All the lines in $\chi_J$ are a reducibility locus of some path in the groupoid;\\
- the wild dynamics of all the Painlevé equations (induced by non linear monodromies, non linear Stokes operators and exponential tori) around an irregular singular point obtained by a confluent process coincide with these canonical birational symplectic dynamics $Dyn_J$ on $\chi_J$.}
\medskip\\
We already know that these dynamics coincide on the skeleton generated by the lines for $P_I$ and $P_{II}$.
With M. Klimes we also conjecture that\medskip\\
\textit{- There is a diagram of families of birational symplectic confluences between the $\chi_J$, which parallels the diagram of confluence of Ohyama and Okumura \cite{OhOk}}.
\medskip\\
We expect that we can prove this fact by a confluent process between the groupoids (extended by families of exponential loops).
\medskip\\
Finally we mention here the initial motivation of this study. In \cite{CanLo} S. Cantat and F. Loray proved the irreducibility of $P_{VI}$ for generic parameters, by using the Malgrange closure of its dynamics \cite{Mal}. 
The aim of the Wuhan conjectures of the second author was to extend their proof in the general case. In the $P_{VI}$ case, the dynamics obviously belongs to the Malgrange groupoid which is a closure of its holonomy pseudogroup. In the general case, we conjecture that \textsl{the wild dynamics is contained in the Malgrange pseudogroup}, but now it is far from obvious. Modulo this conjecture, we would obtain a proof of the irreducibility of the Painlevé V equation for generic values of the parameters. Indeed since the 
extended canonical dynamic contains the confluent dynamic, it contains a copy of the dynamics of a $P_{VI}$ and we can use Theorem $D$ of \cite{CanLo} in order to prove that its Malgrange closure is maximal.

\section*{Appendix. Structure of the symplectic Cremona group}
\addcontentsline{toc}{section}{Appendix. Structure of the symplectic Cremona group}
The Cremona group $\mathit{Cr}=\mathit{Cr}_2$ is the group of birational automorphisms of $\p_2(\C)$. It does not depend on the representative $X$ in the birational class of $\p_2$: $\C^2$, $\p_1\times \p_1$...\\
The symplectic Cremona group is the subgroup $\mathit{Symp}$ of $Cr\simeq Bir(P_1(\C)\times\p_1(\C))$ of the elements which preserve the differential form
$\omega=\frac{du}{u}\wedge\frac{dv}{v}$.
The Cremona group is an old topic which appears at the end of the XIX-th century. One can find an introductive text in \cite{Lam2}. For a study of its subgroups, see \cite{Des2}.
To the contrary, the study of its symplectic subgroup $\mathit{Symp}$ has been developed only since 2005, first by A. Usnich \cite{Us, Us2} and then J. Blanc \cite{Blanc}. We present here without proofs the main results about subgroups of $\mathrm{Symp}$. Some of them are new results. Our terminology is inspired by the algebraic groups, due to some similarities in this non linear infinite dimensional context. We denote
\[T=\{t(\lambda,\mu):\ (u,v)\mapsto(\lambda u,\mu v),\ (\lambda,\mu)\in\C^*\times\C^*\}.\]
We have $T\subset Symp$. We denote by $Z(T)$ its centralizor, and $N(T)$ its normalizor in $\mathit{Symp}$.
We consider the subgroup $W$ of $Symp$ of the monomial maps:
\[W=\{(u,v)\mapsto (u^av^b,u^cv^d),\ \left(\begin{array}{cc}a&b\\c&d\end{array}\right)\in SL_2(\Z)\}.\]
\begin{prop} We have:
\begin{enumerate}
\item $T$ is an algebraic abelian maximal subgroup of $\mathit{Symp}$;
\item $Z(T)=T$;
\item $N(T)=W\ltimes T$.  
\end{enumerate}
\end{prop}
\begin{proof}
$T$ is an algebraic abelian maximal subgroup of $Cr$ \cite{Se}, which preserve $\omega$.\\
The centralizor and normalizor of $T$ in $Cr$ has been computed by
\cite{Dem}, since $\C$ is algebraically closed. We have: $N_{Symp}(T)=N_{Cr}(T)\cap Symp=W$.
\end{proof}
\bigskip\\
Furthermore, according to a theorem of J. Blanc \cite{Blanc}, $Symp$ is generated by $N(T)$ and the order five element $p$ defined by:
\[(u,v)\mapsto\left(v,\frac{1+v}{u}\right).\]
\textbf{The Cartan subgroups.}
\begin{defi}
A Cartan subgroup of $\mathit{Symp}$ is an algebraic abelian maximal subgroup of rank two in $\mathit{Symp}$. 
 $T$ is the standard Cartan subgroup of $\mathit{Symp}$, and $W$ is the Weyl group associated to $T$.
\end{defi}
We have $T=T_1\times T_2$, where
\[T_1=\{t_1(\lambda)=t(\lambda,1),\ \lambda\in\C^*\},\ T_2=\{t_2(\mu)=t(1,\mu),\ \mu\in\C^*\}.\]
All the Cartan subgroups are conjugated. \footnote{We have no reference for this fact; the proof will be delivered in a forthcoming publication. This proof uses a result of J. Blanc \cite{Blanc}.}
\bigskip\\
\textbf{Borel subgroups.} 
Let $dJ_1$ be the subgroup of $Cr\simeq Bir(\p_1\times \p_1)$ of the \emph{De Jonquières maps}, i.e. the rational maps which preserve the first projection:
\[dJ_1=\{dj_1:\ (u,v)\mapsto (a(u),b(u,v)),\ a\in PGL_2(\C),\ b(\cdot,v)\in PGL_2(\C(u))\}.\]
Let $dJS_1$ be the subgroup in $\mathit{Symp}$ of the \emph{symplectic De Jonquières maps}.
For any $\lambda$ in $\C^*$, any $r$ in $\C(X)^*$, the maps:
\[dj_1(\lambda,r):\ (u,v)\mapsto (\lambda u, r(u)v),\ \sigma:\ (u,v)\mapsto (u^{-1},v^{-1})\] 
are examples of elements in $dJS_1$. We set
\[B_1=\{dj_1(\lambda,r),\ \lambda \in \C^*,\  r \in \C(u)^*\}=<T_1,T_2(\C(u))>,\ B_1^-=\{b_1\circ \sigma,\ b_1\in B_1\}.\]
\begin{prop} We have $dJS_1=B_1\cup B_1^-=B_1\rtimes \Z/2\Z$.
\end{prop}
We also define $dJ_2$, $dJS_2$, $B_2$ with respect to the second projection. We have $B_1\cap B_2=T$, and any pair $(wB_1w^{-1},w'B_2w'^{-1})$, $w,w'\in W$ also satisfies this equality.
\begin{defi} $(B_1,B_2)$ is the standard pair of Borel subgroups.\\ Any pair $(wB_1w^{-1},w'B_2w'^{-1})$, $w,w'\in W$ is a pair of Borel subgroups related to the standard Cartan subgroup $T$.
\end{defi}
Notice that the Borel subgroups are no longer algebraic subgroups since there are infinite dimensional groups. It can be proved that $B_1$ and $B_2$ generate $\mathit{Symp}$.
\begin{defi}  The subgroup $U_1$ of $B_1$ of the ``\textit{unipotent}'' elements is defined by:
\[U_1=\{dj_1(1,r), \ r(0)=1\}.\]
\end{defi}
\begin{prop} We have:
\begin{enumerate}
  \item $U_1$ is an abelian subgroup;
	\item $U_1$ is generated by the family $dj_1(1,1+\nu u)$, $\nu \in\C^*$, or by one element $dj_1(1,1+ u)$ and its conjugations with an element of $T_1$ (we say that this element is a pseudo generator of $U_1$);
  \item $[B_1,B_1]=\{dj_1(1,r), \ r(0)=1,\ \deg(r)=0.\}$ Therefore, $[B_1,B_1]\subset U_1$ and $B_1$ is meta-abelian. 
\end{enumerate}
\end{prop}
In this non linear context, we clearly have some analogies with the structure of the algebraic group $SL_2(\C)$ but we also have some differences:\\
- the inclusion $[B_1,B_1]\subset U_1$ is a strict one,  $[B_1,B_1]$ is an invariant subgroup of $U_1$ and we have $U_1/[B_1,B_1]=<dj_1(1,1+ u)>\simeq\Z$.\\
- let $B_1^\natural=<U_1,T_1>$. $B_1^\natural$ is also a meta-abelian group which contains $U_1$, but it contains only a maximal torus of rank 1. We call $B_1^\natural$ a Borel of rank one.

\end{document}